\documentclass[10pt]{article}
\usepackage[a4paper]{geometry}
\usepackage{graphicx,version,color,epstopdf,enumitem,latexsym,amsmath,amsfonts,amsthm,setspace,mathtools}
\usepackage[title]{appendix}
\usepackage{caption}
\graphicspath{{figs/}}
\newtheorem{remark}{Remark}[section]
\newtheorem{theorem}{Theorem}[section]
\newtheorem{lemma}{Lemma}[section]

\newtheorem{corollary}{Corollary}[section]
\newtheorem{definition}{Definition}[section]
\newtheorem{example}{Example}[section]

\numberwithin{equation}{section}
\numberwithin{figure}{section}

\allowdisplaybreaks

\usepackage{hyperref}
\hypersetup{colorlinks=true,citecolor=blue,linkcolor=red,anchorcolor=blue}

\title{Optimal Root-Exponential Convergence of Lightning Plus Polynomial Approximation in Corner Domains\thanks{This work was supported by the National Natural Science Foundation of China (No. 12271528).}}

\author{Shuhuang Xiang\thanks{School of Mathematics and Statistics, INP-LAMA,Central South University, Changsha 410083,  P. R. China, xiangsh@csu.edu.cn} \and Jun Xiang\thanks{School of Mathematics and Statistics, Guizhou University, Guiyang 550025, P. R. China, gs.jxiang24@gzu.edu.cn} \and Shunfeng Yang\thanks{College of Science, Southwest Forestry University, Kunming 650224,  P. R. China, yangshunfeng@163.com} \and Yuee Zhong\thanks{School of Mathematics and Statistics, Central South University, Changsha 410083,  P. R. China, qqlyy0508@163.com}}

\begin{document}

\maketitle

\begin{abstract}
This paper presents a rigorous convergence analysis for the lightning plus polynomial approximation scheme, which employs rational approximations constructed with  preassigned tapered, exponentially clustered poles. This pole placement strategy was originally introduced by Trefethen and his collaborators for the resolution of corner singularities. Ample numerical results indicate that this scheme achieves root-exponential convergence, and in particular attains the same optimal convergence rate as the best rational approximation to $x^\alpha$ on $[0,1]$ established by Stahl.

\bigskip
In this work, we establish optimal root-exponential convergence for the class of prototype functions of the form $g(z)z^\alpha$ or $g(z)z^\alpha\log z$, where $g$ is analytic on a neighborhood of a sector domain $S_\beta$. These results confirm the validity of Conjectures 3.1 and 5.3 stated in [SIAM J. Numer. Anal., 61:2580-2600, 2023], and demonstrate that the choice $\sigma_{\mathrm{opt}} =\frac{\sqrt{2(2 - \beta)}\pi}{\sqrt{\alpha}}$ achieves the theoretically optimal convergence rate $\mathcal{O}\left(e^{-\sqrt{2(2 - \beta)N\alpha}\pi}\right)$.  Notably, for the specific case of  $\beta = 0$, the scheme recovers Stahl's optimal convergence rate  for  $x^\alpha$.
 Furthermore, working within the decomposition framework for corner domains proposed by Gopal and Trefethen, this paper provides a rigorous proof of optimal root-exponential convergence for lightning plus polynomial approximation problems on corner domains, and explicitly derives the optimal pole clustering parameter.
\end{abstract}


\vspace{0.05in}


\section{Introduction}
\label{sec:Int}
In the study of partial differential equations (PDEs) formulated in corner domains with continuous but piecewise analytic boundary conditions exhibiting jumps in the first derivative at the corner points, the solutions may develop isolated branch points at these locations. Extensive research on corner singularities has been conducted by Lewy \cite{1950Lewy}, Lehman \cite{Lehman1954DevelopmentsIT,Lehman1957DevelopmentOT}, and Wasow \cite{Wasow}. Conventional numerical methods face significant challenges in obtaining accurate solutions for such problems \cite{Gopal20192,Gopal2019}. Recently, efficient and robust numerical schemes based on rational functions have been developed to address branch singularities \cite{Brubeck2022,Gopal20192,Gopal2019,Herremans2023,Nakatsukasa2021,Trefethen2024,Trefethen2025,Trefethen202502,TNWNM2021,Xue2024,ZX2024}. Extensive numerical experiments demonstrate that these schemes achieve root-exponential convergence when applied to Laplace, Helmholtz \cite{Gopal20192,Gopal2019,Treweb,Trefethen2024}, and biharmonic equations (including Stokes flow scenarios) \cite{Brubeck2022,Xue2024}.

Gopal and Trefethen \cite{Gopal20192,Gopal2019} initially proposed a powerful and robust lightning plus polynomial (LP) scheme to approximate functions of the form $z^\alpha g(z)$  with corner singularities at $z=0$ within a sector domain defined as
\begin{equation*}
S_\beta=\big\{z: \, z=xe^{\pm \frac{\theta\pi}{2}i} \mbox{\, with\, $x\in [0,1]$ and $\theta\in [0,\beta]$}\big\}\quad \mbox{(see Fig. \ref{Vsector} (left))}
\end{equation*}
where $\beta\in [0,2)$. This scheme has also been successfully applied to solve the Laplace and Helmholtz equations on corner domains. By employing the rational function
\begin{equation}\label{eq:rat}
r_N(z)=\frac{p(z)}{q(z)}=\sum_{j=1}^{N_1}\frac{a_j}{z-p_j}+\sum_{j=0}^{N_2} b_jz^j:=r_{N_1}(z)+P_{N_2}(z),\, N=N_1+N_2
\end{equation}
and introducing uniform exponentially clustered poles
\begin{equation}\label{eq:uniform0}
p_j =-C\exp\left(-\sigma \frac {N_1-j}{\sqrt{N_1}}\right),\quad 1\leq j\leq N_1
\end{equation}
 with $C$ a positive constant and $\sigma>0$, they demonstrated root-exponential convergence $\mathcal{O}(e^{-C_0\sqrt{N}})$ \cite[Theorem 2.2]{Gopal2019} for $\beta\in (0,1)$, where $C_0>0$ is an unspecified constant. Their proof relies on Walsh's Hermite integral formula \cite[Theorem 2 of Chapter 8]{Walsh1965} by considering a rational interpolant with the poles \eqref{eq:uniform0} and the interpolation nodes $\{0,-p_1,\ldots,-p_{N_1-1}\}$.

\begin{figure}[htbp]
\centerline{\includegraphics[height=3.28cm,width=4cm]{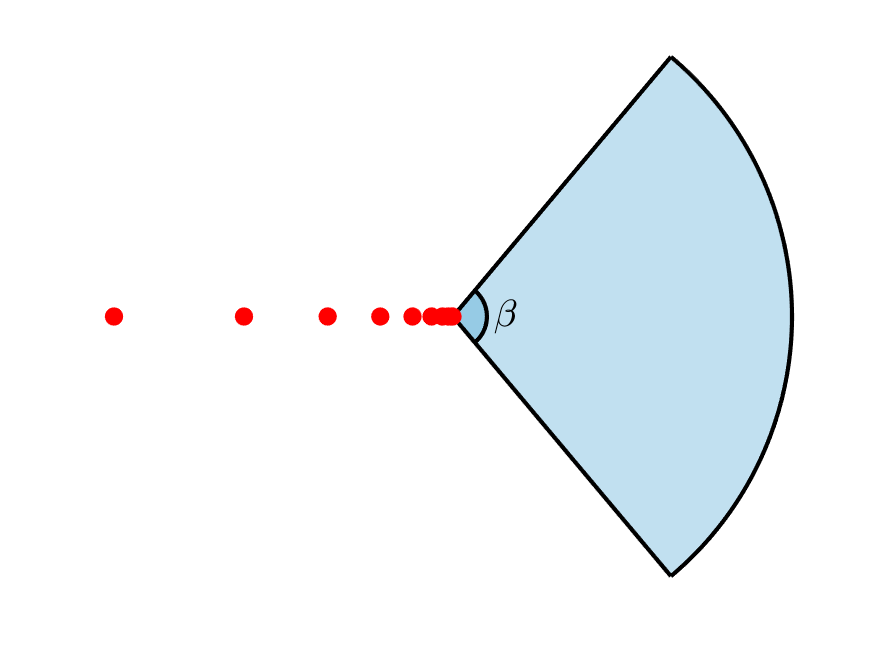}\hspace{2cm}\includegraphics[height=3.28cm,width=4cm]{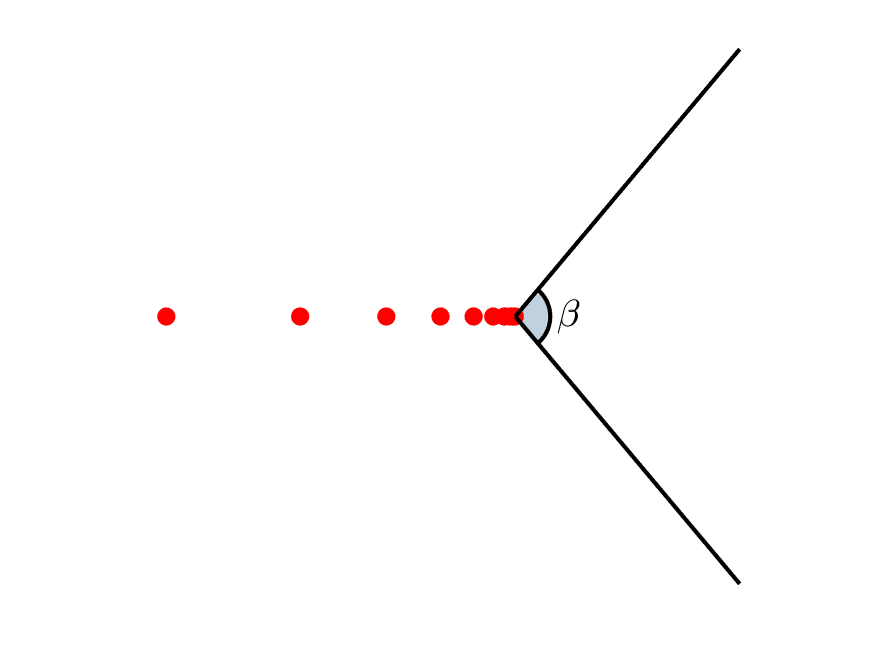}}
\caption{Sector domain (left):
$S_{\beta}=\left\{z: \, z=xe^{\pm \frac{\theta\pi}{2}i}\, x\in [0,1],\,\theta\in [0,\beta]\right\}$
 and
V-shaped domain (right): $V_\beta=\left\{z: \, z=xe^{\pm \frac{\beta\pi}{2}i},\, x\in [0,1]\right\}$ for fixed $\beta\in [0,2)$.
The red points illustrate the distributions of exponentially clustered poles.}
\label{Vsector}
\end{figure}

Another two specific pole-distribution models on the interval $[-C,0]$ were introduced  in Trefethen, Nakatsukasa, and Weideman  \cite{TNWNM2021}. One is a uniform exponential clustering of  $N_1$ poles
\begin{equation*}
\bar{p}_j =-C\exp\big(-\pi (j-1)/\sqrt{N_1}\big),\quad 1\leq j\leq N_1,
\end{equation*}
and the other is a tapered exponential clustering
\begin{equation*}
p_j =-C\exp\left(-\sqrt{2}\pi\big(\sqrt{N_1}-\sqrt{j}\big)/\sqrt{\alpha}\right),\quad 1\le j\le N_1.
\end{equation*}
When applied to the approximation of $x^\alpha$ ($0<\alpha<1$) on $[0,1]$, the LP approximations \eqref{eq:rat} with these two types of poles achieve exact approximation errors of order $\mathcal{O}(e^{-\pi\sqrt{\alpha N}})$
and $\mathcal{O}(e^{-\pi\sqrt{2\alpha N}})$, respectively.

It is well known from  Stahl \cite{Stahl2003} that the best rational approximant $R_{N}^*$ of degree $N$
 to $f(x)=x^\alpha$  has a convergence rate
\begin{align}\label{eq:newmann1}
  \lim_{N \to \infty} e^{2\pi \sqrt{\alpha N}} \max_{x\in [0,1]}|x^\alpha-R_N^*(x)| = 4^{1+\alpha} |\sin \pi \alpha|
\end{align}
for each $\alpha>0$. The existence of the limit in \eqref{eq:newmann1} and its equality to the right-hand side were originally conjectured by Varga and Carpenter \cite{Varga} based on high-precision numerical computations. The complete proof of \eqref{eq:newmann1} for all values of $\alpha>0$ not only resolved a long-standing research problem but also stands as one of the most significant achievements in approximation theory (see  Aptekarev, Nevai
and Totik \cite{Aptekarev}). However, finding  the best rational approximant $R_{N}^*$ is NP-hard in computational complexity \cite{Makila1999}.

To achieve the best convergence order \eqref{eq:newmann1}, Herremans, Huybrechs and Trefethen  \cite{Herremans2023}  proposed a novel LP approximation with a low-degree polynomial supported by the tapered exponentially clustering poles
\begin{equation}\label{eq:tapered2}
p_j =-C\exp(-\sigma(\sqrt{N_1}-\sqrt{j})),\quad 1\leq j\leq N_1,
\end{equation}
where $\sigma>0$. Specifically,
there exist coefficients $\{a_j\}_{j=1}^{N_1}$ and a polynomial $P_{N_2}$ with
$N_2 = \mathcal{O}(\sqrt{N_1})$,  such that the rational function $r_N(x)$ \eqref{eq:rat}, incorporating these tapered lightning poles with $\sigma = \frac{2\pi}{\sqrt{\alpha}}$, satisfies the following asymptotic behavior based on extensive numerical experiments
\begin{equation*}
|r_N(x)-x^\alpha|=\mathcal{O}(e^{-2\pi\sqrt{\alpha N}})
\end{equation*}
as $N \rightarrow \infty$. This result, initially conjectured in \cite{Herremans2023}, demonstrates a substantial improvement in both achievable accuracy and the optimal convergence rate, which is consistent with the best rational approximation \eqref{eq:newmann1} established by Stahl in \cite{Stahl2003}.

{\bf Conjecture 1.1} \cite[Conjecture 3.1]{Herremans2023}
 There exist coefficients $\{a_j\}_{j=1}^{N_1}$ and a polynomial $P_{N_2}$ with
$N_2 = \mathcal{O}(\sqrt{N_1})$, for which the LP approximation $r_N(x)$ \eqref{eq:rat} having
tapered lightning poles \eqref{eq:tapered2} with $\sigma = \frac{2\pi}{\sqrt{\alpha}}$
satisfies
\begin{equation*}
|r_N(x)-x^\alpha|=\mathcal{O}\big(e^{-2\pi\sqrt{\alpha N}}\big)
\end{equation*}
as $N \rightarrow \infty$, uniformly for $x\in [0,1]$.

\bigskip

Conjecture 1.1 has been settled in the special case $0 <\alpha< 1$ in \cite{XYW2024}. Moreover, the  choice $\sigma=\frac{2\pi}{\sqrt{\alpha}}$ achieves the fastest convergence rate among all $\sigma>0$. For more details, see Herremans et al. \cite{Herremans2023} and Xiang et al.  \cite{XYW2024}.

\bigskip
In the context of scientific computing on planar corner domains, solutions of PDEs may exhibit isolated branch singularities at the corner points \cite{Lehman1954DevelopmentsIT,Lehman1957DevelopmentOT,1950Lewy,Wasow}.
Herremans et al. \cite{Herremans2023} introduced an effective and reliable LP scheme based on tapered exponentially clustered poles of the form \eqref{eq:tapered2} for the V-shaped domain defined by
\begin{equation*}
V_\beta = \left\{ z : z = x e^{\pm \frac{\beta\pi}{2}i} \quad \text{with} \quad x \in [0,1] \right\} \quad \beta \in [0,2) \quad \mbox{(see Fig. \ref{Vsector} (right))}.
\end{equation*}
 The same work also puts forward the following conjecture for this domain, which has been empirically supported by extensive numerical experiments.

{\bf Conjecture 1.2} \cite[Conjecture 5.3]{Herremans2023}.
There exist coefficients $\{a_j\}_{j=1}^{N_1}$ and a polynomial $P_{N_2}$ with
$N_2 = \mathcal{O}(\sqrt{N_1})$, for which the LP approximation $r_N(z)$ \eqref{eq:rat} to $z^\alpha$ endowed with tapered lightning poles \eqref{eq:tapered2} parameterized by
\begin{align}\label{eq:optV}
\sigma=\frac{\pi\sqrt{2(2-\beta)}}{\sqrt{\alpha}}
\end{align}
satisfies
\begin{align}\label{eq:rateV}
|r_N(z)-z^\alpha|=\mathcal{O}(e^{-\pi\sqrt{2(2-\beta) N\alpha}})
\end{align}
uniformly for $z\in V_\beta=\{z=xe^{\pm \frac{\beta\pi}{2}i},\ x\in [0,1]\}$ for arbitrary fixed $\beta\in [0,2)$.


Conjecture 1.2 states that the LP approximation, based on the specific  $\sigma=\frac{\pi\sqrt{2(2-\beta)}}{\sqrt{\alpha}}$ for $z^\alpha$ on V-shaped domain $V_\beta$, exhibits a root-exponential convergence rate, which aligns with the best rational approximation in the sense of Stahl \cite{Stahl2003} in the special case $\beta=0$.

\bigskip

\bigskip
In the study of PDEs on domains with corners, the domain $\Omega$ may have either straight or curved sides. The interior angles at the vertices are $\varphi_1\pi, \cdots,\varphi_\ell\pi$ where each $\varphi_k\in (0,2)$. For domains with curved sides, these angles are defined by the tangent rays of the two sides $L_{k,1}$ and $L_{k,2}$
meeting at the common vertex $w_k$
(see Fig. \ref{tangent_covering_domain}).

\begin{figure}[htbp]
\centerline{\includegraphics[width=4.88cm]{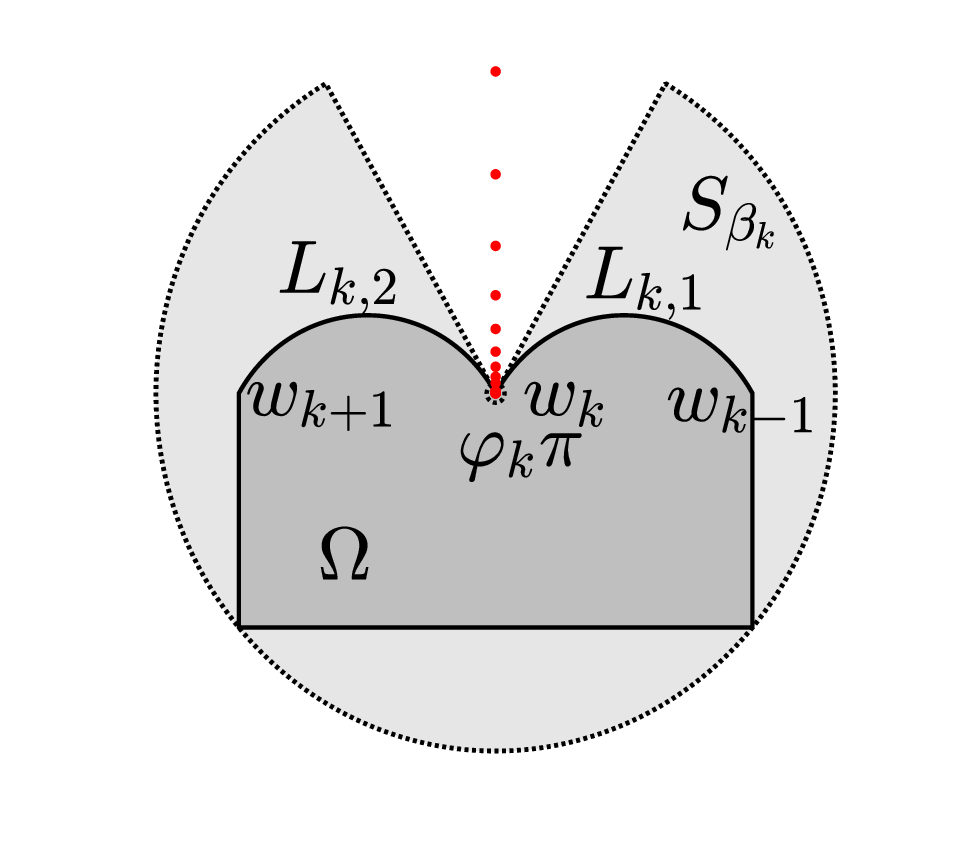}}
\caption{A curvy domain with an interior angle $\varphi_k\pi$, determined by the tangent rays extending from the common vertex. }
\label{tangent_covering_domain}
\end{figure}

In this paper, we provide a complete proof of Conjecture 1.2 for the sector domain $S_\beta$ by employing Poisson's summation formula \cite{Henrici,Stenger} in conjunction with Cauchy's integral theorem and Runge's approximation theorem. Furthermore, we demonstrate that selecting the poles in \eqref{eq:tapered2} with  $\sigma= \frac{\pi\sqrt{2(2-\beta)}}{\sqrt{\alpha}}$
 represents the optimal choice among all candidates in \eqref{eq:tapered2} for achieving the fastest convergence rate, and we therefore designate this value as $\sigma_{\rm opt}$.

 \begin{theorem}\label{mainthm}
There exist coefficients $\{a_j\}_{j=1}^{N_1}$ and a polynomial $P_{N_2}$ with
$N_2 = \mathcal{O}(\sqrt{N_1})$, for which the LP approximation $r_N(z)$ \eqref{eq:rat}  to $z^\alpha$ for $\alpha>0$ endowed with the
 lightning poles \eqref{eq:tapered2} parameterized by $\sigma>0$  satisfies as $N \rightarrow \infty$ that
\begin{equation}\label{eq: rate1}
|r_N(z)-z^\alpha|=\max\left\{\frac{|\sin(\alpha\pi)|}{\alpha\pi},\frac{|\sin(\alpha\pi)|}{(1-(\alpha))\pi}\right\}\cdot\left\{\begin{array}{ll}
\mathcal{O}(e^{-\sigma\alpha\sqrt{N}}),&\sigma\le \sigma_{\rm opt},\\
\mathcal{O}(e^{-\pi\eta\sqrt{2(2-\beta)N\alpha}}),&\sigma> \sigma_{\rm opt},
\end{array}\right.
\end{equation}
 uniformly for $z\in S_{\beta}$, where $\sigma_{\rm opt}=\frac{\pi\sqrt{2(2-\beta)}}{\sqrt{\alpha}}$, $\eta=\frac{\sigma_{\rm opt}}{\sigma}$ and $(\alpha)$ denotes the fractional part of $\alpha$, satisfying $0\le (\alpha)<1$.
\end{theorem}

From Theorem \ref{mainthm} together with
\begin{align}\label{eq:31}
 \frac{|\sin(\alpha\pi)|}{\alpha\pi}\le 1,\quad \frac{|\sin(\alpha\pi)|}{(1-(\alpha))\pi}=\frac{|\sin((1-(\alpha))\pi)|}{(1-(\alpha))\pi}\le 1
\end{align}
for all $\alpha>0$,
we observe that for a specific choice of the parameter $\sigma_{\rm opt}$ in the tapered exponentially clustered poles \eqref{eq:tapered2}, the lightning approximation coupled with a low-degree polynomial basis achieves the optimal convergence rate $\mathcal{O}(e^{-\pi\sqrt{2(2-\beta)N\alpha}})$
uniformly for $z\in S_\beta$ when $\alpha$ is not a positive integer. In particular:
\begin{itemize}
\item Conjecture 1.1 holds for the special case $\sigma_{\rm opt}=\frac{\pi\sqrt{2(2-\beta)}}{\sqrt{\alpha}}$ with $\beta=0$, where $r_N(x)$ attains the fastest convergence rate among all $\sigma>0$.

\item Conjecture 1.2 holds for the special case $\sigma_{\rm opt}$ and $z\in V_\beta$, where $r_N(z)$ also achieves the fastest rate  among all $\sigma>0$.
\end{itemize}

\bigskip
The numerical results presented in Fig. \ref{rateex} confirm that the convergence rate of the LP approximation given by \eqref{eq:rat}, as rigorously established in Theorem \ref{mainthm}, is achieved for all  values of $\sigma$. Meanwhile, Fig. \ref{poled} visualizes the radial distances of the exponentially clustered poles from the singularity at $z=0$.

It is noteworthy that when the clustering parameters in \eqref{eq:tapered2} satisfy $\sigma_1\cdot \sigma_2=\sigma_{\rm opt}^2$,
the corresponding LP approximation \eqref{eq:rat} achieves the same order of convergence even for markedly different pole distributions (see Fig. \ref{rateex}); this follows by writing $e^{-\sigma_1\alpha\sqrt{N}}=e^{-\pi\eta^{-1}\sqrt{2(2-\beta)N\alpha}}$ for $\eta\ge 1$ (i.e., $\sigma_1\le \sigma_{\rm opt}$).

\begin{figure}[htbp]
\centerline{\includegraphics[height=6cm,width=14cm]{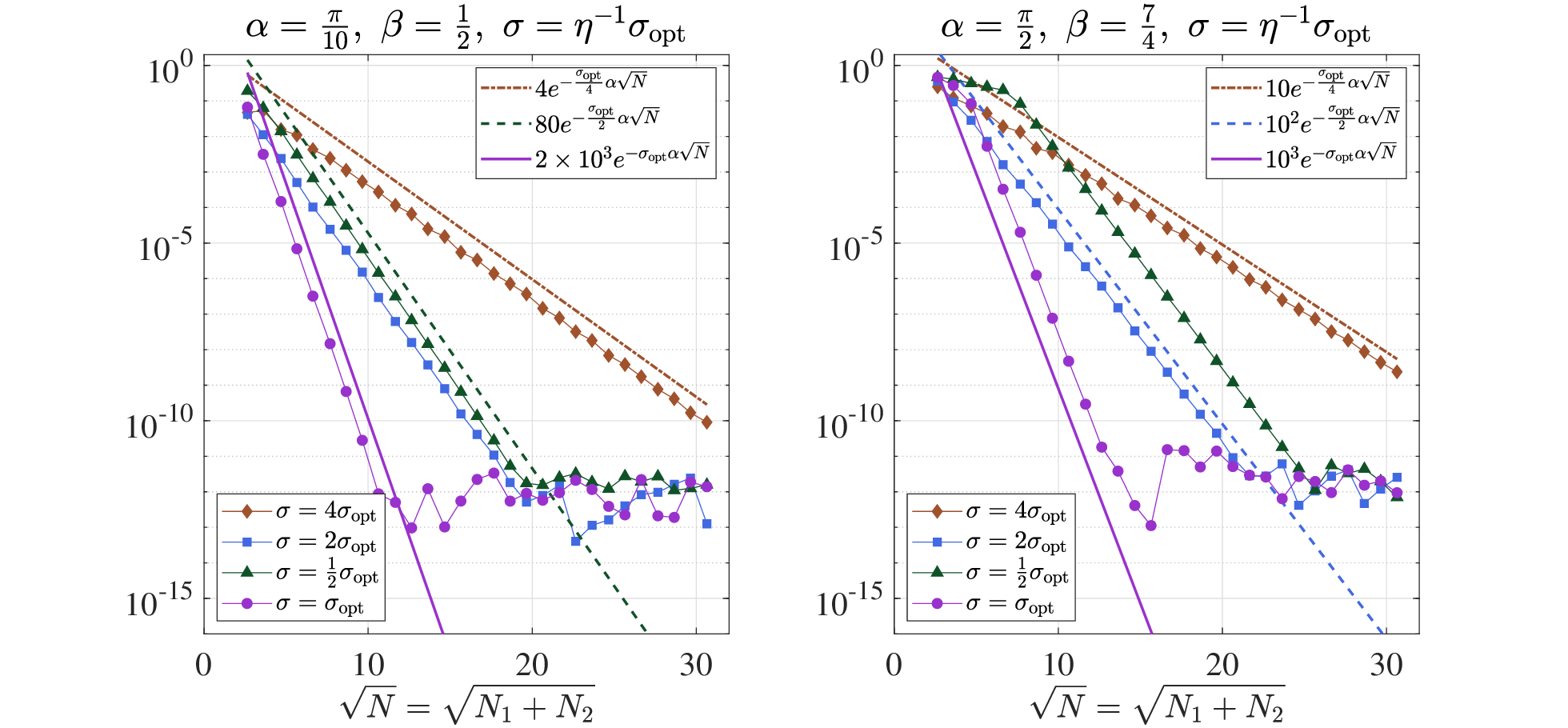}}
\caption{$\|z^\alpha-r_{N,\sigma}(z)\|_{C(S_\beta)}$  with various values of $\sigma$, fixed $\alpha=\frac{\pi}{10},\ \beta=\frac{1}{2}$ (left) and $\alpha=\frac{\pi}{2},\ \beta=\frac{7}{4}$ (right), respectively: $N_1=(2:1:30)^2$ and $N_2={\rm ceil}(1.3\sqrt{N_1})$.}
\label{rateex}
\end{figure}

\begin{figure}[htbp]
\centerline{\includegraphics[height=5.cm,width=9.8cm]{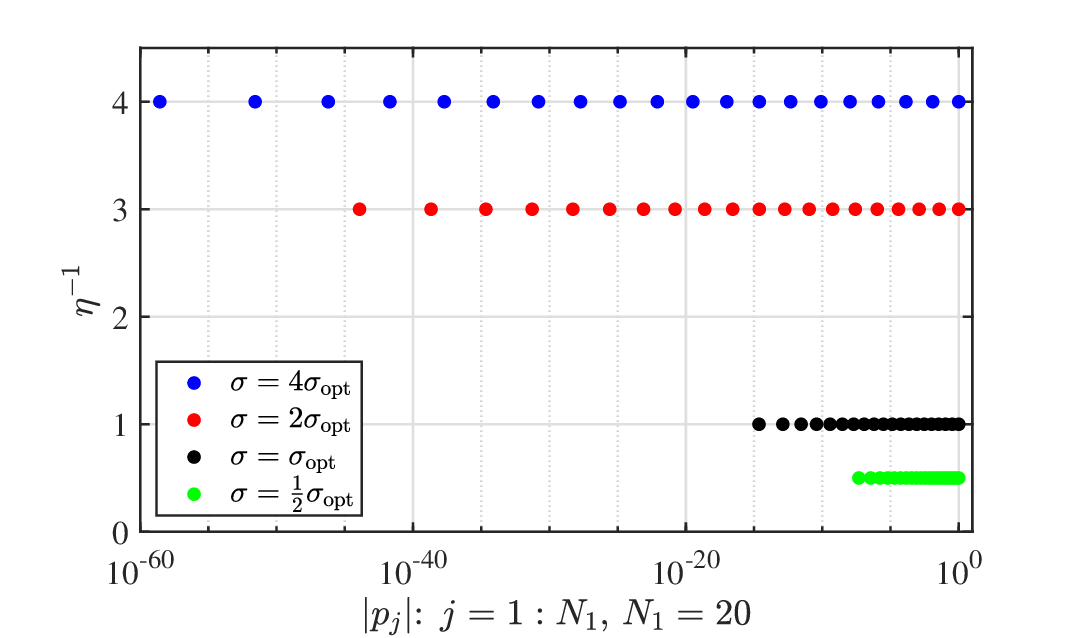}}
\caption{The distance of the poles \eqref{eq:tapered2} from the singularity $z=0$  with various values of $\sigma=\eta^{-1} \sigma_{\rm opt}$ with $\eta^{-1}=4,2,1,0.5$ and fixed $\alpha=\frac{\pi}{10}$ and $\beta=\frac{1}{2}$: $\sigma_{\rm opt}=\pi\sqrt{2(2-\beta)}/\sqrt{\alpha}$.}
\label{poled}
\end{figure}

Furthermore, an analogous result holds for  $z^\alpha\log z$ in $S_{\beta}$.

\begin{theorem}\label{mainthm2}
There exist coefficients $\{\widetilde{a}_j\}_{j=1}^{N_1}$ and a polynomial $\widetilde{P}_{N_2}$ with
$N_2 = \mathcal{O}(\sqrt{N_1})$, for which the LP approximation $\widetilde{r}_N(z)$ \eqref{eq:rat}  to $z^\alpha\log z$ with
tapered lightning poles \eqref{eq:tapered2} parameterized by $\sigma>0$  satisfies as $N \rightarrow \infty$  that

\begin{equation}\label{eq: rate2}
|\widetilde{r}_N(z)-z^\alpha\log{z}|=\left\{\begin{array}{ll}
\mathcal{O}(\sqrt{N}e^{-\sigma\alpha\sqrt{N}}),&\sigma\le \sigma_{\rm opt},\\
\mathcal{O}(e^{-\pi\eta\sqrt{2(2-\beta)N\alpha}}),&\sigma> \sigma_{\rm opt},
\end{array}\right.
\end{equation}
uniformly for $z\in S_{\beta}$. In particular, for $\alpha$ being a positive integer, it holds
\begin{equation}\label{eq: rate22}
|\widetilde{r}_N(z)-z^\alpha\log{z}|=\left\{\begin{array}{ll}
\mathcal{O}(e^{-\sigma\alpha\sqrt{N}}),&\sigma\le \sigma_{\rm opt},\\
\mathcal{O}(e^{-\pi\eta\sqrt{2(2-\beta)N\alpha}}),&\sigma> \sigma_{\rm opt}.
\end{array}\right.
\end{equation}
\end{theorem}

\begin{figure}[htbp]
\centerline{\includegraphics[height=6cm,width=14cm]{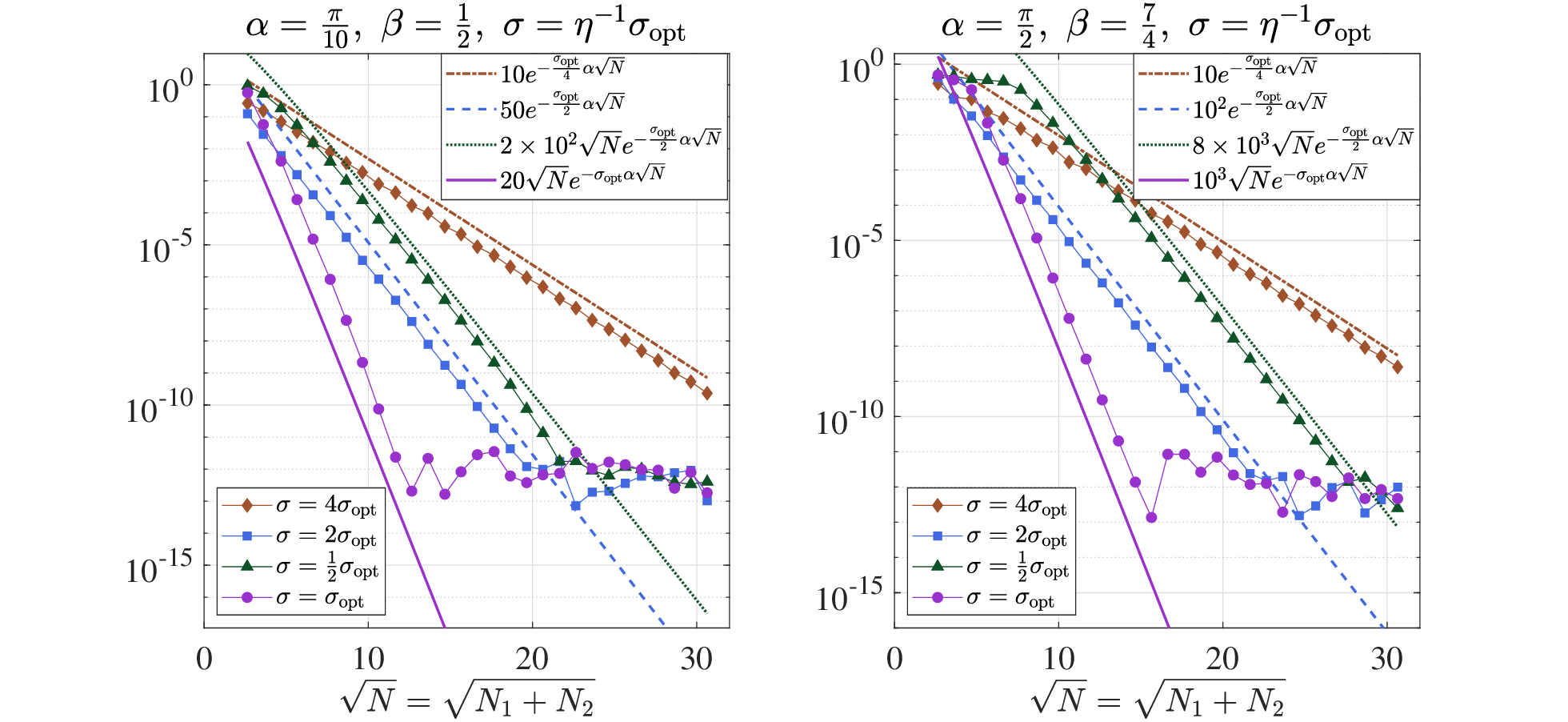}}
\caption{$\|z^\alpha\log z-r_{N,\sigma}(z)\|_{C(S_\beta)}$  with various values of $\sigma$, fixed $\alpha=\frac{\pi}{10},\ \beta=\frac{1}{2}$ (left) and $\alpha=\frac{\pi}{2},\ \beta=\frac{7}{4}$ (right), respectively: $N_1=(2:1:30)^2$ and $N_2={\rm ceil}(1.3\sqrt{N_1})$.}
\label{ratelogex}
\end{figure}

Fig. \ref{ratelogex} illustrates that the convergence rate of the
LP approximation in Theorem \ref{mainthm2} is also attainable for each $\sigma$. In addition, if the clustered parameters in \eqref{eq:tapered2} satisfy  $\sigma_1\cdot \sigma_2=\sigma_{\rm opt}^2$, then the corresponding LP approximation  attains a similar order of convergence, differing by a factor of $\sqrt{N}$ (see Fig. \ref{ratelogex2}).   Theorem \ref{mainthm} and Theorem \ref{mainthm2} can be readily extended to $g(z)z^{\alpha}$ and $g(z)z^{\alpha}\log{z}$ for $g(z)$  analytic on $S_\beta$ by applying Runge's approximation theorem \cite{Gaier1987}.

\begin{figure}[htbp]
\centerline{\includegraphics[height=6cm,width=7.5cm]{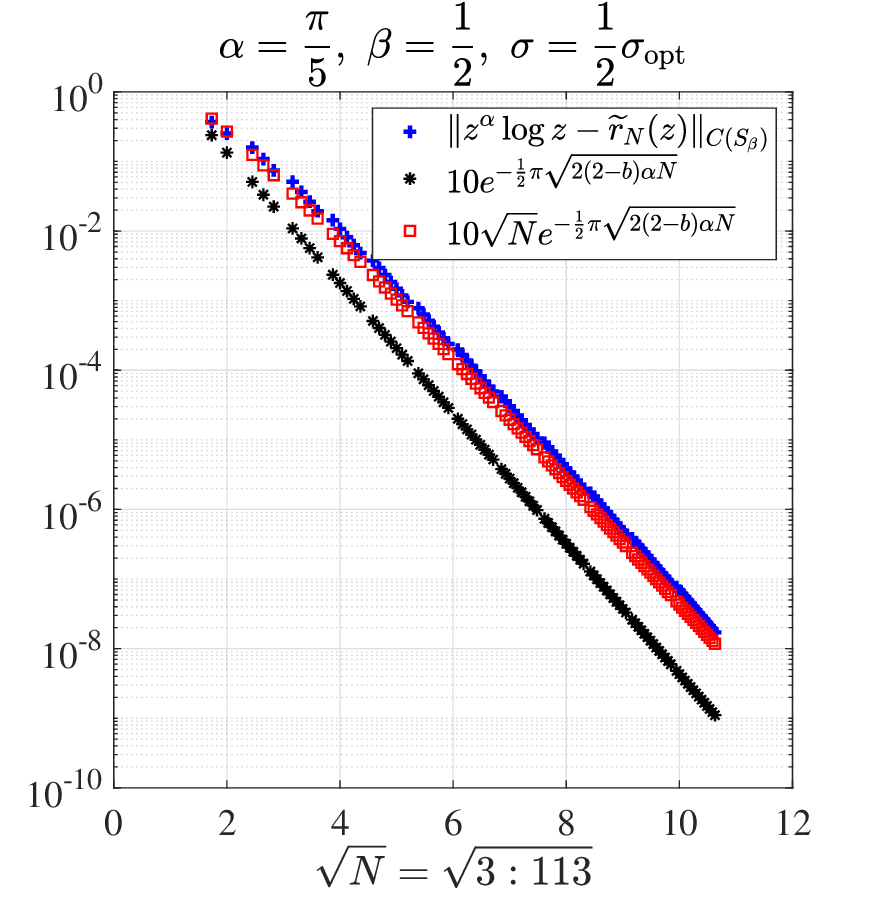}\includegraphics[height=6cm,width=7.5cm]{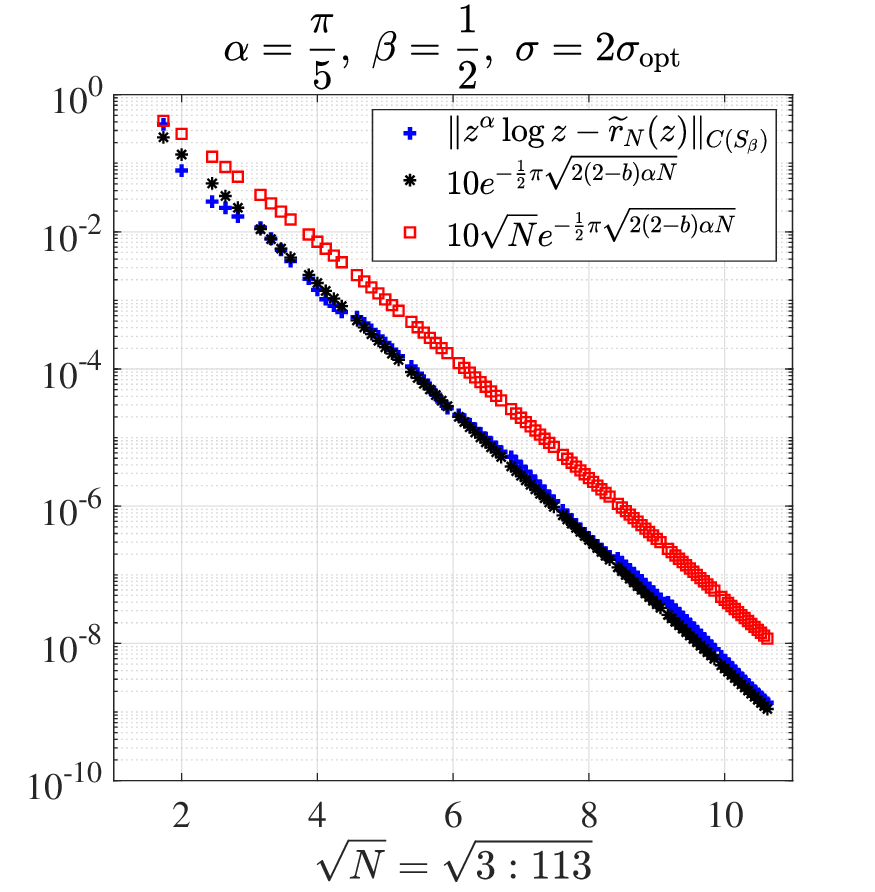}}
\caption{$\|z^\alpha\log z-r_{N,\sigma}(z)\|_{C(S_\beta)}$  with various values of $\sigma$, fixed $\alpha=\frac{\pi}{5}$, $\beta=\frac{1}{2}$ and $\sigma_{\rm opt}= \frac{\pi\sqrt{2(2-\beta)}}{\sqrt{\alpha}}$.}
\label{ratelogex2}
\end{figure}

\bigskip
To investigate the construction of multivariate rational approximations for functions with branch singularities, Boull\'{e}, Herremans and Huybrechs  \cite{boulle2023multivariate} introduced tapered exponentially clustered poles along the imaginary axis, specifically, the square roots of the poles $p_j$ defined in \eqref{eq:tapered2} as given by
\begin{align}\label{eq:1taperedimag}
\widetilde{p}_{\pm j} =&\pm i\sqrt{C}\exp\left(-\frac{\sigma}{2}\big(\sqrt{N_1}-\sqrt{j}\big)\right),\quad 1\le j\le N_1,\quad \sqrt{-1}=i.
\end{align}
 By employing the rational schemes as presented in Theorem \ref{mainthm} and Theorem \ref{mainthm2}, we can directly apply the identities
$$
|x|^2=x^2, \quad |x|^{2\alpha}=x^{2\alpha},\quad 2\pi\eta\sqrt{N\alpha}=\frac{4\pi^2}{\sigma}\sqrt{N}
$$
to establish the following corollary for all $\alpha>0$.

\begin{corollary}\label{mainthm3}
Let $r_N$ be defined in Theorem \ref{mainthm}, then $r_N(x^2)$
with
tapered lightning poles \eqref{eq:1taperedimag}
satisfies for  $\alpha>0$ that
\begin{equation}\label{eq: brateoimag}
\Big| r_N(x^2)-|x|^{2\alpha}\Big|
=\max\left\{\frac{|\sin(\alpha\pi)|}{\alpha\pi},\frac{|\sin(\alpha\pi)|}{(1-(\alpha))\pi}\right\}\cdot
\left\{\begin{array}{ll}
\mathcal{O}(e^{-\sigma\alpha\sqrt{N}}),&\sigma\le \frac{2\pi}{\sqrt{\alpha}},\\
\frac{\mathcal{O}(1)}{e^{\frac{4\pi^2}{\sigma}\sqrt{N}}-1},&\sigma> \frac{2\pi}{\sqrt{\alpha}},
\end{array}\right.
\end{equation}
while $\widetilde{r}_N$ is defined in Theorem  \ref{mainthm2} and $\widetilde{r}_N(x^2)$
with
tapered lightning poles \eqref{eq:1taperedimag}
satisfies that
\begin{equation}\label{eq: brateo_uniform_imag}
\Big|\widetilde{r}_N(x^2)-2|x|^{2\alpha}\log |x|\Big|
=\left\{\begin{array}{ll}
\mathcal{O}(\sqrt{N}e^{-\sigma\alpha\sqrt{N}}),&\sigma\le \frac{2\pi}{\sqrt{\alpha}},\\
\mathcal{O}(e^{-\frac{4\pi^2}{\sigma}\sqrt{N}}),&\sigma> \frac{2\pi}{\sqrt{\alpha}},
\end{array}\right.
\end{equation}
as $N =N_1+N_2\rightarrow \infty$ uniformly for  $x\in [-1,1]$.
\end{corollary}

\bigskip
Furthermore, combining Theorems \ref{mainthm} and \ref{mainthm2} with the decomposition framework of Gopal and Trefethen \cite{Gopal2019}, we rigorously prove root-exponential convergence for LP approximations on corner domains and explicitly derive the optimal pole clustering parameter.

\bigskip
The remainder of the paper is organized as follows. In Section \ref{sec:2}, we construct  the LP schemes for $z^{\alpha}$ and $z^{\alpha}\log{z}$ within the sector domain
$S_{\beta}$. The corresponding asymptotic relation, $N_2=\mathcal{O}(\sqrt{N_1})$, is established via Runge's approximation theorem, adopting the approach outlined by Levin and Saff \cite{LevinSaff} and Herremans et al. \cite{Herremans2023}.
Section \ref{sec:3} applies Cauchy's integral theorem and Poisson's summation formula to composite rectangle rules for integrals over the real line, and provides a rigorous analysis of the convergence rates of the numerical quadratures for $z^{\alpha}$ and $z^{\alpha}\log{z}$. These rates are essential for establishing the root-exponential convergence of the LP approximations. In Section \ref{sec:4}, we prove Theorem \ref{mainthm} and Theorem \ref{mainthm2} and their generalizations, including the optimality of the parameter $\sigma_{\rm opt}=\frac{\pi\sqrt{2(2-\beta)}}{\sqrt{\alpha}}$.
Section \ref{sec:5} extends the analysis to problems with corner singularities, demonstrating the root-exponential convergence of LP approximations and the optimal choice of the clustered parameters. Finally, Section \ref{conclusion} summarizes the main conclusions of the paper.


\section{Preliminaries}\label{sec:2}
Building on Stieltjes's truncated continued fractions (i.e., rational functions), Stenger \cite{Stenger1981,Stenger1986,Stenger1993} constructed explicit rational approximants for functions on the intervals
$[-1,1]$, $[0,+\infty)$, or $(-\infty,+\infty)$ by utilizing integral representations with preassigned poles. Following this framework, we first present the integral representations of $z^\alpha$ and $z^\alpha \log z$, and subsequently formulate the corresponding LP schemes.

\subsection{Rational Approximations to  $z^{\alpha}$ and $z^{\alpha}\log z$}\label{subsec21}
A robust methodology for deriving rational function approximations to singular functions $z^\alpha$ and $z^\alpha \log z$ lies in the composite rectangular rule approximation of their integral representations \cite{XYW2025}.

Recently, by extending the Paley-Wiener theorem to a horizontal strip and leveraging asymptotic results for Fourier transforms, the root-exponential convergence of the LP approximation \eqref{eq:rat} with $N_2=\mathcal{O}(N_1)$
was established in Xiang, Wu and Yang \cite{XYW2025} for approximating functions of the form $z^{\alpha}g(z)$
or $z^\alpha g(z)\log z$ in $S_\beta$ via the uniform exponentially clustered poles \eqref{eq:uniform0}. An exact order of convergence was given for each clustered parameter $\sigma>0$ in \eqref{eq:uniform0}.
Furthermore, the fastest achievable convergence rates have been identified as  $\mathcal{O}(e^{-\pi\sqrt{(2-\beta)N\alpha}})$ for $z^{\alpha}g(z)$ and $\mathcal{O}(\sqrt{N}e^{-\pi\sqrt{(2-\beta)N\alpha}})$ for $z^{\alpha}g(z)\log z$, based on the optimal choice of the clustered parameter $\sigma=\frac{\pi\sqrt{2-\beta}}{\sqrt{\alpha}}$ for $0\le \beta<2$
  \cite{XYW2025}.
  In their framework, the integrand associated with the uniformly exponentially clustered poles \eqref{eq:uniform0} is analytic in a horizontal strip and belongs to $L^2(\mathbb{R})$. By contrast, we will see later that the integrand corresponding to the tapered exponentially clustered poles \eqref{eq:tapered2} exhibits fundamentally different behavior: it is well-defined on
$\mathbb{R_+}$ and has a branch cut at the origin $z=0$, yet it fails to lie in $L^2(\mathbb{R_+})$.


\begin{lemma}\label{complex_int_repre}\cite{XYW2025}
Let $\alpha>0$ and $m\ge\lfloor\alpha\rfloor$ where $\lfloor\alpha\rfloor$ denotes the largest integer not larger than $\alpha$.
Suppose that $s_1,\ldots,s_m$ are $m$ distinct numbers located outside $(-\infty,0]$.
Then it holds for all $z\in\mathbb{C}\setminus(-\infty,0)$ that
\begin{align}
z^{\alpha}=&\frac{\sin{(\alpha\pi)}}{(-1)^{m}\pi}
\int_0^{+\infty}\frac{zy^{\alpha-1}}{y+z}\left(\prod\limits_{k=1}^{m}\frac{z-s_k}{y+s_k}\right)\mathrm{d}y
+z\mathcal{L}[z^{\alpha-1};s_1,\ldots,s_m],\label{eq:cint_gener}\\
z^{\alpha}\log{z} =& \frac{\sin(\alpha\pi)}{(-1)^m\pi}\int_{0}^{+\infty}
\frac{zy^{\alpha-1}\log{y}}{y+z}\left(\prod_{k=1}^{m}\frac{z-s_k}{y+s_k}\right)\mathrm{d}y\notag\\
&+\frac{\cos(\alpha\pi)}{(-1)^{m}}\int_0^{+\infty}\frac{zy^{\alpha-1}}{y+z}
\left(\prod\limits_{k=1}^{m}\frac{z-s_k}{y+s_k}\right)\mathrm{d}y
+z\mathcal{L}[z^{\alpha-1}\log{z};s_1,\ldots,s_m],\label{eq:cint_log_gener}
\end{align}
where $\mathcal{L}[X(z);s_1,\ldots,s_m]$ denotes the Lagrange interpolating polynomial at nodes $s_1,\ldots,s_m$ for $X(z)=z^{\alpha-1}$ and
$z^{\alpha-1}\log{z}$, respectively.
Especially, it holds for $0<\alpha<1$ that
\begin{align*}
z^{\alpha}=&\frac{\sin{(\alpha\pi)}}{\pi}
\int_0^{+\infty}\frac{zy^{\alpha-1}}{y+z}\mathrm{d}y,\\
z^{\alpha}\log{z}=&\frac{\sin{(\alpha\pi)}}{\pi}
\int_{0}^{+\infty}\frac{zy^{\alpha-1}\log{y}}{y+z}\mathrm{d}y
+\cos{(\alpha\pi)}\int_{0}^{+\infty}\frac{zy^{\alpha-1}\mathrm{d}y}{y+z}.
\end{align*}
\end{lemma}

To achieve sharper bounds on LP approximations in Theorem \ref{mainthm} and Theorem \ref{mainthm2}, we define $s_k$ as the roots of the shifted Chebyshev polynomial of the first kind $\mathrm{T}_{m}\left(\frac{1}{\omega}(s-\delta)-1\right)$, given explicitly by
$s_k=\delta+\omega\left(1+\cos\frac{(2m -2k+1)\pi}{2m}\right)\in (\delta,\delta+2\omega)$ for $k=1,2,\ldots,m$.
To balance the truncation error of the integrals, the approximation error derived from Runge's theorem in Subsection \ref{subsec22}, and the quadrature error in Section \ref{sec:3}, we set $\delta=1$, and $\omega=\frac{1}{2}$.
\begin{itemize}
\item For $\beta\not=0$, let $\aleph=\delta+\omega+\frac{3}{2}=3$. We then obtain
\begin{align}\label{eq:211}
\left\{\begin{array}{l}
\mathbb{T}_{m,\beta}=\max\limits_{z\in S_{\beta}}\prod\limits_{k=1}^{m}|z-s_k|\le \prod\limits_{k=1}^{m}\left(1+\delta+\omega+\omega\cos\frac{(2m -2k+1)\pi}{2m}\right)
\le (1+\delta+\omega)^{m},\\
\frac{\mathbb{T}_{m,\beta}}{\delta^m}=\mathbb{T}_{m,\beta}<\left(\delta+\omega+\frac{3}{2}\right)^{m}=\aleph^m.\end{array}\right.
\end{align}

\item For $\beta=0$, let $\aleph=\delta+\omega+\frac{1}{2}=2$. We then obtain
\begin{align}\label{eq:212}\left\{\begin{array}{l}
\mathbb{T}_{m,\beta}=\max\limits_{z\in [0,1]}\prod\limits_{k=1}^{m}|z-s_k|\le \prod\limits_{k=1}^{m}\left(\delta+\omega+\omega\cos\frac{(2m -2k+1)\pi}{2m}\right)
\le\left(\delta+\omega\right)^{m},\\
\frac{\mathbb{T}_{m,\beta}}{\delta^m}=\mathbb{T}_{m,\beta}<\left(\delta+\omega+\frac{1}{2}\right)^{m}=\aleph^m.\end{array}\right.
\end{align}
\end{itemize}

The derivation of the exponentially clustered poles  \eqref{eq:tapered2} from \eqref{eq:cint_gener} and \eqref{eq:cint_log_gener} relies crucially on the exponential transformation  $y=Ce^{\frac{1}{\alpha}t}$. By applying this transformation, \eqref{eq:cint_gener} and \eqref{eq:cint_log_gener} lead to the following representations for $z\in S_{\beta}$
\begin{align}
z^\alpha=&\frac{\sin(\alpha\pi)}{(-1)^{m}\alpha\pi}
\int_{-\infty}^{+\infty}
\frac{zC^{\alpha}e^t}{Ce^{\frac{1}{\alpha}t}+z}
\left(\prod\limits_{k=1}^{m}\frac{z-s_k}{Ce^{\frac{1}{\alpha}t}+s_k}\right)\mathrm{d}t
+z\mathcal{L}[z^{\alpha-1};s_1,\ldots,s_m]\label{eq:zalpha}
\end{align}
where $m=\lfloor\alpha\rfloor=\alpha-(\alpha)$, and
\begin{align}
z^{\alpha}\log{z} =& \frac{\sin(\alpha\pi)}{(-1)^m\alpha^2\pi}\int_{-\infty}^{+\infty}
\frac{zC^{\alpha}te^{t}}{Ce^{\frac{1}{\alpha}t}+z}
\left(\prod_{k=1}^{m}\frac{z-s_k}{Ce^{\frac{1}{\alpha}t}+s_k}\right)\mathrm{d}t\notag\\
&+\left[\frac{\sin(\alpha\pi)\log{C}}{(-1)^m\alpha\pi}+\frac{\cos(\alpha\pi)}{(-1)^{m}\alpha}\right]
\int_{-\infty}^{+\infty}\frac{zC^{\alpha}e^{t}}{Ce^{\frac{1}{\alpha}t}+z}
\left(\prod_{k=1}^{m}\frac{z-s_k}{Ce^{\frac{1}{\alpha}t}+s_k}\right)\mathrm{d}t\label{eq:cint_log_gener_substituting}\\
&+z\mathcal{L}[z^{\alpha-1}\log{z};s_1,\ldots,s_m]\notag
\end{align}
where  $m=\lceil\alpha\rceil$ represents the smallest integer greater than or equal to $\alpha$. In contrast to the behavior of $z^\alpha$, when $\alpha$ is not a positive integer, Theorem \ref{mainthm2} for $z^\alpha\log z$ holds uniformly for all $z\in S_\beta$  as $\alpha$ tends to a positive integer, as evidenced by the error bounds \eqref{boundforwidetildeE} and the definition of $\kappa$ given in \eqref{eq:kappa}.

Thereafter, from \eqref{eq:zalpha} and \eqref{eq:cint_log_gener_substituting} the rational approximations for $z^{\alpha}$ and $z^{\alpha}\log{z}$ can be derived from the discretization of truncations in a finite interval of infinite improper integrals

\begin{align}\label{eq:truncerrors}
&\int_{-\infty}^{+\infty}
\frac{zC^{\alpha}t^l e^{t}}{Ce^{\frac{1}{\alpha}t}+z}
\left(\prod_{k=1}^{m}\frac{z-s_k}{Ce^{\frac{1}{\alpha}t}+s_k}\right)\mathrm{d}t\notag\\
=&\left\{\int_{-\infty}^{-T}+\int_{-T}^{\kappa T}+\int_{\kappa T}^{+\infty}\right\}
\frac{zC^{\alpha}t^l e^t}{Ce^{\frac{1}{\alpha}t}+z}
\left(\prod\limits_{k=1}^{m}\frac{z-s_k}{Ce^{\frac{1}{\alpha}t}+s_k}\right)\mathrm{d}t, \quad T>0,\quad l=0,1.
\end{align}

The first and third integrals in \eqref{eq:truncerrors} can be estimated using the techniques outlined in \cite[(3.3)-(3.10)]{XYW2025}. For the sake of clarity, we present the detailed derivations here, which will be referenced in subsequent analysis.

For $z=|z|e^{\pm i\frac{\theta\pi}{2}}=xe^{\pm i\frac{\theta\pi}{2}}\in S_\beta$ with $0\le x\le 1$ and $t\in \mathbb{R}$, we have
\begin{align}\label{eq:est}
\big{|}Ce^{\frac{1}{\alpha}t}+z\big{|}
=&\left\{\begin{array}{ll}
\sqrt{C^2e^{\frac{2}{\alpha}t}+2Cxe^{\frac{1}{\alpha}t}\cos\frac{\theta\pi}{2} +x^2},&0\le \theta\le 1\\
\sqrt{\left(Ce^{\frac{1}{\alpha}t}+x\cos\frac{\theta\pi}{2}\right)^2+x^2\sin^2\frac{\theta\pi}{2}},&1< \theta\le \beta< 2\\
\sqrt{\left(Ce^{\frac{1}{\alpha}t}\cos\frac{\theta\pi}{2}+x\right)^2+C^2e^{\frac{2}{\alpha}t}\sin^2\frac{\theta\pi}{2}},&1< \theta\le \beta< 2
\end{array}\right.\notag\\
\ge& {\cal X}(\theta)\max\left\{
x,
Ce^{\frac{1}{\alpha}t}\right\}\ge {\cal X}(\beta)\max\left\{
x,
Ce^{\frac{1}{\alpha}t}\right\}
\end{align}
where $${\cal X}(\theta)=\left\{
\begin{array}{ll}
1,&0\le \theta\le 1,\\
\sin{\frac{\theta\pi}{2}}
\ge\sin\frac{\beta\pi}{2},&1< \theta\le \beta<2.
\end{array}\right.$$
From \eqref{eq:est} we obtain, for all $t\le 0$,
\begin{align}\label{ine:minus_u}
{\displaystyle\left|\frac{zC^{\alpha}t^l e^t}{Ce^{\frac{1}{\alpha}t}+z}
\left(\prod\limits_{k=1}^{m}\frac{z-s_k}{Ce^{\frac{1}{\alpha}t}+s_k}\right)\right|
\le \frac{x|t|^lC^{\alpha}e^t} {x{\cal X}(\beta)}\cdot\frac{\mathbb{T}_{m,\beta}}{\delta^{m}}
=\frac{|t|^l\mathbb{T}_{m,\beta} C^{\alpha}e^t}{\delta^{m}{\cal X}(\beta)}}
\end{align}
and consequently
\begin{align}\label{eq:inequ_neg}
\left|\int_{-\infty}^{-T}
\frac{zC^{\alpha}t^l e^t}{Ce^{\frac{1}{\alpha}t}+z}
\left(\prod\limits_{k=1}^{m}\frac{z-s_k}{Ce^{\frac{1}{\alpha}t}+s_k}\right)\mathrm{d}t\right|
\le\frac{\mathbb{T}_{m,\beta}C^{\alpha}(1+T)^le^{-T}}{\delta^{m}{\cal X}(\beta)}.
\end{align}
While for $t>0$, setting
\begin{align}\label{eq:kappa}
\kappa=\frac{\alpha}{m+1-\alpha}=\left\{\begin{array}{ll}\frac{\alpha}{1-(\alpha)},&l=0,\\ \frac{\alpha}{2-(\alpha)},&l=1,\,\mbox{$\alpha$ is not a positive integer},\\ \alpha,&l=1,\,\mbox{$\alpha$ is a positive integer},\end{array}\right.
\end{align}
we obtain
\begin{align}\label{ieq:positive_u}
\left|\frac{zt^lC^{\alpha}e^t}{Ce^{\frac{1}{\alpha}t}+z}
\left(\prod\limits_{k=1}^{m}\frac{z-s_k}{Ce^{\frac{1}{\alpha}t}+s_k}\right)\right|
\le \frac{xt^lC^{\alpha}e^t}{Ce^{\frac{1}{\alpha}t}{\cal X}(\beta)}\cdot\frac{\mathbb{T}_{m,\beta}}{C^{m}e^{\frac{m}{\alpha}t}}
\le\frac{\mathbb{T}_{m,\beta}t^le^{-\frac{1}{\kappa}t}}{C^{m+1-\alpha}{\cal X}(\beta)}
\end{align}
and therefore
\begin{align}\label{eq:inequ_pos}
\int_{\kappa T}^{+\infty}
\left|\frac{zC^{\alpha}t^l e^t}{Ce^{\frac{1}{\alpha}t}+z}
\left(\prod\limits_{k=1}^{m}\frac{z-s_k}{Ce^{\frac{1}{\alpha}t}+s_k}\right)\right|\mathrm{d}t
\le\frac{\mathbb{T}_{m,\beta}\kappa(\kappa+\kappa T)^le^{-T}}{C^{m+1-\alpha}{\cal X}(\beta)}.
 \end{align}
Combining these results yields the following bound for the truncated error $\widehat{E}^{(l)}_T(z)$
\begin{align}\label{eq:boundforwidehatE}
\left|\widehat{E}^{(l)}_T(z)\right|=&\left|\left\{\int_{-\infty}^{-T}+\int_{\kappa T}^{+\infty}\right\}
\frac{zC^{\alpha}t^l e^t}{Ce^{\frac{1}{\alpha}t}+z}
\left(\prod\limits_{k=1}^{m}\frac{z-s_k}{Ce^{\frac{1}{\alpha}t}+s_k}\right)\mathrm{d}t\right|\notag\\
\le&
\frac{(1+T)^l\mathbb{T}_{m,\beta}C^{\alpha}e^{-T}}{{\cal X}(\beta)}
\left(\frac{1}{\delta^{m}}+\frac{\kappa^{l+1}}{C^{m+1}}\right).
\end{align}

Building upon the integral over the finite interval $[-T,\kappa T]$ in \eqref{eq:truncerrors}, we apply the transformation $t+T=\sqrt{u}$.
This transformation, originally introduced for $\sqrt{x}$ in \cite{Herremans2023},
facilitates the derivation of the formula presented in \eqref{eq:tapered2}. Consequently, the original integral in \eqref{eq:truncerrors} can be expressed as
\begin{align}\label{eq:rootrat}
&\int_{-T}^{\kappa T}
\frac{zC^{\alpha}t^l e^{t}}{Ce^{\frac{1}{\alpha}t}+z}
\left(\prod_{k=1}^{m}\frac{z-s_k}{Ce^{\frac{1}{\alpha}t}+s_k}\right)\mathrm{d}t\notag\\
=&\int_0^{(\kappa+1)^2 T^2}
 \frac{1}{2 \sqrt{u}}\frac{z C^{\alpha}\left(\sqrt{u}-T\right)^le^{\sqrt{u}-T}}{Ce^{\frac{1}{\alpha}(\sqrt{u}-T)}+z}
 \left(\prod\limits_{k=1}^{m}\frac{z-s_k}{Ce^{\frac{1}{\alpha}(\sqrt{u}-T)}+s_k}\right)\,\mathrm{d}u.
\end{align}

Let
 \begin{align*}
T&=\sqrt{N_1h}=\sigma\alpha\sqrt{N_1},\quad \mathcal{N}_th =(\kappa+1)^2 T^2,\quad N_t=\lceil\mathcal{N}_t\rceil=\lceil(\kappa+1)^2N_1\rceil (\le \mathcal{N}_t+1).
\end{align*}
We discretize \eqref{eq:rootrat} over the finite interval $[0,N_th]$ using the rectangular rule with step length  $h=\sigma^2\alpha^2$. The quadrature points are set at $u = jh$
for $j=1,2,\ldots,N_t$, which yields the following rational approximation
\begin{align}\label{eq:intC1}
&\int_{-\infty}^{+\infty}
\frac{zC^{\alpha}t^l e^{t}}{Ce^{\frac{1}{\alpha}t}+z}
\left(\prod_{k=1}^{m}\frac{z-s_k}{Ce^{\frac{1}{\alpha}t}+s_k}\right)\,\mathrm{d}t\notag\\
=&\int_0^{(\kappa+1)^2 T^2}
 \frac{z C^{\alpha}(\sqrt{u}-T)^le^{\sqrt{u}-T}}{2 \sqrt{u}\left(Ce^{\frac{1}{\alpha}(\sqrt{u}-T)}+z\right)}
 \left(\prod\limits_{k=1}^{m}\frac{z-s_k}{Ce^{\frac{1}{\alpha}(\sqrt{u}-T)}+s_k}\right)\,\mathrm{d}u +\widehat{E}^{(l)}_T(z)\\
=&\int_0^{N_th}
 \frac{1}{2 \sqrt{u}}\frac{z C^{\alpha}(\sqrt{u}-T)^le^{\sqrt{u}-T}}{Ce^{\frac{1}{\alpha}(\sqrt{u}-T)}+z}
 \left(\prod\limits_{k=1}^{m}\frac{z-s_k}{Ce^{\frac{1}{\alpha}(\sqrt{u}-T)}+s_k}\right)\,\mathrm{d}u
 +E^{(l)}_T(z)\notag\\
:=& r^{(l)}_{N_t}(z) +E^{(l)}_Q(z)+E^{(l)}_T(z),\notag
\end{align}
where
 $r^{(l)}_{N_t}(z)$ is the rational approximation defined by
\begin{align}\label{eq:ECrat1CC}
 r^{(l)}_{N_t}(z)=&h\sum_{j=1}^{N_t}
\frac{zC^{\alpha}(\sqrt{jh}-T)^le^{\sqrt{jh}-T}}{2 \sqrt{jh}\left(Ce^{\frac{1}{\alpha}(\sqrt{jh}-T)}+z\right)}
\left(\prod\limits_{k=1}^{m}\frac{z-s_k}{Ce^{\frac{1}{\alpha}(\sqrt{jh}-T)}+s_k}\right)
\end{align}
derived from the composite rectangular rule, $E_Q^{(l)}(z)$ is the quadrature error
\begin{align*}
E^{(l)}_Q(z)=\int_0^{N_th}
 \frac{1}{2 \sqrt{u}}\frac{z C^{\alpha}(\sqrt{u}-T)^le^{\sqrt{u}-T}}{Ce^{\frac{1}{\alpha}(\sqrt{u}-T)}+z}
 \left(\prod\limits_{k=1}^{m}\frac{z-s_k}{Ce^{\frac{1}{\alpha}(\sqrt{u}-T)}+s_k}\right)\,\mathrm{d}u
- r^{(l)}_{N_t}(z),
\end{align*}
and $E^{(l)}_T(z)$ is defined below and can be estimated using the bounds from \eqref{eq:inequ_pos} and \eqref{eq:boundforwidehatE} as follows
\begin{align}\label{boundforEl(z)}
|E^{(l)}_T(z)|
=&\left|\widehat{E}^{(l)}_T(z)-\int_{(\kappa+1)^2 T^2}^{N_th}
\frac{zC^{\alpha}(\sqrt{u}-T)^le^{\sqrt{u}-T}}{2 \sqrt{u}\left(Ce^{\frac{1}{\alpha}(\sqrt{u}-T)}+z\right)}
\left(\prod\limits_{k=1}^{m}\frac{z-s_k}{Ce^{\frac{1}{\alpha}(\sqrt{u}-T)}+s_k}\right)\,\mathrm{d}u\right|\notag\\
\le&\left|\widehat{E}^{(l)}_T(z)\right|+\left|\int_{(\kappa+1)^2 T^2}^{+\infty}
\frac{zC^{\alpha}(\sqrt{u}-T)^le^{\sqrt{u}-T}}{2 \sqrt{u}\left(Ce^{\frac{1}{\alpha}(\sqrt{u}-T)}+z\right)}
\left(\prod\limits_{k=1}^{m}\frac{z-s_k}{Ce^{\frac{1}{\alpha}(\sqrt{u}-T)}+s_k}\right)\,\mathrm{d}u\right|\\
\le &|\widehat{E}^{(l)}_T(z)|+\int_{\kappa T}^{+\infty}
\left|\frac{zC^{\alpha}t^le^{t}}{Ce^{\frac{1}{\alpha}t}+z}
\left(\prod\limits_{k=1}^{m}\frac{z-s_k}{Ce^{\frac{1}{\alpha}t}+s_k}\right)\right|\,\mathrm{d}t\notag\\
\le&\frac{2(1+T)^{l}\mathbb{T}_{m,\beta}C^{\alpha}}{{\cal X}(\beta)e^{T}}
\left(\frac{1}{\delta^{m}}+\frac{\kappa^{l+1}}{C^{m+1}}\right).\notag
\end{align}

\begin{figure}[htbp]
\centerline{\includegraphics[height=4.6cm,width=6cm]{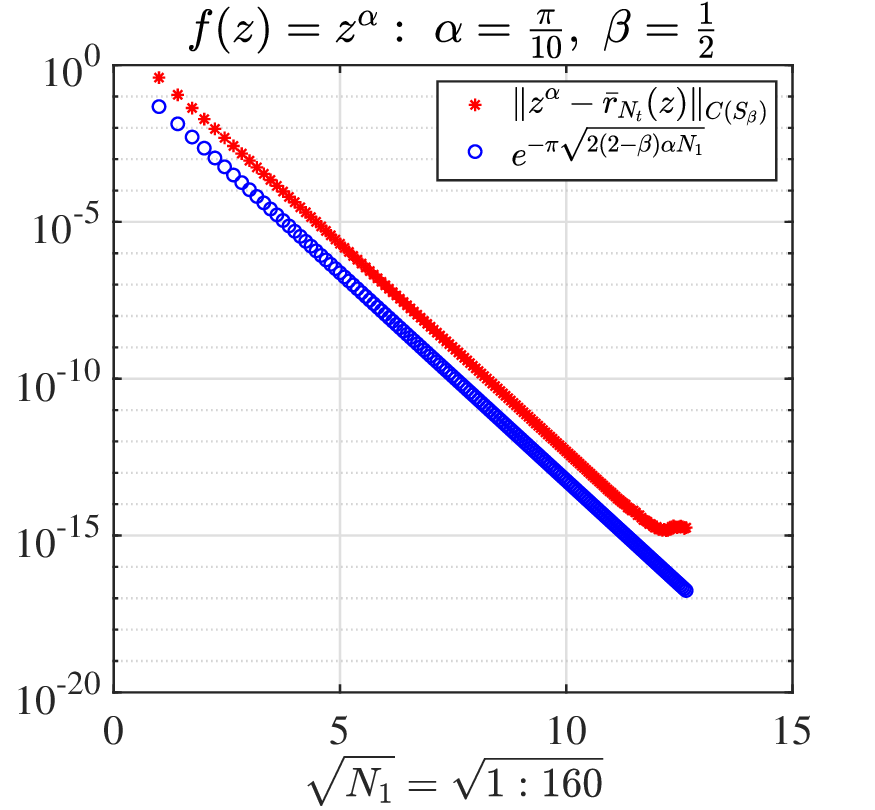}\includegraphics[height=4.6cm,width=6cm]{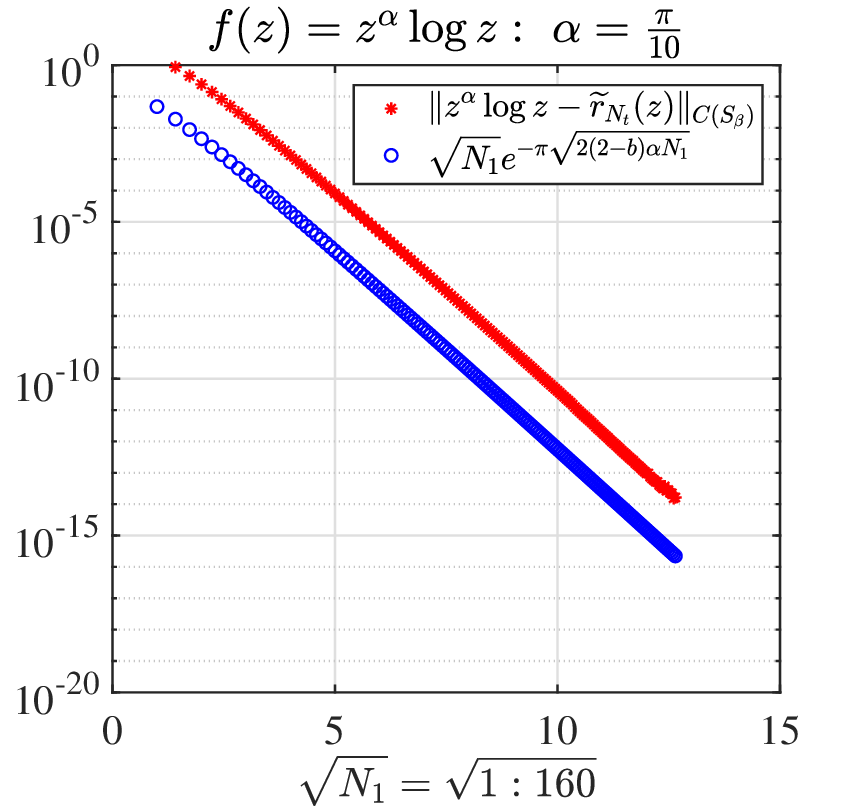}}
\caption{$\|z^\alpha-\overline{r}_{N_t}(z)\|_{C(S_\beta)}$ (left) and $\|z^\alpha\log z-\widetilde{r}_{N_t}(z)\|_{C(S_\beta)}$ (right) with  $\sigma=\frac{\pi\sqrt{2(2-\beta)}}{\sqrt{\alpha}}$ and $N_t=\lceil(\kappa+1)^2N_1\rceil$ where $\kappa$ is defined by \eqref{eq:kappa}.}
\label{rationalappr}
\end{figure}

Fig.~\ref{rationalappr} depicts the infinity-norm of the error in approximating $z^\alpha$ (left) and $z^\alpha\log z$ (right) by the rational functions
$\overline{r}_{N_t}(z)$ and $\widetilde{r}_{N_t}(z)$, respectively, over the sector $S_\beta$. These functions are derived from the representations \eqref{eq:zalpha} and \eqref{eq:cint_log_gener_substituting} combined with \eqref{eq:ECrat1CC}, correspond to $z\in S_{\beta}$ with parameters $\alpha=\frac{\pi}{10}$, $\beta=\frac{1}{2}$,  $\sigma=\frac{\pi\sqrt{2(2-\beta)}}{\sqrt{\alpha}}$ and $N_t=\lceil(\kappa+1)^2N_1\rceil $.

\bigskip
Furthermore, incorporating the exponential clustered poles $$p_j=-Ce^{\frac{1}{\alpha}(\sqrt{jh}-T)}
=-Ce^{-\sigma(\sqrt{N_1}-\sqrt{j})},\quad 1\le j\le N_t$$
into \eqref{eq:ECrat1CC},
the rational function $r_{N_t}^{(l)}(z)$ can be further rewritten as
\begin{align}\label{eq:ECrat1C}
r^{(l)}_{N_t}(z)=&h\sum_{j=1}^{N_t}\frac{z|p_j|^{\alpha}(\sqrt{jh}-T)^l}{2 \sqrt{jh}(z-p_j)}
\left(\prod\limits_{k=1}^{m}\frac{z-s_k}{s_k-p_j}\right)\notag\\
=&\sum_{j=1}^{N_1}\left(
\frac{p_j|p_j|^{\alpha}h[\sqrt{jh}-T]^l}{2 \sqrt{jh}(z-p_j)}+\frac{|p_j|^{\alpha}h[\sqrt{jh}-T]^l}{2 \sqrt{jh}}\right)
\left(\prod\limits_{k=1}^{m}\frac{z-s_k}{s_k-p_j}
\right)\\
&+\sum_{j=N_1+1}^{N_t}
\frac{z|p_j|^{\alpha}h[\sqrt{jh}-T]^l}{2\sqrt{jh}(z-p_j)}
\left(\prod\limits_{k=1}^{m}\frac{z-s_k}{s_k-p_j}
\right)\notag\\
=&\sum_{j=1}^{N_1}\frac{a^{(l)}_j}{z-p_j}
+\sum_{j=N_1+1}^{N_t}
\frac{z|p_j|^{\alpha}h(\sqrt{jh}-T)^l}{2\sqrt{jh}(z-p_j)}
\left(\prod\limits_{k=1}^{m}\frac{z-s_k}{s_k-p_j}
\right)+P^{(l)}_m(z)\notag\\
:=&r^{(l)}_{N_1}(z)+r^{(l)}_2(z)+P^{(l)}_m(z)\notag
 \end{align}
where $a^{(l)}_j=(-1)^{m}\frac{hp_j|p_j|^\alpha(\sqrt{jh}-T)^l}{2\sqrt{jh}}$ ($1\le j\le N_1$) is obtained by applying Cauchy's residue theorem, while $P^{(l)}_m(z)$ represents a polynomial of degree $m$.

\bigskip
It is worth noting that, by definition, $N_t=\lceil(\kappa+1)^2N_1\rceil$, which is significantly larger than $\sqrt{N_1}$. Subsequently, we show that the term
$r^{(l)}_2(z)$ in  \eqref{eq:ECrat1C} can be approximated by a polynomial of degree $N_2=\mathcal{O}(\sqrt{N_1})$.

\subsection{LP approximations with $N_2=\mathcal{O}(\sqrt{N_1})$ for $z^\alpha$ and $z^\alpha\log z$}\label{subsec22}
To establish LP approximations \eqref{eq:rat} with $N_2=\mathcal{O}(\sqrt{N_1})$, following the approaches in \cite{Herremans2023,XYW2024,XYW2025}, it is essential to invoke Runge's approximation theorem \cite[pp. 76-77]{Gaier1987,Walsh1965}.

\begin{theorem}\cite[1885, Runge]{Gaier1987}
Suppose $K\subset \mathbb{C}$ is compacted, $K^C=\mathbb{C}\setminus K$ is connected, and $f$ is analytic on $K$. Then there exist polynomials $\left\{P_n\right\}_{n=1}^{\infty}$ such that
\begin{align*}
\lim_{n\rightarrow \infty}\max_{z\in K}|f(z)-P_n(z)|=0.
\end{align*}
\end{theorem}

Based on a sequence of finitely connected domains \cite[pp. 8-9]{Walsh1965}, $P_n$ ($n=1,2\ldots$) are chosen as the interpolation polynomials constructed for the $n+1$ Fekete points $\{z_k\}_{k=0}^n$ on $\partial K$ satisfying
\begin{align*}
\|f-P_n\|_{C(K)}=\max_{z\in K}|f(z)-P_n(z)|=\mathcal{O}(q^n)
\end{align*}
for some $q\in(0,1)$ independent of $n$.

Observe that $p_j<-C$ for $N_1<j\le N_t$. We define the neighborhood $\Omega_{\rho}$ of $S_\beta$ as follows
\begin{align*}
\Omega_{\rho}=&\left\{z:\, z=xe^{\pm i\frac{\theta}{2}\pi},\, 0\le x\le 1+\mathbf{d}, \, 0\le \theta\le 2,\, \Re(z)\ge -{\bf d}\right\},&0\le \beta\le 1,\\
\Omega_{\rho}=&\left\{z:\, z=xe^{\pm i\frac{\theta}{2}\pi},\, 0\le x\le 1+\mathbf{d}, \, 0\le \theta\le 2\right\}\\
&\hspace{0.66cm}
 {\Large\setminus} \left\{z:\, z=-{\bf d}+xe^{\pm i\frac{ \theta }{2}\pi},\, 0< x\le 1+{\bf d}, \, \beta< \theta\le 2\right\},&1< \beta<2,
\end{align*}
where $\mathbf{d}=\min\left\{\frac{1}{2},\frac{C}{2}\right\}$ (see Fig. \ref{neighborhood}).
\begin{figure}[htbp]
\centerline{\includegraphics[width=8.8cm]{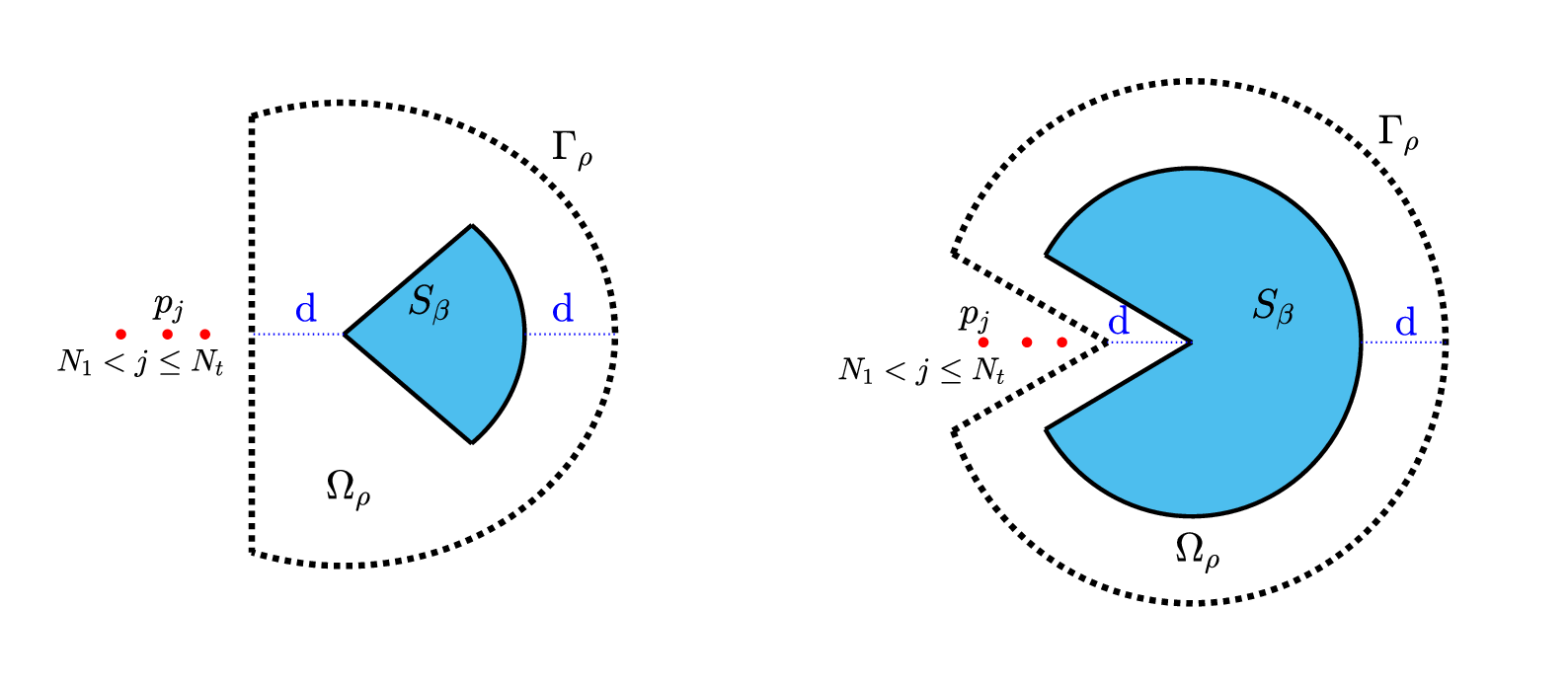}}
\vspace{-0.36cm}
\caption{Neighborhood $\Omega_{\rho}\supset S_{\beta}$ for $0\le \beta\le 1$ (left) and $1<\beta<2$ (right).}
\label{neighborhood}
\end{figure}
Under this construction,
$r_2^{(l)}(z)$ is analytic on $\Omega_\rho$. Furthermore, since $\Omega_{\rho}\subset \mathbf{B}(0,1+\mathbf{d})$, where $\mathbf{B}(0,1+\mathbf{d})$
denotes the disk of radius $1+\mathbf{d}$ centered at the origin, it follows that $|z|\le\frac{3}{2}$, $|z-s_k|\le |z|+s_k\le s_k+\frac{3}{2}$ for all $z\in \Omega_\rho$. Consequently, we have
$$\max\limits_{z\in  \Omega_\rho}\prod\limits_{k=1}^{m}|z-s_k|\le\prod\limits_{k=1}^{m}\left(s_k+\frac{3}{2}\right)
=\prod\limits_{k=1}^{m}\left[\delta+\omega+\frac{3}{2}+\omega\cos{\frac{(2k-1)\pi}{2m}}\right]
\le\left(\delta+\omega+\frac{3}{2}\right)^{m}=\aleph^m,
$$
where the last inequality follows from the arithmetic mean-geometric mean  inequality. In particular, for $\beta=0$, the aforementioned upper bound can be replaced by $\left(\delta+\omega+\frac{1}{2}\right)^{m}$, which is a direct consequence of the condition $|z-s_k|\le s_k+\frac{1}{2}$ for $z\in [-{\bf d},1]$.

We now proceed to establish an upper bound for $|r_2^{(l)}|$. Analogous to \eqref{eq:est}, for $0\le\beta\le 1$, let $z=xe^{\pm i\frac{ \theta }{2}\pi}\in \Omega_\rho$. We then have for  $t\ge 0$ that
\begin{align*}
\left| Ce^{\frac{1}{\alpha}t}+xe^{\pm i\frac{ \theta }{2}\pi}\right|
=& \sqrt{Ce^{\frac{2}{\alpha}t}+2xCe^{\frac{1}{\alpha}t}\cos\frac{\theta\pi}{2} +x^2}
\ge Ce^{\frac{1}{\alpha}t}, \quad  0\le \theta\le 1,\\
\left| Ce^{\frac{1}{\alpha}t}+xe^{\pm i\frac{ \theta }{2}\pi}\right|
\ge & Ce^{\frac{1}{\alpha}t}-|\Re(z)|
\ge Ce^{\frac{1}{\alpha}t}-\frac{1}{2}C\ge \frac{1}{2} Ce^{\frac{1}{\alpha}t}, \quad  1< \theta\le 2,\, -{\bf d}\le \Re(z)<0.
\end{align*}
On the other hand, for $1<\beta<2$, it follows that $\big{|}Ce^{\frac{1}{\alpha}t}+z\big{|}\ge Ce^{\frac{1}{\alpha}t}{\cal X}(\beta)$ for $z=xe^{\pm i\frac{ \theta }{2}\pi}\in \Omega_\rho$ with $0\le \theta\le \beta$ and $0\le x\le 1+\mathbf{d}$. Furthermore, for $z=-\frac{1}{2}u+xe^{\pm i\frac{ \theta }{2}\pi}\in \Omega_\rho$ with $0\le u\le \min\{1,C\}$ and $0\le x\le 1+{\bf d}$, we obtain for $t\ge 0$ that
\begin{align*}
&\left| Ce^{\frac{1}{\alpha}t}-\frac{1}{2}u+xe^{\pm i\frac{ \theta }{2}\pi}\right| \\
=&\begin{cases}
\sqrt{\left(Ce^{\frac{1}{\alpha}t}-\frac{1}{2}u\right)^2+2x\left(Ce^{\frac{1}{\alpha}t}-\frac{1}{2}u\right)\cos\frac{\theta\pi}{2} +x^2}, & 0\le \theta\le 1, \\
\sqrt{\left(\left(Ce^{\frac{1}{\alpha}t}-\frac{1}{2}u\right)\cos\frac{\theta\pi}{2}+x\right)^2+\left(Ce^{\frac{1}{\alpha}t}-\frac{1}{2}u\right)^2\sin^2\frac{\theta\pi}{2}}, & 1< \theta\le \beta< 2,
\end{cases} \\
\ge&\ \left(Ce^{\frac{1}{\alpha}t}-\frac{1}{2}u\right){\cal X}(\beta) \ge \left(Ce^{\frac{1}{\alpha}t}-\frac{1}{2}C\right){\cal X}(\beta) \ge \frac{1}{2}Ce^{\frac{1}{\alpha}t}{\cal X}(\beta).
\end{align*}
This implies that $\big{|}Ce^{\frac{1}{\alpha}t}+z\big{|}\ge \frac{1}{2}Ce^{\frac{1}{\alpha}t}{\cal X}(\beta)$ holds for all $z\in \Omega_\rho$.

Moreover,  for $N_1<j\le N_t-l$, by $Ce^{\frac{1}{\alpha}(\sqrt{jh}-T)}+s_k>Ce^{\frac{1}{\alpha}(\sqrt{jh}-T)}$ and
$$
(\sqrt{jh}-T)^l\le \sqrt{\mathcal{N}_th}-T\le \kappa T,\quad l=1;\quad\quad
(\sqrt{jh}-T)^l=1,\quad l=0,
$$
we deduce by $|z|\le \frac{3}{2}$ for $z\in \Omega_\rho$ that
\begin{align}\label{eq:boundonr2}
\left|r^{(l)}_2(z)\right|=&\left|h\sum_{j=N_1+1}^{N_t}\frac{z C^\alpha(\sqrt{jh}-T)^l e^{\sqrt{jh}-T}}{2\sqrt{jh}\left(Ce^{\frac{1}{\alpha}(\sqrt{jh}-T)}+z\right)}
\left(\prod\limits_{k=1}^{m}\frac{z-s_k}{Ce^{\frac{1}{\alpha}(\sqrt{jh}-T)}+s_k}\right)\right|\notag\\
\le&h\sum_{j=N_1+1}^{N_t-l}\frac{3(\sqrt{jh}-T)^le^{-\frac{1}{\kappa}(\sqrt{jh}-T)}\aleph^m}
{2\sqrt{jh}C^{m+1-\alpha}{\cal X}(\beta)}+lh\frac{3(\sqrt{N_th}-T)^le^{-\frac{1}{\kappa}(\sqrt{N_th}-T)}\aleph^m}
{2\sqrt{N_th}C^{m+1-\alpha}{\cal X}(\beta)}\\
\le&\frac{3(\kappa T)^l\aleph^m}{{\cal X}(\beta)C^{m-\alpha+1}}\int_{N_1h}^{+\infty}
\frac{e^{-\frac{1}{\kappa}(\sqrt{t}-T)}}{2 \sqrt{t}}\,\mathrm{d}t+
l\frac{\aleph^m T}{{\cal X}(\beta)C^{m-\alpha+1}}\notag\\
=&\frac{3\kappa^{l+1} T^l+l T}{{\cal X}(\beta)C^{m+1-\alpha}}\aleph^m,\notag
\end{align}
where the second inequality in \eqref{eq:boundonr2} is justified by the following estimates. For the first term, we have
$$
(\sqrt{jh}-T)^l\le (\sqrt{\mathcal{N}_th}-T)^l=(\kappa T)^l,\quad j=N_1+1,\ldots, N_t-l,\quad l=0,1.
$$
Furthermore, noting the monotonicity of the function $\frac{e^{-\frac{1}{\kappa}(\sqrt{t}-T)}}{2 \sqrt{t}}$ for $t>0$ (since $\frac{\mathrm{d}}{\mathrm{d}t}\frac{e^{-\frac{1}{\kappa}(\sqrt{t}-T)}}{2 \sqrt{t}}
=-\frac{e^{-\frac{1}{\kappa}(\sqrt{t}-T)}}{2^2 t}\left(\frac{1}{\kappa}+t^{-\frac{1}{2}}\right)<0$), the composite rectangular rule yields
\begin{align*}
h\sum_{j=N_1+1}^{N_t-l}3(\sqrt{jh}-T)^l
\frac{e^{-\frac{1}{\kappa}(\sqrt{jh}-T)}}
{2\sqrt{jh}}\le& 3(\kappa T)^l\int_{N_1h}^{+\infty}
\frac{e^{-\frac{1}{\kappa}(\sqrt{t}-T)}}{2 \sqrt{t}}\,\mathrm{d}t\\
=&3(\kappa T)^l\int_{T}^{+\infty}e^{-\frac{1}{\kappa}(u-T)}\,\mathrm{d}u=3\kappa^{l+1} T^l.
\end{align*}
For the second term, the condition $T^2=N_1h\ge 1$ ensures that
\begin{align*}
 h\frac{3(\sqrt{N_th}-T)e^{-\frac{1}{\kappa}(\sqrt{N_th}-T)}}
{2 \sqrt{N_th}}\le \frac{3he^{-1}\kappa}{2 (\kappa+1)T}=
\frac{3e^{-1}\kappa T}{2 (\kappa+1)N_1}<T,
\end{align*}
where we have used the inequality  $te^{-t}\le e^{-1}<1$ for $t\ge 0$.

\bigskip
We next demonstrate that $N_2$  can be chosen as $N_2=\mathcal{O}(\sqrt{N_1})$. According to the Hermite interpolation formula \cite[Theorem 2 of Chapter 8]{Walsh1965}, we have
\begin{align*}
r^{(l)}_2(z)-q^{(l)}_n(z)=\frac{1}{2\pi i}\int_{\Gamma_\rho}\frac{\ell_{n+1}(z)}{\ell_{n+1}(t)}\frac{r^{(l)}_2(t)}{t-z}\,\mathrm{d}t,\quad \ell_{n+1}(z)=\prod_{k=0}^n(z-z_k),\quad z\in S_\beta.
\end{align*}
This, together with the relation $\min_{t\in \Gamma_\rho,\,z\in S_\beta}|t-z|={\bf d}{\cal X}(\beta)$,   yields the following estimate
\begin{align*}
\|r^{(l)}_2-q^{(l)}_n\|_{C(S_\beta)}\le \frac{s(\Gamma_\rho)\|r^{(l)}_2\|_{C(\Omega_\rho)}}{2\pi\min_{t\in \Gamma_\rho,\,z\in S_\beta}|t-z|}\frac{\|\ell_{n+1}\|_{C(S_\beta)}}{\min_{{t\in \Gamma_\rho}|\ell_{n+1}(t)|}}\le \frac{3\|r^{(l)}_2\|_{C(\Omega_\rho)}}{2\mathbf{d}{\cal X}(\beta)}\frac{\|\ell_{n+1}\|_{C(S_\beta)}}{\min_{{t\in \Gamma_\rho}|\ell_{n+1}(t)|}},
\end{align*}
where $\Gamma_\rho$ is the boundary of $\Omega_\rho$, $s(\Gamma_\rho)$ is the length of $\Gamma_\rho$ satisfying $s(\Gamma_\rho)\le 2\pi(1+{\bf d})\le 3\pi$ and $\{z_k\}_{k=0}^n$ are the $n+1$ Fekete points on $\partial S_\beta$,
then following Levin and Saff \cite[(7)-(9)]{LevinSaff} and Taylor and Totik \cite[p. 469]{TaylorTotik2010}, there exists a $\rho>1$ and polynomial $q^{(l)}_n$ of degree $n$ such that
\begin{align}\label{Runge12}
\limsup_{n\rightarrow \infty}\|r^{(l)}_2-q^{(l)}_n\|_{C(S_\beta)}^{1/n}\le \limsup_{n\rightarrow \infty}\left(\frac{3}{2{\bf d}{\cal X}(\beta)
}\max_{z\in \Omega_\rho}\left|r^{(l)}_2\right|\right)^{1/n}\frac{1}{\rho}\le \limsup_{n\rightarrow \infty}\|r^{(l)}_2\|_{C(\Omega_\rho)}^{1/n}\frac{1}{\rho}.
\end{align}

In addition, from  the definition of limsup and \eqref{Runge12}, for fixed $\rho_0$ satisfying $1<\rho_0<\rho$ there is an positive integer $n_0$ such that
$$
\|r^{(l)}_2-q^{(l)}_n\|_{C(S_\beta)}\le \|r^{(l)}_2\|_{C(\Omega_\rho)}\frac{1}{\rho_0^n}
$$
for $n\ge n_0$, where $n_0$, $\rho$ and $\rho_0$ are independent of $T$ and $r^{(l)}_2$.
Analogous to \cite[p. 5]{Herremans2023} there is a polynomial $q^{(l)}_{N_2}(z)$ with
$N_2\ge\frac{\sigma\alpha}{\log{\rho_0}}\sqrt{N_1}=\mathcal{O}(\sqrt{N_1})\ge m$
such that $\rho_0^{-N_2}\le e^{-\sigma\alpha\sqrt{N_1}}=e^{-T}$ for $z\in S_\beta$ and then by \eqref{Runge12}
\begin{align}\label{polynomial app}
\bigg|E^{(l)}_{PA}(z)\bigg|=\bigg|r^{(l)}_2(z)-q^{(l)}_{N_2}(z)\bigg|
\le \frac{3\kappa^{l+1} T^l+l T}{{\cal X}(\beta)C^{m+1-\alpha}}\aleph^me^{-T}.
\end{align}
Consequently, by \eqref{eq:ECrat1C} and denoting
$P^{(l)}_{N_2}(z)=q^{(l)}_{N_2}(z)+P^{(l)}_m(z)$ with $N_2\ge m$, we obtain
\begin{align}\label{app_for_rNt}
r^{(l)}_{N_t}(z)=r^{(l)}_{N_1}(z)+P^{(l)}_{N_2}(z)+E^{(l)}_{PA}(z).
\end{align}

\subsection{LP approximation for $z^{\alpha}$ with rigorous error representations}\label{subsec31}

By combining
\eqref{eq:zalpha}, \eqref{eq:intC1}, \eqref{polynomial app} and \eqref{app_for_rNt} with $m=\lfloor\alpha\rfloor$, the LP approximation for $z^{\alpha}$ with $z\in S_{\beta}$ is established as
\begin{align}\label{ratappforzalpha}
  z^{\alpha}
  =&\frac{\sin(\alpha\pi)}{(-1)^{m}\alpha\pi}
  \left[r^{(0)}_{N_t}(z)+E_Q^{(0)}(z)+E^{(0)}_T(z)\right]
  +z\mathcal{L}[z^{\alpha-1};s_1,\ldots,s_m]\notag\\
  =&\frac{\sin(\alpha\pi)}{(-1)^{m}\alpha\pi}
  \left[r^{(0)}_{N_1}(z)+P^{(0)}_{N_2}(z)+E^{(0)}_{PA}(z)+E_Q^{(0)}(z)+E^{(0)}_T(z)\right]+z\mathcal{L}[z^{\alpha-1};s_1,\ldots,s_m] \\
  =&r_{N_1}(z)+P_{N_2}(z)+E(z):=r_{N}(z)+E(z)\notag
\end{align}
with
\begin{align*}
r_{N_1}(z)=&\sum_{j=1}^{N_1}\frac{a_j}{z-p_j},\ a_j=\frac{hp_j|p_j|^\alpha\sin(\alpha\pi)}{2\alpha\pi \sqrt{jh}},\ 1\le j\le N_1,\\
P_{N_2}(z) =&\frac{\sin(\alpha\pi)}{(-1)^{m}\alpha\pi}P^{(0)}_{N_2}(z)+z\mathcal{L}[z^{\alpha-1};s_1,\ldots,s_m],\\
E(z)=&\frac{\sin(\alpha\pi)}{(-1)^{m}\alpha\pi}\left[
E^{(0)}_T(z)+E_Q^{(0)}(z)+E^{(0)}_{PA}(z)\right],
\end{align*}
where $P_{N_2}(z)$ is a polynomial of degree $N_2=\mathcal{O}(\sqrt{N_1})$ (provided $N_2\ge m$) and $E(z)$ satisfies by \eqref{eq:211}, \eqref{eq:212}, \eqref{boundforEl(z)} and \eqref{polynomial app} that
\begin{align}\label{boundforEbar}
\left|E(z)\right|
\le&\frac{|\sin(\alpha\pi)|}{\alpha\pi}\bigg[
\frac{2\mathbb{T}_{m,\beta}C^{\alpha}}{{\cal X}(\beta)e^T}\left(\frac{1}{\delta^{m}}+\frac{\kappa }{C^{m+1}}\right)
+\frac{3\aleph^m\kappa e^{-T}}{C^{m+1-\alpha}{\cal X}(\beta)}
+\left|E_Q^{(0)}(z)\right|\bigg]\notag\\
\le&\frac{2|\sin(\alpha\pi)|\aleph^mC^{\alpha}}{\alpha\pi{\cal X}(\beta)e^T}
+\frac{|\sin(\alpha\pi)|\left[2 \aleph^m+3\aleph^m\right]}{(1-(\alpha))\pi C^{1-(\alpha)}{\cal X}(\beta)e^{T}}
+\frac{|\sin(\alpha\pi)|}{\alpha\pi}\left|E_Q^{(0)}(z)\right|,
\end{align}
where $\frac{|\sin(\alpha\pi)|}{\alpha\pi}$ and $\frac{|\sin(\alpha\pi)|}{(1-(\alpha))\pi}$ are uniformly bounded for all $\alpha>0$ from \eqref{eq:31}.

\subsection{LP approximation for $z^{\alpha}\log {z}$  with rigorous error representations}\label{subsec32}
 Analogously, from \eqref{eq:cint_log_gener_substituting}, \eqref{eq:intC1}, \eqref{polynomial app} and \eqref{app_for_rNt} with $m=\lceil\alpha\rceil$ we also establish the LP approximation for $z^{\alpha}\log{z}$ with $z\in S_{\beta}$
\begin{align}\label{ratappforzalphalog}
  z^{\alpha}\log{z}
  =&\frac{\sin(\alpha\pi)}{(-1)^m\alpha^2\pi}
  \left[r^{(1)}_{N_t}(z)+E_Q^{(1)}(z)+E^{(1)}_T(z)\right]
  +\bigg[\frac{\sin(\alpha\pi)\log{C}}{(-1)^m\alpha\pi}
  +\frac{\cos(\alpha\pi)}{(-1)^{m}\alpha}\bigg]\notag\\
  &\cdot\Big[r^{(0)}_{N_t}(z)+E_Q^{(0)}(z)+E^{(0)}_T(z)\Big]+z\mathcal{L}[z^{\alpha-1}\log{z};s_1,\ldots,s_m]\\
  =&\frac{\sin(\alpha\pi)}{(-1)^m\alpha^2\pi}
  \left[r^{(1)}_{N_1}(z)+P^{(1)}_{N_2}(z)+E^{(1)}_{PA}(z)+E_Q^{(1)}(z)+E^{(1)}_T(z)\right]\notag\\
  &+\left[\frac{\sin(\alpha\pi)\log{C}}{(-1)^m\alpha\pi}+\frac{\cos(\alpha\pi)}{(-1)^{m}\alpha}\right]
  \Big[r^{(0)}_{N_1}(z)+P^{(0)}_{N_2}(z)+E^{(0)}_{PA}(z)+E_Q^{(0)}(z)+E^{(0)}_T(z)\Big]\notag\\
  &+z\mathcal{L}[z^{\alpha-1}\log{z};s_1,\ldots,s_m]\notag\\
  =&\widetilde{r}_{N_1}(z)+\widetilde{P}_{N_2}(z)+\widetilde{E}(z):=\widetilde{r}_{N}(z)+\widetilde{E}(z)\notag
\end{align}
with
\begin{align*}
&\widetilde{r}_{N_1}(z)=\sum_{j=1}^{N_1}\frac{\widetilde{a}_j}{z-p_j},\\ &\widetilde{a}_j=\frac{hp_j|p_j|^\alpha(\sqrt{jh}-T)}{ 2\sqrt{jh}}\left\{\frac{\sin(\alpha\pi)}{\alpha^2\pi}+
\frac{\sin(\alpha\pi)\log{C}}{\alpha\pi}
+\frac{\cos(\alpha\pi)}{\alpha}\right\}\notag
\end{align*}
for $1\le j\le N_1$ according to Cauchy's residue theorem, and
\begin{align*}
\widetilde{P}_{N_2}(z) =&\frac{\sin(\alpha\pi)}{(-1)^{m}\alpha^2\pi}P^{(1)}_{N_2}(z)
+\left[\frac{\sin(\alpha\pi)\log{C}}{(-1)^m\alpha\pi}
+\frac{\cos(\alpha\pi)}{(-1)^{m}\alpha}\right]P^{(0)}_{N_2}(z)\notag\\
&+z\mathcal{L}[z^{\alpha-1}\log{z};s_1,\ldots,s_m],\\
\widetilde{E}(z)=&\frac{\sin(\alpha\pi)}{(-1)^{m}\alpha^2\pi}\left[
E^{(1)}_T(z)+E_Q^{(1)}(z)+E^{(1)}_{PA}(z)\right]\notag\\
&+\left[\frac{\sin(\alpha\pi)\log{C}}{(-1)^m\alpha\pi}
+\frac{\cos(\alpha\pi)}{(-1)^{m}\alpha}\right]
\left[E^{(0)}_T(z)+E_Q^{(0)}(z)+E^{(0)}_{PA}(z)\right],
\end{align*}
where $\widetilde{P}_{N_2}(z)$ is a polynomial of degree $N_2$ and $\widetilde{E}(z)$ satisfies by \eqref{boundforEl(z)}, \eqref{polynomial app} and $\kappa=\frac{\alpha}{\lceil\alpha\rceil+1-\alpha}$ that
\begin{align}\label{boundforwidetildeE}
&\left|\widetilde{E}(z)\right|
\le\frac{|\sin(\alpha\pi)|}{\alpha^2\pi}\bigg[
\frac{2(1+T)\aleph^mC^{\alpha}}{{\cal X}(\beta)e^{T}}
\left(1+\frac{\kappa^2}{C^{m+1}}\right)
+\frac{(3\kappa^2 T+ T)\aleph^{m}}{C^{m+1-\alpha}{\cal X}(\beta)}e^{-T}+\left|E_Q^{(1)}(z)\right|\bigg]\notag\\
&+\left|\frac{\sin(\alpha\pi)\log{C}}{(-1)^m\alpha\pi}
+\frac{\cos(\alpha\pi)}{(-1)^{m}\alpha}\right|
\bigg[
\frac{2\aleph^mC^{\alpha}}{{\cal X}(\beta)e^T}\left(1+\frac{\kappa }{C^{m+1}}\right)+\frac{3\kappa\aleph^me^{-T}}{C^{m+1-\alpha}{\cal X}(\beta)}
+\left|E_Q^{(0)}(z)\right|\bigg].
\end{align}

\begin{remark}
Given the results obtained in \eqref{ratappforzalpha}, \eqref{boundforEbar}, \eqref{ratappforzalphalog}, and \eqref{boundforwidetildeE}, proving Theorems \ref{mainthm} and \ref{mainthm2} can be reduced to bounding the uniform quadrature errors $E^{(l)}_Q(z)$ ($z\in S_\beta$, $l=0,1$) by means of Poisson's summation formula.
\end{remark}


\section{Poisson's summation formula and Cauchy's integral theorem}\label{sec:3}
Define for $l=0,1$ that
\begin{subequations}
\begin{align}
f^{(l)}(u,z)=&\frac{1}{2 \sqrt{u}}\frac{zC^{\alpha}(\sqrt{u}-T)^le^{\sqrt{u}-T}}
{Ce^{\frac{1}{\alpha}(\sqrt{u}-T)}+z}
\left(\prod_{k=1}^{m}\frac{z-s_k}{Ce^{\frac{1}{\alpha}(\sqrt{u}-T)}+s_k}\right),\label{eq:func}\\
I^{(l)}(z)=&\int_0^{N_t h}f^{(l)}(u,z)\mathrm{d}u,\quad T=\sqrt{N_1h}=\sigma\alpha\sqrt{N_1},\label{eq:quadraturec}\\
E^{(l)}_Q(z)=&I^{(l)}(z)-h\sum_{k=1}^{N_t}f^{(l)}(kh,z)=I^{(l)}(z)- r^{(l)}_{N_t}(z),\quad h=\sigma^2\alpha^2.\label{eq:quadraqll}
\end{align}
\end{subequations}

To derive the convergence rate for the composite rectangular rule in \eqref{eq:quadraqll}, we apply Poisson's summation formula (cf. \cite[(10.6-21)]{Henrici}, \cite[Theorem 1.3.1]{Stenger}, \cite{Trefethen2014SIREV,XYW2025}) along with the decay asymptotics of discrete Fourier transforms.

\begin{theorem}\cite[Theorem 1.3.1]{Stenger}\label{StengerPossionFormula}
Let $w\in L^2(\mathbb{R})$ and let $w$ and its Fourier transform
$\mathfrak{F}[w](\xi)=\int_{-\infty}^{\infty}w(u)e^{-2\pi iu\xi}\mathrm{d}u$ for $\xi$ and $u$ in $\mathbb{R}$, satisfy the
conditions
$$
w(u)=\lim_{t\rightarrow 0^+}\frac{w(u-t)+w(u+t)}{2},\quad \mathfrak{F}[w](\xi)=\lim_{t\rightarrow 0^+}\frac{\mathfrak{F}[w](\xi-t)+\mathfrak{F}[w](\xi+t)}{2}.
$$
Then, for all $h> 0$,
\begin{align}\label{Possionsummationformula}
h\sum_{n=-\infty}^{+\infty}w(nh)e^{2\pi inh u}= \sum_{n=-\infty}^{+\infty}\mathfrak{F}[w]\left(\frac{n}{h}+u\right).
\end{align}
\end{theorem}

From \eqref{Possionsummationformula} with $u=0$ and by $\mathfrak{F}[w]\left(0\right)=\int_{-\infty}^{+\infty}w(u)\mathrm{d}u$, it follows
\begin{align}\label{QuadratureErrorfor_w}
E^{w}_{Q}:=\int_{-\infty}^{+\infty}w(u)\mathrm{d}u
-h \sum_{j=-\infty}^{+\infty}w(jh)
=-\sum_{n\not=0}\mathfrak{F}[w]\left(\frac{n}{h}\right).
\end{align}

\bigskip
As previously mentioned, by applying the Paley-Wiener theorem to a horizontal strip \cite{XYW2025}, the optimal quadrature errors for uniformly exponentially clustered poles \eqref{eq:uniform0} have been directly established as $\mathcal{O}(e^{-\pi\sqrt{(2-\beta)N\alpha}})$ for $z^{\alpha}$ and $\mathcal{O}(\sqrt{N}e^{-\pi\sqrt{(2-\beta)N\alpha}})$ for $z^{\alpha}\log z$, based on the optimal choice of the clustered parameter $\sigma=\frac{\pi\sqrt{2-\beta}}{\sqrt{\alpha}}$ \cite{XYW2025}, where
   the associated integrand is analytic in a horizontal strip and belongs to
$L^2(\mathbb{R})$. However, it is important to note that while the integrand $f^{(l)}(u,z)$ is well-defined for $u>0$, it fails to belong to $L^2(\mathbb{R_+})$ for any fixed $z\in S_\beta\setminus\{0\}$. Furthermore, $f^{(l)}(u,z)$  exhibits a branch singularity at $u=0$.

\bigskip
To remedy this, we introduce a smooth extension of $f^{(l)}$
defined by
 \begin{align}\label{eq:extension_of_fnew11110}
{\bar f}^{(l)}(u,z)=
\begin{cases}
f^{(l)}(u,z), & u\ge \hbar,\\
f^{(l)}(\hbar,z), & -\hbar\le u\le \hbar,\\
f^{(l)}(-u,z), & u\le-\hbar,
\end{cases}
\end{align}
where $\hbar>0$ is an integer multiple of $h$. This construction bridges the gap in regularity and ensures square-integrability near the origin.

Consequently, for any fixed $z\in S_\beta$, the function ${\bar f}^{(l)}$
is continuous and piecewise smooth on $\mathbb{R}$. Furthermore, in conjunction with the bounds established in \eqref{ine:minus_u} and \eqref{ieq:positive_u}, it can be shown that
$\bar{f}^{(l)}(u,z)\in L^2(\mathbb{R})$ via the change of variables $t=\sqrt{u}-T$ provided $T\ge 1$
\begin{align*}
&\int_{-\infty}^{+\infty}|{\bar f}^{(l)}(u,z)|^2\,\mathrm{d}u=2\hbar|f^{(l)}(\hbar,z)|^2+2\int_{\hbar}^{+\infty}|f^{(l)}(u,z)|^2\,\mathrm{d}u\notag\\
\le&2\hbar |f^{(l)}(\hbar,z)|^2+2\int_{\hbar}^{(\kappa+1)^2T}|f^{(l)}(u,z)|^2\,\mathrm{d}u+\int_{\kappa T}^{+\infty}
\frac{1}{t+T}\left|\frac{zC^{\alpha}t^l e^t}{Ce^{\frac{1}{\alpha}t}+z}
\left(\prod\limits_{k=1}^{m}\frac{z-s_k}{Ce^{\frac{1}{\alpha}t}+s_k}\right)\right|^2\,\mathrm{d}t\\
\le& 2\hbar|f^{(l)}(\hbar,z)|^2+2\int_{\hbar}^{(\kappa+1)^2T}|f^{(l)}(u,z)|^2\,\mathrm{d}u+\frac{2\mathbb{T}_{m,\beta}^2{\cal X}(\beta)^{-2}}{C^{2m+2-2\alpha}} \int_0^{+\infty}t^{2l}e^{-\frac{2}{\kappa}t}\,\mathrm{d}t<+\infty.\notag
\end{align*}
Furthermore, there exists a positive constant $M$ such that, uniformly for all $\xi\in \mathbb{R}$,
\begin{align}\label{eq:fouriertran20}
\int_{-\infty}^{+\infty}|{\bar f}^{(l)}(u,z)e^{-2i\pi u\xi}|\,\mathrm{d}u
\le&2\hbar|f^{(l)}(\hbar,z)|+2\int_{\hbar}^{+\infty}|f^{(l)}(u,z)|\,\mathrm{d}u\notag\\
=&2\hbar|f^{(l)}(\hbar,z)|+2\int_{\sqrt{\hbar}-T}^{+\infty}\left|\frac{zC^{\alpha}t^l e^t}{Ce^{\frac{1}{\alpha}t}+z}
\left(\prod\limits_{k=1}^{m}\frac{z-s_k}{Ce^{\frac{1}{\alpha}t}+s_k}\right)\right|\,\mathrm{d}t\\
\le& 2\hbar|f^{(l)}(\hbar,z)|+\frac{2\mathbb{T}_{m,\beta}}{\delta^{m}}\frac{C^{\alpha}}{{\cal X}(\beta)}
\int_{-\infty}^0|t|^{l}e^{t}\,\mathrm{d}t\notag\\
&+\frac{2\mathbb{T}_{m,\beta}}{C^{m+1-\alpha}{\cal X}(\beta)} \int_0^{+\infty}t^{l}e^{-\frac{1}{\kappa}t}\,\mathrm{d}t<M.\notag
\end{align}
This uniform bound \eqref{eq:fouriertran20} implies that the Fourier transform $\mathfrak{F}[{\bar f}^{(l)}](\xi)$ is continuous on $\mathbb{R}$. Hence,
${\bar f}^{(l)}(u,z)$ and $\mathfrak{F}[{\bar f}^{(l)}](\xi)$ satisfy the conditions in  \cite[Theorem 1.3.1]{Stenger} and \cite[(10.6-12)]{Henrici}.

Based on the above, we define the $h$-periodic function in $v$ as
\begin{equation}\label{eq:per0}
F(v,z)=\sum_{k=-\infty}^{\infty}{\bar f}^{(l)}(kh+v,z), \quad v\in[0,h].
\end{equation}
The uniform convergence of this series can be readily verified from the definitions of  $f^{(l)}$ and ${\bar f}^{(l)}$, in conjunction with \eqref{ieq:positive_u}.
In accordance with \cite[p. 270]{Henrici} and the uniform convergence \eqref{eq:per0} on $v$, it follows that the $n$-th Fourier coefficient of $F(v,z)$ for $n\not=0$ satisfies
\allowdisplaybreaks[4]
\begin{align}\label{eq:fourier_c0}
c_n=&\frac{1}{h}\mathfrak{F}[{\bar f}^{(l)}]\big{(}\frac{n}{h}\big{)}
=\frac{1}{h}\int_0^{h}F(v,z)e^{-i\frac{2n\pi}{h}v}\,\mathrm{d}v =\frac{1}{h}
\int_{0}^{h}\sum_{k=-\infty}^{\infty}{\bar f}^{(l)}(kh+v,z)e^{-i\frac{2n\pi}{h}v}\,\mathrm{d}v\notag\\
=&\frac{1}{h}\sum_{k=-\infty}^{\infty}
\int_{kh}^{(k+1)h}{\bar f}^{(l)}(u,z)e^{-i\frac{2n\pi}{h}u}\,\mathrm{d}u\\
=&\frac{1}{h}
\int_{\hbar}^{+\infty} f^{(l)}(u,z)e^{-i\frac{2n\pi}{h}u}du+\frac{1}{h}
\int_{\hbar}^{+\infty} f^{(l)}(u,z)e^{i\frac{2n\pi}{h}u}\,\mathrm{d}u
\notag\\
&+\frac{2}{h}\int_{0}^{\hbar}f^{(l)}(\hbar,z)\cos{\big{(}\frac{2n\pi}{h}u\big{)}}\,\mathrm{d}u\notag\\
=&\frac{1}{h}
\int_{\hbar}^{+\infty} f^{(l)}(u,z)e^{-i\frac{2n\pi}{h}u}\,\mathrm{d}u+\frac{1}{h}
\int_{\hbar}^{+\infty} f^{(l)}(u,z)e^{i\frac{2n\pi}{h}u}\,\mathrm{d}u,\notag
\end{align}
where the term $\frac{2}{h}\int_{0}^{\hbar}f^{(l)}(\hbar,z)\cos{\big{(}\frac{2n\pi}{h}u\big{)}}\,\mathrm{d}u$
vanishes because $\hbar>0$ is an integer multiple of $h$;  consequently
 $$\frac{2}{h}\int_{0}^{\hbar}f^{(l)}(\hbar,z)\cos{\big{(}\frac{2n\pi}{h}u\big{)}}\,\mathrm{d}u=\frac{2 f^{(l)}(\hbar,z)}{h}\int_{0}^{\hbar}\cos{\big{(}\frac{2n\pi}{h}u\big{)}}\,\mathrm{d}u=0.$$

\bigskip
Directly analyzing the asymptotics of $\mathfrak{F}[ f^{(l)}]
\big(\frac{n}{h}\big)$ is quite challenging. Instead, we switch to a generalized Fourier transform. Notice that the integrals in the last line of \eqref{eq:fourier_c0} can be rewritten as
\begin{align}\label{eq:fourier_c10}
&\frac{1}{h}
\int_{\hbar}^{+\infty} f^{(l)}(u,z)e^{\pm i\frac{2n\pi}{h}u}\,\mathrm{d}u\notag\\
=&\frac{1}{h}
\int_{\hbar}^{+\infty}\frac{1}{2 \sqrt{u}}
\frac{zC^{\alpha}(\sqrt{u}-T)^le^{\sqrt{u}-T}}
{Ce^{\frac{1}{\alpha}(\sqrt{u}-T)}+z}
\left(\prod_{k=1}^{m}\frac{z-s_k}{Ce^{\frac{1}{\alpha}(\sqrt{u}-T)}+s_k}\right)e^{\pm i\frac{2n\pi}{h}u}\,\mathrm{d}u\\
=&\frac{1}{h}
\int_{\sqrt{\hbar}}^{+\infty}
\frac{zC^{\alpha}(y-T)^le^{y-T}}
{Ce^{\frac{1}{\alpha}(y-T)}+z}
\left(\prod_{k=1}^{m}\frac{z-s_k}{Ce^{\frac{1}{\alpha}(y-T)}+s_k}\right)e^{\pm i\frac{2n\pi}{h}y^2}\,\mathrm{d}y\notag\\
:=&\frac{1}{h}
\int_{\sqrt{\hbar}}^{+\infty}f^{(l,*)}(y,z)e^{\pm i\frac{2n\pi}{h}y^2}\,\mathrm{d}y\notag
\end{align}
where \begin{align}\label{eq:uniffun}
f^{(l,*)}(y,z)=\frac{zC^{\alpha}(y-T)^le^{y-T}}
{Ce^{\frac{1}{\alpha}(y-T)}+z}
\left(\prod_{k=1}^{\ell}\frac{z-s_k}{Ce^{\frac{1}{\alpha}(y-T)}+s_k}\right).
\end{align}
Combining \eqref{eq:fourier_c0} with the last identity in \eqref{eq:fourier_c10}, we can obtain delicate estimates for the discrete Fourier transform $\mathfrak{F}[ \overline{f}^{(l)}]
\big(\frac{n}{h}\big)$
\begin{align}\label{eq:fourier_cpath00}
\mathfrak{F}[ \overline{f}^{(l)}]
\big(\frac{n}{h}\big)=hc_n=&
\int_{\sqrt{\hbar}}^{+\infty}
\frac{zC^{\alpha}(y-T)^le^{y-T}}
{Ce^{\frac{1}{\alpha}(y-T)}+z}
\left(\prod_{k=1}^{m}\frac{z-s_k}{Ce^{\frac{1}{\alpha}(y-T)}+s_k}\right)e^{-i\frac{2n\pi}{h}y^2}\,\mathrm{d}y\notag\\
&+\int_{\sqrt{\hbar}}^{+\infty}
\frac{zC^{\alpha}(y-T)^le^{y-T}}
{Ce^{\frac{1}{\alpha}(y-T)}+z}
\left(\prod_{k=1}^{m}\frac{z-s_k}{Ce^{\frac{1}{\alpha}(y-T)}+s_k}\right)e^{ i\frac{2n\pi}{h}y^2}\,\mathrm{d}y
\end{align}
via Cauchy's integral theorem and the residue theorem. These estimates, together with Poisson's summation formula \eqref{QuadratureErrorfor_w}, finally yield the quadrature error in \eqref{eq:quadraqll}.

\bigskip
Below, we examine in detail the analytic properties of the function $f^{(l,*)}$  defined in \eqref{eq:uniffun}.
It is worth noting that for any fixed nonzero $z\in S_\beta$, defined as either
$z^+=xe^{\frac{\theta\pi}{2}i}$ or $z^-=xe^{-\frac{\theta\pi}{2}i}$  with $0\le \theta\le \beta$, the condition $s_k\in (\delta,\delta+2\omega)=(1,2)$ ensures that $z^{\pm}\notin\{s_1,s_2,\ldots,s_m\}$. Consequently, the function
$f^{(l,*)}(u,z^{\pm})=f^{(l,*)}(u,xe^{i\frac{\pm\theta\pi}{2}})$ exhibits simple poles at
\begin{equation}\label{eq:all_poles_fux_C}
u_j(z^{\pm})=T+\alpha\log{\frac{x}{C}}
+i\alpha\pi\left(2j-1\pm\frac{\theta}{2}\right),\quad j=0,\pm1,\ldots,
\end{equation}
and at
\begin{equation}\label{eq:all_poles_fux_C_sl}
u_j(s_k)=T+\alpha\log{\frac{s_k}{C}}
+i\alpha\pi\left(2j-1\right),\quad j=0,\pm1,\ldots,\quad k=1,\ldots,m.
\end{equation}

Among the poles $\left\{u_j(z^{\pm})\right\}$, the two closest to the real axis are
\begin{align*}
&u_0(z^+)=T+\alpha\log{\frac{x}{C}}
-i\alpha\pi\left(1-\frac{\theta}{2}\right),\quad u_1(z^+)=T+\alpha\log{\frac{x}{C}}
+i\alpha\pi\left(1+\frac{\theta}{2}\right),\\
&u_0(z^-)=T+\alpha\log{\frac{x}{C}}
-i\alpha\pi\left(1+\frac{\theta}{2}\right),\quad u_1(z^-)=T+\alpha\log{\frac{x}{C}}
+i\alpha\pi\left(1-\frac{\theta}{2}\right).
\end{align*}
Similarly, in each family $\{u_j(s_k)\}$, the closest poles to the real line are $u_0(s_k)$ and $u_1(s_k)$, which are symmetrically located.

In particular, for sufficiently large $A>T>0$ and any fixed $c_0>0$, the following bound holds on the vertical segments $A\pm it$, $0\le t\le  c_0$ \cite[(5.13)]{XYW2025}
\begin{align}\label{path10}
\left|f^{(l,*)}(A\pm it,z)\right|
\le&\frac{|z||A\pm it-T|^l\left|C^{\alpha}e^{A\pm it-T}\right|}{\big|Ce^{\frac{1}{\alpha}(A\pm it-T)}+z\big|}
\left|\prod\limits_{k=1}^{m}\frac{z-s_k}{Ce^{\frac{1}{\alpha}(A\pm it-T)}+s_k}\right|\notag\\
\le&\frac{(\sqrt{(A-T)^2+t^2})^lC^{\alpha}e^{A-T}\mathbb{T}_{m,\beta}}{\left|Ce^{\frac{1}{\alpha}(A\pm it-T)}\right|-|z|}
\prod\limits_{k=1}^{m}\frac{1}{\left|Ce^{\frac{1}{\alpha}(A\pm it-T)}\right|-s_k}\\
\le&\frac{(\sqrt{(A-T)^2+t^2})^lC^{\alpha}e^{A-T}\mathbb{T}_{m,\beta}}{Ce^{\frac{1}{\alpha}(A-T)}-s_m}
\prod\limits_{k=1}^{m}\frac{1}{Ce^{\frac{1}{\alpha}(A-T)}-s_m}\notag\\
\le& \frac{(\sqrt{(A-T)^2+t^2})^lC^{\alpha}\mathbb{T}_{m,\beta}}{\Big[Ce^{\frac{A-T}{(m+1)\kappa}}-s_m e^{-\frac{A-T}{m+1}}\Big]^{m+1}}\notag
\end{align}
which tends to zero as $A\rightarrow +\infty$, independent of $t\in [0,c_0]$ and $z\in S_\beta$.

\bigskip
To achieve optimal convergence, we assume
\begin{align}\label{eq:onT}
T>\max\{\sqrt{h}+\alpha, \sqrt{h}+\alpha-\alpha\log{\frac{1}{C}}, \sigma\alpha\left\lceil \frac{\sigma+2}{\sigma} \right\rceil+\alpha-\alpha\log{\frac{1}{C}},1\}.
\end{align}
We then define $x^*=Ce^{\frac{1}{\alpha}(\sqrt{h}+\alpha-T)}\in (0,1)$. Additionally, we introduce
  $m_0$ and $\lambda_0$
  as the smallest positive integers such that
 $4m_0\alpha\pi^2\ge h$ and $\sqrt{\lambda_0h}\ge \sqrt{h}+2\alpha=(\sigma+2)\alpha$, respectively. Their explicit expressions are
 \begin{align}\label{eq:ML}
  m_0=\left\lceil \frac{h}{4\alpha\pi^2}\right\rceil=\left\lceil\frac{\sigma^2\alpha}{4\pi^2}\right\rceil,\quad
\lambda_0 = \left\lceil \frac{\sigma+2}{\sigma} \right\rceil^2.
\end{align}

\begin{lemma}\label{infty}Let $f^{(l,*)}$ be defined in \eqref{eq:uniffun} with $z\in S_{\beta}$ and $z\not=0$. Then
\begin{align}\label{eq:intestall}
\left| \int_{\sqrt{h}\pm 2im_0\alpha\pi}^{+\infty\pm 2im_0\alpha\pi} f^{(l,*)}(y,z)e^{\pm i\frac{2n\pi}{h}y^2} \,\mathrm{d}y\right|
\le \frac{\aleph^m(T+2m_0\alpha\pi)^l}
{n^2\sigma\alpha{\cal X}(\beta) C^{-\alpha}}e^{-T}+\frac{\kappa(2m_0\alpha\pi+\kappa)^l\aleph^m}
{{\cal X}(\beta)C^{m+1-\alpha}}e^{-2nT}.
\end{align}
\end{lemma}
\begin{proof}
Note that
\begin{align}\label{eq:intestimate1}
&\left|\int_{\sqrt{h}-2im_0\alpha\pi}^{+\infty-2im_0\alpha\pi}f^{(l,*)}(y,z)e^{- i\frac{2n\pi}{h}y^2} \,\mathrm{d}y\right|
=\left|\int_{\sqrt{h}}^{+\infty}f^{(l,*)}(t-2im_0\alpha\pi,z)e^{-i\frac{2n\pi}{h}(t-2im_0\alpha\pi)^2}
\,\mathrm{d}t\right|\notag\\
&\le\int_{\sqrt{h}}^{+\infty}\left|\frac{zC^{\alpha}(t-2im_0\alpha\pi-T)^le^{t-2im_0\alpha\pi-T}}
{Ce^{\frac{1}{\alpha}(t-2im_0\alpha\pi-T)}+z}\right|\cdot\left|\prod_{k=1}^{m}\frac{z-s_k}{Ce^{\frac{1}{\alpha}(t-2im_0\alpha\pi-T)}+s_k}\right| e^{-\frac{8m_0n\alpha\pi^2}{h}t}\,\mathrm{d}t\\
&\le\left(\int_{\sqrt{h}}^{T}+\int_{T}^{+\infty}\right)\frac{|z|C^{\alpha}(|t-T|+2m_0\alpha\pi)^l e^{t-T}e^{-\frac{8m_0n\alpha\pi^2}{h}t}}
{\bigg|Ce^{\frac{1}{\alpha}(t-T)}+z\bigg|}
\cdot\prod_{k=1}^{m}\frac{|z-s_k|}{Ce^{\frac{1}{\alpha}(t-T)}+s_k}\,\mathrm{d}t.\notag
\end{align}
We now bound the two resulting integrals separately.
For the interval $[\sqrt{h},T]$: Applying \eqref{eq:est}, \eqref{ine:minus_u}, the condition $4m_0\alpha\pi^2\ge h$, and the inequality $e^{t}\ge 1+t>t\,(t>0)$ gives
\begin{align}\label{eq:intest2}
&\int_{\sqrt{h}}^{T}\frac{C^{\alpha}\mathbb{T}_{m,\beta}(|t-T|+2m_0\alpha\pi)^le^{t-T}}
{{\cal X}(\beta)\delta^{m}}
e^{-\frac{8m_0n\alpha\pi^2}{h}t}\,\mathrm{d}t\notag\\
\le&\frac{C^{\alpha}\mathbb{T}_{m,\beta}(T+2m_0\alpha\pi)^l}
{{\cal X}(\beta)\delta^{m}}e^{-T}\int_{\sqrt{h}}^{T}
e^{-\frac{4m_0n\alpha\pi^2}{h}t}\,\mathrm{d}t\\
\le&\frac{C^{\alpha}\mathbb{T}_{m,\beta}(T+2m_0\alpha\pi)^l}
{{\cal X}(\beta)\delta^{m}}e^{-T}\cdot \frac{h}{4m_0n\alpha\pi^2e^{\frac{4m_0n\alpha\pi^2}{\sqrt{h}}}}\notag\\
< & \frac{h\sqrt{h}C^{\alpha}\mathbb{T}_{m,\beta}(T+2m_0\alpha\pi)^l}
{16m_0^2n^2\alpha^2\pi^4{\cal X}(\beta)\delta^{m}}e^{-T}
\le \frac{ C^{\alpha}\aleph^m(T+2m_0\alpha\pi)^l}
{n^2\sigma\alpha{\cal X}(\beta)}e^{-T},\notag
\end{align}
where we applied $\frac{h}{4\pi^2\alpha}\le m_0\le \frac{h}{4\pi^2\alpha}+1=\frac{\sigma^2\alpha}{4\pi^2}+1$ to the last inequality of \eqref{eq:intest2}.

For the interval $[T,+\infty)$: From \eqref{eq:kappa}, \eqref{ieq:positive_u}, we obtain
\begin{align}\label{eq:intest3}
&\int_T^{+\infty}\frac{C^{\alpha}\mathbb{T}_{m,\beta}((t-T)+2m_0\alpha\pi)^le^{t-T}}
{{\cal X}(\beta) C^{m+1}e^{\frac{m+1}{\alpha}(t-T)}}
e^{-\frac{8m_0n\alpha\pi^2}{h}t}\,\mathrm{d}t\notag\\
\le & e^{-\frac{8m_0n\alpha\pi^2}{h}T}\frac{\mathbb{T}_{m,\beta}}{{\cal X}(\beta)C^{m+1-\alpha}}\int_{T}^{+\infty}(t-T+2m_0\alpha\pi)^l
e^{-\frac{1}{\kappa}(t-T)}\,\mathrm{d}t\\
\le&e^{-2nT}\frac{\kappa(2m_0\alpha\pi+\kappa)^l\aleph^m}
{{\cal X}(\beta)C^{m+1-\alpha}}.\notag
\end{align}
These observations together imply that \eqref{eq:intestall} holds for the integral
\begin{equation*}
\left|\int_{\sqrt{h}-2im_0\alpha\pi}^{+\infty-2im_0\alpha\pi}f^{(l,*)}(y,z)e^{- i\frac{2n\pi}{h}y^2} \mathrm{d}y\right|.
\end{equation*}
Analogously, the same bound holds for
 \begin{align*}
\left| \int_{\sqrt{h}+2im_0\alpha\pi}^{+\infty+2im_0\alpha\pi} f^{(l,*)}(y,z)e^{ i\frac{2n\pi}{h}y^2} \,\mathrm{d}y\right|
\le \frac{\aleph^m(T+2m_0\alpha\pi)^l}
{n^2\sigma\alpha{\cal X}(\beta) C^{-\alpha}}e^{-T}+\frac{\kappa(2m_0\alpha\pi+\kappa)^l\aleph^m}
{{\cal X}(\beta)C^{m+1-\alpha}}e^{-2nT}.
\end{align*}
\end{proof}

  Assume $\Re(u_0(z))=T+\alpha\log{\frac{x}{C}}>\sqrt{h}+\alpha$,  which is equivalent to the condition $x\in (x^*,1]$. We denote
 $$
   S_{\beta}'=\left\{z:\,\, z=xe^{i\frac{\pm \theta\pi}{2}}\in S_{\beta}\mbox{\,\, with $x\in (x^*,1]$}\right\}.
 $$
 From \eqref{eq:all_poles_fux_C} and \eqref{eq:all_poles_fux_C_sl}, the function $f^{(l,*)}(u,z^+)$ is holomorphic in the strip
\begin{align*}
\Xi=\left\{u\in\mathbb{C}:\ |\Re(u)|\ge \sqrt{h}, \ |\Im(u)|\le  2m_0\alpha\pi\right\}
\end{align*}
except for a finite number of simple poles. Under the condition in \eqref{eq:onT} that $T>\sqrt{h}+\alpha-\alpha\log\frac{1}{C}\ge \sqrt{h}+\alpha-\alpha\log\frac{\delta}{C}$ (which in turn implies $T+\alpha\log{\frac{s_k}{C}}>\sqrt{h}+\alpha$), these poles are given, for $j=1,2,\ldots, m_0$ and $k=1,2,\ldots,m$, by
\begin{align}\label{eq:strippoles}
\left\{\begin{array}{l}u{(2j-1,0)}=T+\alpha\log{\frac{x}{C}}
-i\alpha\pi\left(2j-1-\frac{\theta}{2}\right),\\
\hspace{0.65cm}u{(2j,0)}=T+\alpha\log{\frac{x}{C}}
+i\alpha\pi\left(2j-1+\frac{\theta}{2}\right),\\
u{(2j-1,k)}=T+\alpha\log{\frac{s_k}{C}}-i(2j-1)\alpha\pi,\\
\hspace{0.65cm}u{(2j,k)}=T+\alpha\log{\frac{s_k}{C}}+i(2j-1)\alpha\pi.\end{array}\right.
\end{align}

\begin{lemma}\label{residual} Let $f^{(l,*)}(u,z)$ be defined as in \eqref{eq:uniffun} for $z\in S_{\beta}'$, with $\sigma_{\rm opt}=\frac{\pi\sqrt{2(2-\beta)}}{\sqrt{\alpha}}$ and $\eta=\frac{\sigma_{\rm opt}}{\sigma}$. Then
\begin{subequations}\label{asy00}
\begin{align}\label{asy00odd}
\sum_{n=1}^{+\infty}\sum_{j=1}^{m_0}\sum_{k=0}^{m}\left|\mathrm{Res}\left(f^{(l,*)}(u,z)e^{-i\frac{2n\pi}{h}u^2},u{(2j-1,k)}\right)\right|
=\left\{\begin{array}{ll}
\mathcal{O}(T^le^{-T}),&\sigma\le\sigma_{\rm opt},\\
\mathcal{O}(e^{-\eta^2T}),&\sigma> \sigma_{\rm opt},\end{array}\right.
\end{align}
\begin{align}\label{asy00even}
\sum_{n=1}^{+\infty}\sum_{j=1}^{m_0}\sum_{k=0}^{m}\left|\mathrm{Res}\left(f^{(l,*)}(u,z)e^{i\frac{2n\pi}{h}u^2},u{(2j,k)}\right)\right|
=\left\{\begin{array}{ll}
\mathcal{O}(T^le^{-T}),&\sigma\le\sigma_{\rm opt},\\
\mathcal{O}(e^{-\eta^2T}),&\sigma> \sigma_{\rm opt},\end{array}\right.
\end{align}
\end{subequations}
where the constants implicit in the  $\mathcal{O}(\cdot)$ terms are independent of  $z\in S_\beta'$ and $T$.
\end{lemma}
\begin{proof}
Without loss of generality, we restrict our attention to the estimates in \eqref{asy00} for  $z=z^+=xe^{i\frac{\theta\pi}{2}}\in S_{\beta}'$. The conjugate case,
$z=z^-=xe^{-i\frac{\theta\pi}{2}}\in S_{\beta}'$, follows analogously.

(i) Residue at the typical pole $u(2j-1,0)$: Consider the poles arising from the equation
$$Ce^{\frac{1}{\alpha}(u{(2j-1,0)}-T)}+z=0, \quad {\rm i.e.,}\quad u{(2j-1,0)}=T+\alpha\log{\frac{x}{C}}
-i\alpha\pi\left(2j-1-\frac{\theta}{2}\right).$$
Since $\sqrt{h}+\alpha< \Re(u(2j-1,0))=T+\alpha\log{\frac{x}{C}}\le T+\alpha\log{\frac{1}{C}}$, it follows
\begin{align}\label{eq:res1}
&\left|\mathrm{Res}\left(f^{(l,*)}(u,z)e^{-i\frac{2n\pi}{h}u^2},u(2j-1,0)\right)\right|\notag\\
=&\left|\lim_{u\rightarrow u(2j-1,0)}(u-u(2j-1,0))\frac{zC^{\alpha}(u-T)^le^{u-T}}
{Ce^{\frac{1}{\alpha}(u-T)}+z}
\left(\prod_{k=1}^{m}\frac{z-s_k}{Ce^{\frac{1}{\alpha}(u-T)}+s_k}\right)e^{-i\frac{2n\pi}{h}u^2}\right|\\
=&\left|\lim_{u\rightarrow u(2j-1,0)}(u-u(2j-1,0))\frac{C^{\alpha}(u-T)^le^{u-T}}
{1-e^{\frac{1}{\alpha}(u-u(2j-1,0))}}
\left(\prod_{k=1}^{m}\frac{z-s_k}{Ce^{\frac{1}{\alpha}(u-T)}+s_k}\right)e^{-i\frac{2n\pi}{h}u^2}\right|\notag\\
=&\left|\lim_{u\rightarrow u(2j-1,0)}\frac{u-u(2j-1,0)}
{e^{\frac{1}{\alpha}(u-u(2j-1,0))}-1}
\right|\notag\\
&\cdot\left|e^{-i\frac{2n\pi}{h}[u(2j-1,0)]^2}(u(2j-1,0)-T)^lC^{\alpha}e^{u(2j-1,0)-T}\left(\prod_{k=1}^{m}\frac{z-s_k}{s_k-z}\right)\right|\notag\\
=&\alpha x^{\alpha}\left|\alpha\log{\frac{x}{C}}-i\left(2j-1-\frac{\theta}{2}\right)\alpha\pi\right|^le^{-\frac{2n\pi^2\alpha }{h}(4j-2-\theta)(T+\alpha\log{\frac{x}{C}})}\notag\\
\le&\alpha x^{\alpha}\left[\left|\alpha\log{\frac{x}{C}}\right|+2m_0\alpha\pi\right]^le^{-\frac{2n\pi^2\alpha }{h}(4j-2-\theta)(T+\alpha\log{\frac{x}{C}})}.\notag
\end{align}
Summing the upper bound in \eqref{eq:res1} over $j=1,\ldots,m$ and $n\ge 1$, and using the relations $h=\sigma^2\alpha^2$, $\eta^2=\frac{\sigma_{\rm opt}^2}{\sigma^2}=\frac{2\pi^2\alpha(2-\beta)}{h}$, along with the condition $T+\alpha\log{\frac{x}{C}}>\sqrt{h}+\alpha=(\sigma+1)\alpha$, we obtain
\begin{align}\label{asy0000}
&\sum_{n=1}^{+\infty}\sum_{j=1}^{m_0}\left|\mathrm{Res}\left(f^{(l,*)}(u,z)e^{-i\frac{2n\pi}{h}u^2},u{(2j-1,0)}\right)\right|\notag\\
\le&\alpha x^{\alpha}\left[\left|\alpha\log{\frac{x}{C}}\right|+2m_0\alpha\pi\right]^l\sum_{n=1}^{+\infty}\frac{e^{-\frac{2n\pi^2\alpha(2-\theta) }{h}(T+\alpha\log{\frac{x}{C}})}}
{1-e^{-\frac{8n\pi^2\alpha}{h}(T+\alpha\log{\frac{x}{C}})}}\\
\le&\alpha x^{\alpha}\left[\left|\alpha\log{\frac{x}{C}}\right|+2m_0\alpha\pi\right]^l\frac{1}{1-e^{-\frac{8\pi^2(\sigma+1)\alpha^2}{h}}}
\sum_{n=1}^{+\infty}e^{-\frac{2n\pi^2\alpha(2-\theta) }{h}(T+\alpha\log{\frac{x}{C}})}\notag\\
\le&\alpha x^{\alpha}\left[\left|\alpha\log{\frac{x}{C}}\right|+2m_0\alpha\pi\right]^l\frac{\sigma^2}{8\pi^2(\sigma+1)}e^{\frac{8\pi^2(\sigma+1)}{\sigma^2}}
\frac{1}{e^{\frac{2\pi^2\alpha(2-\theta) }{h}(T+\alpha\log{\frac{x}{C}})}-1}\notag\\
\le&\alpha x^{\alpha}\left[\left|\alpha\log{\frac{x}{C}}\right|+2m_0\alpha\pi\right]^l\frac{\sigma^2e^{\frac{8\pi^2(\sigma+1)}{\sigma^2}}}{8\pi^2(\sigma+1)}
\frac{e^{\frac{2\pi^2\alpha(2-\theta)(\sigma+1)\alpha}{h}}}{e^{\frac{2\pi^2\alpha(2-\theta)(\sigma+1)\alpha}{h}}-1}
e^{-\frac{2\pi^2\alpha(2-\theta) }{h}(T+\alpha\log{\frac{x}{C}})}\notag\\
\le&\alpha x^{\alpha}\left[\left|\alpha\log{\frac{x}{C}}\right|+2m_0\alpha\pi\right]^l\frac{\sigma^2e^{\frac{8\pi^2(\sigma+1)}{\sigma^2}}}{8\pi^2(\sigma+1)}
\frac{he^{\frac{2\pi^2\alpha(2-\theta)(\sigma+1)\alpha}{h}}}{2\pi^2\alpha(2-\theta)(\sigma+1)\alpha}
e^{-\frac{2\pi^2\alpha(2-\theta) }{h}(T+\alpha\log{\frac{x}{C}})}\notag\\
\le&\alpha x^{\alpha}\left[\left|\alpha\log{\frac{x}{C}}\right|+2m_0\alpha\pi\right]^l\frac{\sigma^4e^{\frac{12\pi^2(\sigma+1)}{\sigma^2}}}{16\pi^4(\sigma+1)^2(2-\theta)}
e^{-\frac{2\pi^2\alpha(2-\theta) }{h}(T+\alpha\log{\frac{x}{C}})}\notag\\
\le&\alpha x^{\alpha}\left[\left|\alpha\log{\frac{x}{C}}\right|+2m_0\alpha\pi\right]^l\frac{\sigma^4e^{\frac{12\pi^2(\sigma+1)}{\sigma^2}}}{16\pi^4(\sigma+1)^2(2-\beta)}
e^{-\eta^2 (T+\alpha\log{\frac{x}{C}})}.\notag
\end{align}
In the proof of the third and fifth identities in \eqref{asy0000}, we applied the inequality $\frac{e^t}{e^t-1}\le t^{-1}e^t$ for $t>0$.
 Specifically, we used this bound with $t=\frac{8\pi^2(\sigma+1)}{\sigma^2}$ and
$t=\frac{2\pi^2\alpha(2-\theta)(\sigma+1)\alpha}{h}$, respectively.

Thus, for the case $l=0$, the maximum of $x^{\alpha}e^{-\eta^2 (T+\alpha\log{\frac{x}{C}})}$ over $[x^*,1]$ is bounded by
\begin{align}\label{eq:balance}
 \max_{x\in [x^*,1]} x^{\alpha}e^{-\eta^2 (T+\alpha\log{\frac{x}{C}})}=\left\{\begin{array}{ll} (x^*)^{\alpha}e^{-\eta^2 (T+\alpha\log{\frac{x^*}{C}})}=  C^\alpha e^{(1-\eta^2)(\sigma+1)\alpha} e^{-T},&\eta\ge 1,\\
 e^{-\eta^2 (T+\alpha\log{\frac{1}{C}})}=  C^{\alpha\eta^2} e^{-\eta^2T},&\eta< 1,\end{array}\right.
\end{align}
which follows from the fact that $\frac{\mathrm{d}}{\mathrm{d}x}\log\left(x^{\alpha}e^{-\eta^2 (T+\alpha\log{\frac{x}{C}})}\right)=\frac{\alpha}{x}(1-\eta^2)$. Consequently, we derive
\begin{align}\label{asy0000-0}
\sum_{n=1}^{+\infty}\sum_{j=1}^{m_0}\left|\mathrm{Res}\left(f^{(l,*)}(u,z)e^{-i\frac{2n\pi}{h}u^2},u{(2j-1,0)}\right)\right|\le \left\{\begin{array}{ll}\frac{\alpha\sigma^4e^{\frac{12\pi^2(\sigma+1)}{\sigma^2}}}{16\pi^4(\sigma+1)^2(2-\beta)} C^\alpha e^{-T},&\eta\ge 1,\\
\frac{\alpha\sigma^4e^{\frac{12\pi^2(\sigma+1)}{\sigma^2}}}{16\pi^4(\sigma+1)^2(2-\beta)}   C^\alpha e^{-\eta^2T},&\eta< 1.\end{array}\right.
\end{align}

\bigskip
To achieve the optimal convergence rate in \eqref{eq:quadraqll} for the special case $l=1$ without the factor $T$ for the case $\sigma>\sigma_{\rm opt}$, we proceed by considering different regimes for  $x$.

If $x\in [\min\{1,C\},1]$, then from \eqref{asy0000} together with the definition of  $m_0$,
we obtain directly
\begin{align*}
&\sum_{n=1}^{+\infty}\sum_{j=1}^{m_0}\left|\mathrm{Res}\left(f^{(l,*)}(u,z)e^{-i\frac{2n\pi}{h}u^2},u{(2j-1,0)}\right)\right|
=\mathcal{O}(e^{-\eta^2T}).
\end{align*}

If $x\in [x^*,\min\{1,C\}]$, define
$$
\varsigma(x)=-\alpha x^{\alpha}\log\left(\frac{x}{C}\right)e^{-\frac{2\pi^2\alpha(2-\beta)}{h}(T+\alpha\log{\frac{x}{C}})}=-\alpha x^{\alpha}\log\left(\frac{x}{C}\right)
e^{-\eta^{2}(T+\alpha\log\frac{x}{C})}(\ge0).
$$
Notably for $\sigma\le \sigma_{\rm opt}$ (i.e., $\eta\ge 1$), the derivative of $\varsigma(x)$
satisfies
$$
\frac{\mathrm{d}}{\mathrm{d}x}\varsigma(x)=-\alpha x^{\alpha-1}e^{-\eta^2(T+\alpha\log{\frac{x}{C}})}
\left[\alpha \left(1-\eta^2\right)\log{\frac{x}{C}}+1\right]\le 0,
$$
then function $\varsigma$ is decreasing and
$$
0<\varsigma(x)\le \varsigma(x^*)\le   C^\alpha(T-\sqrt{h}-\alpha)e^{-T+(1-\eta^2)(\sqrt{h}+\alpha)}\le  C^\alpha Te^{-T}.
$$
For $\sigma> \sigma_{\rm opt}$ (i.e., $\eta< 1$), it is easy to verify that
$\frac{d}{dx}\varsigma(x)<0$ for $x\ge C$, whereas $\frac{\mathrm{d}}{\mathrm{d}x}\varsigma(x)<0$  for sufficiently small$x>0$.
Moreover, $\widehat{x}=Ce^{-\frac{1}{\alpha(1-\eta^2)}}$ is
the unique critical point of  $\varsigma(x)$ for $x\in (0,C](\supseteq [x^*,\min\{1,C\}])$, which yields for $x\in [x^*,\min\{1,C\}]$ that
$$
\varsigma(x)\le \left\{\begin{array}{ll}\varsigma(\widehat{x})=\frac{C^{\alpha}e^{-\eta^2T}}{e(1-\eta^2)}\,(\le e^{-1}C^{\alpha}(T-\sqrt{h}-\alpha)e^{-\eta^2T}),
&\frac{1}{1-\eta^2}\le T-\sqrt{h}-\alpha\, ({\rm i.e.,}\,\, \widehat{x}\in [x^*,C]),\\
\varsigma\left(x^*\right)=\mathcal{O}(T^le^{-T}),&\frac{1}{1-\eta^2}> T-\sqrt{h}-\alpha\, ({\rm i.e.,}\,\, \widehat{x}<x^*),\end{array}\right.
$$
where the constants implicit in the above two $\mathcal{O}(\cdot)$ terms are independent of $T$ and  $z \in S_\beta'$.

Collecting the two cases gives, for $x\in (x^*,1]$,
\begin{align}\label{asy000}
&\sum_{n=1}^{+\infty}\sum_{j=1}^{m_0}\left|\mathrm{Res}\left(f^{(l,*)}(u,z)e^{-i\frac{2n\pi}{h}u^2},u{(2j-1,0)}\right)\right|
=\left\{\begin{array}{ll}
\mathcal{O}(T^le^{-T}),&\sigma\le\sigma_{\rm opt},\\
\mathcal{O}(e^{-\eta^2T}),&\sigma> \sigma_{\rm opt},\end{array}\right.
\end{align}
where we have utilized the property that $Te^{-T}=\mathcal{O}(e^{-\eta^2T})$ is uniformly bounded for all $T>0$ when $\eta<1$.
Furthermore, the  constants in the $\mathcal{O}(\cdot)$ terms in \eqref{asy000}  are uniform, i.e., independent of both $T$ and  $z \in S_\beta'$.

(ii) Residues at the poles $u(2j-1,k)$: Consider the poles determined by the equation
$$Ce^{\frac{1}{\alpha}(u{(2j-1,k)}-T)}+s_k=0, \quad {\rm i.e.,}\quad u{(2j-1,k)}=T+\alpha\log{\frac{s_k}{C}}-i(2j-1)\alpha\pi.$$
For these poles, the residue can be bounded as follows
\begin{align*}
&\left|\mathrm{Res}\left[f^{(l,*)}(u,z)e^{-i\frac{2n\pi}{h}u^2},u{(2j-1,k)}\right]\right|\notag\\
=&\left|\alpha zs_k^{\alpha-1}\left(\alpha\log{\frac{s_k}{C}}-i(2j-1)\alpha\pi\right)^le^{-\frac{4(2j-1)n\pi^2\alpha }{h}(T+\alpha\log{\frac{s_k}{C}})}
\prod_{\displaystyle \substack{v=1 \\ v \neq k}}^{m}\frac{z-s_v}{s_v-s_k}\right|\\
\le&\alpha s_k^{\alpha-1}\left(\alpha\left|\log{\frac{s_k}{C}}\right|+2m_0\alpha\pi\right)^l\mathbb{T}_{m,\beta}e^{-\frac{4(2j-1)n\pi^2\alpha }{h}(T+\alpha\log{\frac{s_k}{C}})}
\prod_{\displaystyle \substack{v=1 \\ v \neq k}}^{m}\frac{1}{|s_v-s_k|}\notag\\
\le&\alpha \max\{\delta^{\alpha-1}, (\delta+2\omega)^{\alpha-1}\} \left(\alpha\max_{1\le k\le m}\left|\log{\frac{s_k}{C}}\right|+2m_0\alpha\pi\right)^l\mathbb{T}_{m,\beta}\notag\\
&\hspace{0.36cm}\cdot e^{-\frac{4(2j-1)n\pi^2\alpha }{h}(T+\alpha\log{\frac{1}{C}})}\frac{2^{m-1}}{m\omega^{m-1}}\notag\\
=&\mathcal{O}(1) e^{-\frac{4(2j-1)n\pi^2\alpha }{h}(T+\alpha\log{\frac{1}{C}})},\notag
\end{align*}
where  the constant implicit in the
$\mathcal{O}(\cdot)$ term is independent of $n$, $T$ and $z\in S_\beta'$.
In this estimate, the product term is bounded using the standard estimate for the Chebyshev points $x_k=\cos\frac{(2k-1)\pi}{2m}$ ($k=1,\ldots,m$) given in \cite[(2.48)]{Mason2003}
\begin{align*}
\prod_{\displaystyle \substack{v=1\\v\neq k}}^{m}\frac{1}{|s_v-s_k|}
=& \frac{1}{\omega^{m-1}}\prod_{\displaystyle \substack{v=1\\v\neq k}}^{m}\frac{1}{|x_v-x_k|}\le\frac{2^{m-1}\big|\sin\frac{(2k+1)\pi}{2m}\big|}{m\omega^{m-1}}\le \frac{2^{m-1}}{m\omega^{m-1}}.
\end{align*}
This inequality relies on the exact representation
\begin{align*}
\prod_{\displaystyle \substack{k=1 \\ k \neq j}}^{m}(x_j-x_k)=\frac{1}{2^{m-1}}\frac{d}{dt}T_m(t)\bigg|_{t=x_j}=\frac{m}{2^{m-1}}\frac{\sin\frac{m(2j+1)\pi}{2m}}{\sin\frac{(2j+1)\pi}{2m}}=
\frac{(-1)^jm}{2^{m-1}\sin\frac{(2j+1)\pi}{2m}}.
\end{align*}

Following a procedure analogous to that of \eqref{asy0000}, we sum over the indices $j$, $k$ and $n$. By invoking the condition
$T+\alpha\log{\frac{1}{C}}>\sqrt{h}+\alpha=(\sigma+1)\alpha$, we arrive at
\begin{align}\label{asy0k}
&\sum_{n=1}^{+\infty}\sum_{j=1}^{m_0}\sum_{k=1}^{m}\left|\mathrm{Res}\left[f^{(l,*)}(u,z)e^{-i\frac{2n\pi}{h}u^2},u(2j-1,k)\right]\right|\notag\\
=& \mathcal{O}(1) \sum_{n=1}^{+\infty}\sum_{j=1}^{m_0} e^{-\frac{4(2j-1)n\pi^2\alpha }{h}(T+\alpha\log{\frac{1}{C}})}\\
=&\mathcal{O}(1) \sum_{n=1}^{+\infty}e^{-\frac{4n\pi^2\alpha }{h}(T+\alpha\log{\frac{1}{C}})}=\frac{\mathcal{O}(1)e^{-\frac{4\pi^2\alpha }{h}(T+\alpha\log{\frac{1}{C}})}}{1-e^{-\frac{4\pi^2\alpha }{h}(T+\alpha\log{\frac{1}{C}})}} \notag\\
=&\mathcal{O}(1)e^{-\frac{2}{2-\beta}\eta^2T}
=\left\{\begin{array}{ll}
\mathcal{O}(e^{-T}),&\sigma\le\sigma_{\rm opt},\\
\mathcal{O}(e^{-\eta^2T}),&\sigma> \sigma_{\rm opt},
\end{array}\right.\notag
\end{align}
where we have used the fact that $\frac{4\pi^2\alpha}{h}=\frac{2}{2-\beta}\eta^2$ and $\frac{2}{2-\beta}\ge 1$. It is important to note that all implicit constants in the
$\mathcal{O}(\cdot)$ terms are independent of  $T$ and $z\in S_\beta'$.

By combining \eqref{asy000} with \eqref{asy0k}, we arrive at \eqref{asy00odd}. Note that all implicit constants in the
$\mathcal{O}(\cdot)$ terms are independent of both $z\in S_\beta'$ and $T$. The derivation of \eqref{asy00even} proceeds via a similar argument, replacing$-\theta$ with
$+\theta$ in \eqref{eq:res1} and \eqref{asy0000}, respectively. Consequently, the upper bounds in the final inequality of \eqref{eq:res1} and in the penultimate line of \eqref{asy0000} are correspondingly reduced, while the upper bound in \eqref{asy0k} remains unchanged due to the symmetry of the poles about the real axis.
\end{proof}

\begin{theorem}\label{eq:quaderrthm}
Let $f^{(l)}(u,z)$ be defined as in \eqref{eq:func}, and let $z=xe^{i\frac{\pm \theta\pi}{2}}\in S_{\beta}$ with $0\le \theta\le \beta$. Then the quadrature error \eqref{eq:quadraqll} satisfies
\begin{align}\label{errquad0}
E^{(l)}_Q(z)=\left\{\begin{array}{ll}
\mathcal{O}(T^le^{-T})=\mathcal{O}(T^le^{-\pi\eta^{-1}\sqrt{2(2-\beta)N\alpha}}),&\sigma\le \sigma_{\rm opt},\\
\mathcal{O}(e^{-\pi\eta\sqrt{2(2-\beta)N\alpha}}),&\sigma> \sigma_{\rm opt},
\end{array}\right.\quad \eta=\frac{\sigma_{\rm opt}}{\sigma},
\end{align}
uniformly for $z\in S_{\beta}$. Moreover, the constant implied by the $\mathcal{O}(\cdot)$ notation is independent of $T$ and $z\in S_{\beta}$.
\end{theorem}
\begin{proof}
Since  $E^{(l)}_Q(0)=0$ is immediate  from $f^{(l)}(u,0)\equiv 0$,
it suffices to bound $E^{(l)}_Q(z)$ for $z\not=0$. To avoid repetition, we concentrate on the convergence rate in \eqref{errquad0} for  $z=xe^{i\frac{\theta\pi}{2}}\in S_{\beta}$ with $0\le \theta\le \beta$  and  $z\not=0$.
The case  $z=xe^{-i\frac{\theta\pi}{2}}$
follows analogously.


\textbf{Case 1: $\Re(u_0(z))=T+\alpha\log{\frac{x}{C}}>\sqrt{h}+\alpha$ (i.e., $z\in S_\beta'$)}. In \eqref{eq:extension_of_fnew11110} and \eqref{eq:fourier_c10}, we set $\hbar=h$.
 Without loss of generality, we focus on estimating  $c_n$ for $n\ge 1$
using Cauchy's integral formula and the residue theorem. The case $n\le -1$ is treated analogously by shifting the integration path upward or downward.

From \eqref{eq:all_poles_fux_C} and \eqref{eq:all_poles_fux_C_sl}, the function $f^{(l,*)}(u,z)$ is holomorphic in the strip $\Xi$
except for a finite number of simple poles listed in \eqref{eq:strippoles}.

With the choice  $\hbar=h$, equation \eqref{eq:fourier_cpath00} provides the following representation for the Fourier coefficient
\begin{align}\label{eq:fourier_cpath0}
\mathfrak{F}[\overline{f}^{(l)}]
\big(\frac{n}{h}\big)=hc_n=&
\int_{\sqrt{h}}^{+\infty}
\frac{zC^{\alpha}(y-T)^le^{y-T}}
{Ce^{\frac{1}{\alpha}(y-T)}+z}
\left(\prod_{k=1}^{m}\frac{z-s_k}{Ce^{\frac{1}{\alpha}(y-T)}+s_k}\right)e^{-i\frac{2n\pi}{h}y^2}\,\mathrm{d}y\notag\\
&+\int_{\sqrt{h}}^{+\infty}
\frac{zC^{\alpha}(y-T)^le^{y-T}}
{Ce^{\frac{1}{\alpha}(y-T)}+z}
\left(\prod_{k=1}^{m}\frac{z-s_k}{Ce^{\frac{1}{\alpha}(y-T)}+s_k}\right)e^{ i\frac{2n\pi}{h}y^2}\,\mathrm{d}y.
\end{align}
Moreover, estimate \eqref{path10} guarantees that for $0\le t\le 2m_0\alpha\pi$
\begin{align*}
|f^{(l,*)}(A \pm it,z)e^{\pm i\frac{2n\pi}{h}(A\pm-it)^2}|
\le |f^{(l,*)}(A \pm it,z)|\rightarrow 0 \mbox{\quad as $A\rightarrow +\infty$}.
  \end{align*}
This ensure that
\begin{align}\label{eq:Cauchycon0}
\lim_{A\to+\infty} \int_{0}^{2m_0\alpha\pi}
f^{(l,*)}(A \pm it, z) e^{\pm i\frac{2n\pi}{h}(A \pm it)^{2}} \,\mathrm{d}t = 0.
\end{align}
Hence, applying Cauchy's integral theorem to $f^{(l,*)}(y,z)e^{\pm i\frac{2n\pi}{h}y^2}$
(see the contour in Fig. \ref{integral_contour}), we can evaluate $hc_n$
as
\begin{align}\label{eq:fourier_cpath}
hc_n
=&\int_{\sqrt{h}-2im_0\alpha\pi}^{+\infty-2im_0\alpha\pi}f^{(l,*)}(y,z)e^{- i\frac{2n\pi}{h}y^2} \,\mathrm{d}y+ \int_{\sqrt{h}+2im_0\alpha\pi}^{+\infty+2im_0\alpha\pi}f^{(l,*)}(y,z)e^{ i\frac{2n\pi}{h}y^2} \,\mathrm{d}y\notag\\
&+i\int_{0}^{2m_0\alpha\pi}\big{[}f^{(l,*)}(\sqrt{h}+it,z)e^{i\frac{2n\pi}{h}(\sqrt{h}+it)^2}-f^{(l,*)}(\sqrt{h}-it,z)e^{-i\frac{2n\pi}{h}(\sqrt{h}-it)^2})\big{]}
\, \mathrm{d}t\\
&-2\pi i\sum_{j=1}^{m_0}\sum_{k=0}^{m}\mathrm{Res}\left[f^{(l,*)}(u,z)e^{-i\frac{2n\pi}{h}u^2},u(2j-1,k)\right]\notag\\
&-2\pi i\sum_{j=1}^{m_0}\sum_{k=0}^{m}\mathrm{Res}\left[f^{(l,*)}(u,z)e^{i\frac{2n\pi}{h}u^2},u(2j,k)\right].\notag
\end{align}

\begin{figure}[htp]
\centerline{\includegraphics[height=6.36cm,width=12.6cm]{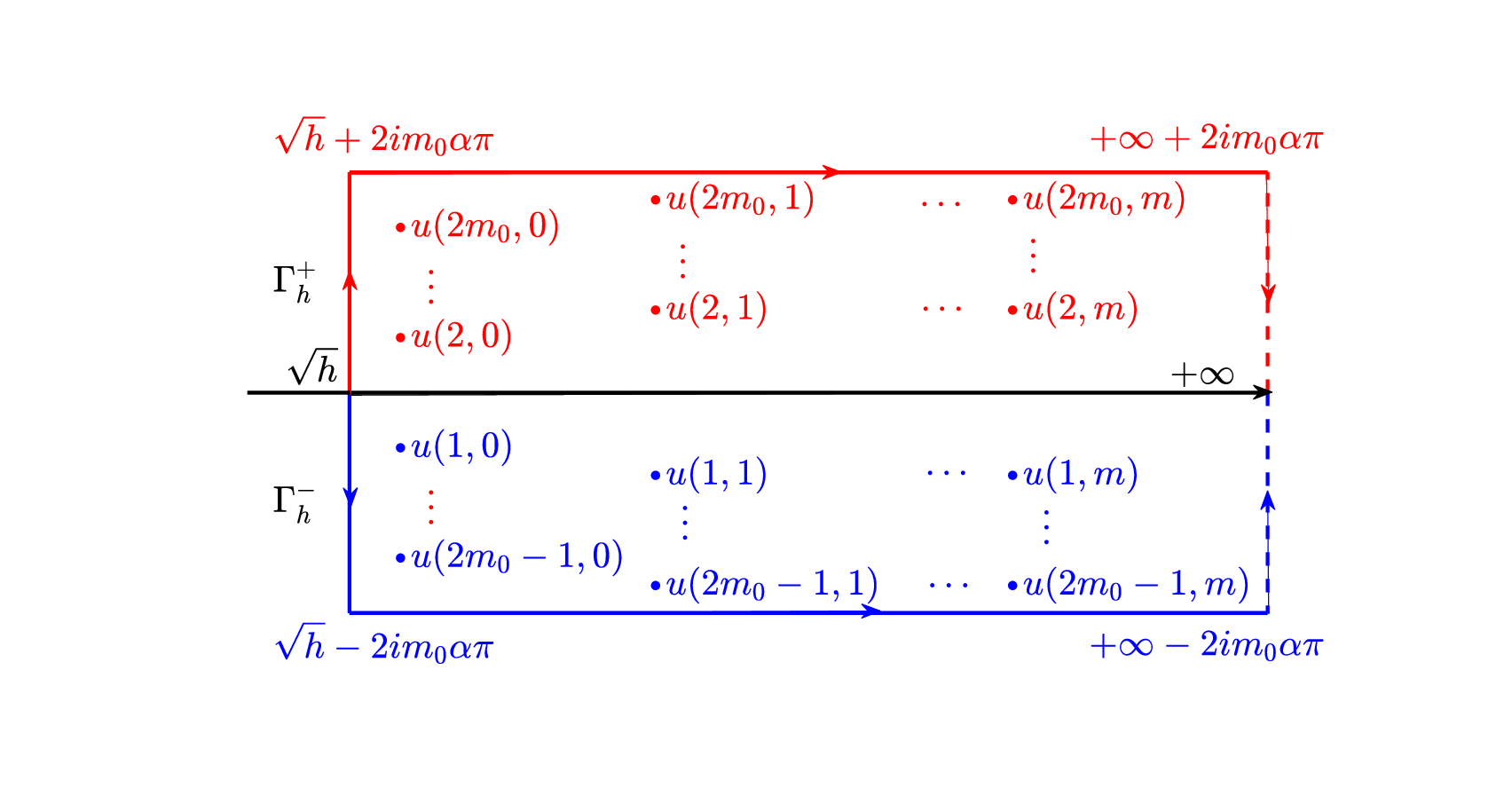}}
\vspace{-0.8cm}
\caption{The integral contours $\Gamma^{-}_{h}$ (blue) and $\Gamma^{+}_{h}$ (red). The poles of $f^{(l,*)}(u,z)$ lying in the interior of contour $\Gamma^{\mp}_h$ for the two integrals in \eqref{eq:fourier_cpath0} are $u(2j,k),\, u(2j-1,k)$, $j=1,\cdots,m_0,\ k=0,1,\cdots,m$.}
\label{integral_contour}
\end{figure}

In view of Lemma \ref{infty} and Lemma \ref{residual}, it suffices to estimate the third integral in \eqref{eq:fourier_cpath}.
The third integral  can be bounded by the representation given below
\begin{align}\label{eq:intestimate5}
&\int_0^{2m_0\pi\alpha}\left[f^{(l,*)}(\sqrt{h}+i t,z)e^{i\frac{2n\pi}{h}(\sqrt{h}+i t)^2}-f^{(l,*)}(\sqrt{h}-i t,z)e^{-i\frac{2n\pi}{h}(\sqrt{h}-i t)^2}\right]
\,\mathrm{d}t\notag\\
=&\int_0^{2m_0\pi\alpha}e^{-\frac{4n\pi}{\sqrt{h}}t}\left[f^{(l,*)}(\sqrt{h}+i t,z)-f^{(l,*)}(\sqrt{h}-i t,z)\right]e^{-i\frac{2n\pi}{h}t^2}\,\mathrm{d}t\\
&-\int_0^{2m_0\pi\alpha}e^{-\frac{4n\pi}{\sqrt{h}}t}f^{(l,*)}(\sqrt{h}-i t,z)\left[e^{i\frac{2n\pi}{h}t^2}-e^{-i\frac{2n\pi}{h}t^2}\right]
\,\mathrm{d}t.\notag
\end{align}

For the first term on the right-hand side of \eqref{eq:intestimate5}, we use the identities $z=-Ce^{\frac{1}{\alpha}(u(1,0)-T)}$ to express $f^{(l,*)}(\sqrt{h}+i t,z)$ as
\begin{align*}
f^{(l,*)}(\sqrt{h}+i t,z)=&\frac{zC^{\alpha}(\sqrt{h}+i t-T)^le^{\sqrt{h}+i t-T}}
{Ce^{\frac{1}{\alpha}(\sqrt{h}+i t-T)}+z}
\prod_{k=1}^{m}\frac{z-s_k}{Ce^{\frac{1}{\alpha}(\sqrt{h}+i t-T)}+s_k}\\
=&\frac{C^{\alpha}(\sqrt{h}+i t-T)^le^{\sqrt{h}+i t-T}}
{1-e^{\frac{1}{\alpha}(\sqrt{h}+it-u(1,0))}}
\prod_{k=1}^{m}\frac{z-s_k}{Ce^{\frac{1}{\alpha}(\sqrt{h}+i t-T)}+s_k}.
\end{align*}
This representation is analytic for $t\in [-\frac{\alpha\pi}{2}, \frac{\alpha\pi}{2}]$. This follows from
$$
\left|1-e^{\frac{1}{\alpha}(\sqrt{h}+it-u(1,0))}\right|\ge 1-e^{\frac{1}{\alpha}(\sqrt{h}-\Re(u(1,0)))}\ge 1-e^{-1}
$$
since $\Re(u(1,0))>\sqrt{h}+\alpha$, together with the inequality
\begin{align}\label{eq:sk}
\left|Ce^{\frac{1}{\alpha}(\sqrt{h}+it-T)}+s_k\right|=\sqrt{C^2e^{\frac{2}{\alpha}(\sqrt{h}-T)}+2s_kCe^{\frac{1}{\alpha}(\sqrt{h}-T)}\cos\left(\frac{t}{\alpha}\right)+s_k^2}
\ge s_k.
\end{align}
Consequently, the function is continuous in both variables over this interval. Applying the mean-value theorem \cite[Theorem 10]{Mcleod1965} then gives
$$
|f^{(l,*)}(\sqrt{h}+i t,z)-f^{(l,*)}(\sqrt{h}-i t,z)|\le 2t\|\partial_tf^{(l,*)}(\sqrt{h}+i t,z)\|_{C([-{\frac{\pi\alpha}{2}},{\frac{\pi\alpha}{2}}]\times S_\beta)}.
$$
Next, we split the integration at $t=\frac{\pi\alpha}{2}$ and obtain the following two bounds
\begin{align}\label{eq:3-1}
&\left|\int_0^{\frac{\pi\alpha}{2}}e^{-\frac{4n\pi}{\sqrt{h}}t}\left[f^{(l,*)}(\sqrt{h}+i t,z)- f^{(l,*)}(\sqrt{h}-i t,z)\right]e^{i\frac{2n\pi}{h}t^2}\,\mathrm{d}t\right|\notag\\
\le&2\|\partial_tf^{(l,*)}(\sqrt{h}+i t,z)\|_{C([-{\frac{\pi\alpha}{2}},{\frac{\pi\alpha}{2}}]\times S_\beta)} \int_0^{\frac{\pi\alpha}{2}}te^{-\frac{4n\pi}{\sqrt{h}}t}\,\mathrm{d}t\\
\le& \frac{h}{8n^2\pi^2}\|\partial_tf^{(l,*)}(\sqrt{h}+i t,z)\|_{C([-{\frac{\pi\alpha}{2}},{\frac{\pi\alpha}{2}}]\times S_\beta)}\notag
\end{align}
and
\begin{align}\label{eq:3-2}
&\left|\int_{\frac{\pi\alpha}{2}}^{2m_0\pi\alpha}e^{-\frac{4n\pi}{\sqrt{h}}t}\left[f^{(l,*)}(\sqrt{h}+i t,z)- f^{(l,*)}(\sqrt{h}-i t,z)\right]e^{i\frac{2n\pi}{h}t^2}\,\mathrm{d}t\right|\notag\\
\le&2\|f^{(l,*)}(\sqrt{h}+i t,z)\|_{C([-2m_0\alpha\pi,2m_0\alpha\pi]\times S_\beta)} \int_{\frac{\pi\alpha}{2}}^{2m_0\pi\alpha}e^{-\frac{4n\pi}{\sqrt{h}}t}\,\mathrm{d}t\\
\le&\|f^{(l,*)}(\sqrt{h}+i t,z)\|_{C([-2m_0\alpha\pi,2m_0\alpha\pi]\times S_\beta)}\frac{\sqrt{h}}{2n\pi}e^{-\frac{2n\pi^2\alpha}{\sqrt{h}}}\notag\\
\le& \frac{e^{-1}\sigma^2\alpha}{4n^2\pi^3}\|f^{(l,*)}(\sqrt{h}+i t,z)\|_{C([-2m_0\alpha\pi,2m_0\alpha\pi]\times S_\beta)},\notag
\end{align}
where we used  the inequality $te^{-t}\le e^{-1}$ (for $t>0$) with $t=\frac{2n\pi^2\alpha}{\sqrt{h}}$ and $h=\sigma^2\alpha^2$ to the second inequality in \eqref{eq:3-2}.

For the second term on the right-hand side of \eqref{eq:intestimate5}, we bound the sine factor by $\left|\sin\left(\frac{2n\pi}{h}t^2\right)\right|\le \frac{2n\pi}{h}t^2$, which leads to
\begin{align}\label{eq:3-3}
&\left|\int_0^{2m_0\pi\alpha}e^{-\frac{4n\pi}{\sqrt{h}}t}f^{(l,*)}(\sqrt{h}-i t,z)\left[e^{i\frac{2n\pi}{h}t^2}-e^{-i\frac{2n\pi}{h}t^2}\right]
\,\mathrm{d}t\right|\notag\\
\le&2\|f^{(l,*)}(\sqrt{h}-i t,z)\|_{C([0,2m_0\alpha\pi]\times S_\beta)}\int_0^{2m_0\pi\alpha}e^{-\frac{4n\pi}{\sqrt{h}}t}\left|\sin\left(\frac{2n\pi}{h}t^2\right)\right|
\,\mathrm{d}t\\
\le&\frac{4n\pi}{h}\|f^{(l,*)}(\sqrt{h}-i t,z)\|_{C([0,2m_0\alpha\pi]\times S_\beta)}\int_0^{2m_0\pi\alpha}t^2e^{-\frac{4n\pi}{\sqrt{h}}t}
\,\mathrm{d}t\notag\\
\le&\frac{\sqrt{h}}{8n^2\pi^2}\|f^{(l,*)}(\sqrt{h}-i t,z)\|_{C([0,2m_0\alpha\pi]\times S_\beta)}.\notag
\end{align}

Moreover, applying  the identities $z=-Ce^{\frac{1}{\alpha}(u(1,0)-T)}$ and $s_k=-Ce^{\frac{1}{\alpha}(u(1,k)-T)}$, together with the lower bounds for $t\in [-2m_0\alpha\pi,2m_0\alpha\pi]$
 $$\left|e^{\frac{1}{\alpha}(\sqrt{h}-it-u(1,k))}-1\right|\ge 1-e^{-1},\quad \left|e^{\frac{1}{\alpha}(u(1,k)-\sqrt{h}+it)}-1\right|\ge e-1, \quad k=0,1,\ldots,m,
$$
we obtain the following estimates
$$
\left|\frac{z}{Ce^{\frac{1}{\alpha}(\sqrt{h}-it-T)}+z}\right|=\frac{1}{\left|1-e^{\frac{1}{\alpha}(\sqrt{h}-it-u(1,0))}\right|}\le \frac{1}{1-e^{-1}},
$$
 $$
 \left|\frac{Ce^{\frac{1}{\alpha}(\sqrt{h}-it-T)}}{Ce^{\frac{1}{\alpha}(\sqrt{h}-it-T)}+z}\right|=\frac{1}{\left|1-e^{\frac{1}{\alpha}(u(1,0)-\sqrt{h}+it)}\right|}\le \frac{1}{e-1},
$$
$$
 \left|\frac{1}{Ce^{\frac{1}{\alpha}(\sqrt{h}-it-T)}+s_k}\right|=\frac{1}{s_k\left|1-e^{\frac{1}{\alpha}(\sqrt{h}-it-u(1,k))}\right|}\le \frac{1}{(1-e^{-1})s_k},
$$
and
$$
 \left|\frac{Ce^{\frac{1}{\alpha}(\sqrt{h}-it-T)}}{Ce^{\frac{1}{\alpha}(\sqrt{h}-it-T)}+s_k}\right|=\frac{1}{\left|1-e^{\frac{1}{\alpha}(u(1,k)-\sqrt{h}+it)}\right|}\le \frac{1}{e-1}.
$$
Combining these bounds with \eqref{eq:sk}, direct computations yield
\begin{align}\label{eq:l1}
\|f^{(l,*)}(\sqrt{h}+i t,z)\|_{C([-2m_0\alpha\pi,2m_0\alpha\pi]\times S_\beta)}
\le \frac{C^{\alpha}(T+2m_0\alpha\pi)^l\mathbb{T}_{m,\beta}}
{(1-e^{-1})^{m+1}\delta^{m}e^{T-\sqrt{h}}}=\mathcal{O}(T^le^{-T})
\end{align}
and
\begin{align}\label{eq:l2}
&\|\partial_tf^{(l,*)}(\sqrt{h}+i t,z)\|_{C([-{\frac{\pi\alpha}{2}},{\frac{\pi\alpha}{2}}]\times S_\beta)}\notag\\
\le& \frac{C^{\alpha}\mathbb{T}_{m,\beta}e^{\sqrt{h}-T}}
{(1-e^{-1})^{m+1}\delta^{m}}
\cdot\left\{l+\left(T+\frac{\pi\alpha}{2}\right)^l\left[1+\frac{1}{\alpha(e-1)}+\frac{m}{\alpha(e-1)}\right]\right\}\\
=&\mathcal{O}(T^le^{-T}).\notag
\end{align}

Therefore,  substituting \eqref{eq:l1} and \eqref{eq:l2} into \eqref{eq:3-1}--\eqref{eq:3-3} gives
\begin{align}\label{eq:intestimate3int}
&\left|\int_0^{2m_0\pi\alpha}\left[f^{(l,*)}(\sqrt{h}+i t,z)e^{i\frac{2n\pi}{h}(\sqrt{h}+i t)^2}-f^{(l,*)}(\sqrt{h}-i t,z)e^{-i\frac{2n\pi}{h}(\sqrt{h}-i t)^2}\right]
\,\mathrm{d}t\right|\notag\\
\le&\frac{\sigma^2\alpha^2}{8n^2\pi^2}\frac{C^{\alpha}\mathbb{T}_{m,\beta}e^{\sqrt{h}-T}}
{\delta^{m}(1-e^{-1})^{m+1}}
\cdot\left\{l+\left(T+\frac{\pi\alpha}{2}\right)^l\left[1+\frac{1+m}{\alpha(e-1)}\right]\right\}\\
&+\frac{\sigma\alpha}{8n^2\pi^2}\frac{C^{\alpha}(T+2m_0\alpha\pi)^l\mathbb{T}_{m,\beta}}
{(1-e^{-1})^{m+1}\delta^{m}e^{T-\sqrt{h}}}
\cdot\left[1+\frac{2\sigma}{e\pi}\right]\notag\\
=&\mathcal{O}(1)n^{-2} T^le^{-T},\notag
\end{align}
where the constant implicit in $\mathcal{O}(1)$  is independent of  $n$, $T$ and $z\in S_\beta'$.

\bigskip
Thus, all the estimates in \eqref{eq:fourier_cpath} on $hc_n$ ($n\ge 1$) are established.
Following an analogous approach, the corresponding estimates  for $c_n$ ($n\le -1$) can be obtained by replacing $n$ with $|n|$ and adapting the contour integration paths used in Fig. \ref{integral_contour} accordingly.

Therefore, from \eqref{eq:fourier_cpath}, Lemma \ref{infty}, Lemma \ref{residual}, and \eqref{eq:intestimate3int}, we obtain
\begin{align}\label{Poissonsum10}
&h\sum_{n\ne 0}|c_n|=\sum_{n\ne0}\bigg|\mathfrak{F}[{\bar f}^{(l)}]\big{(}\frac{n}{h}\big{)}\bigg|\notag\\
=&\mathcal{O}(1)\cdot\left\{T^l\sum_{n\ne0}\frac{e^{-T}}{n^2}
+\sum_{n\ne0}e^{-2|n|T}+ \sum_{n=1}^{+\infty}\sum_{j=1}^{2m_0}\sum_{k=0}^{m}\left|\mathrm{Res}\left(f^{(l,*)}(u,z)e^{(-1)^{j}i\frac{2n\pi}{h}u^2},u{(j,k)}\right)\right|
\right\}\\
=&\mathcal{O}(1)\cdot\left\{T^le^{-T}+ \sum_{n=1}^{+\infty}\sum_{j=1}^{2m_0}\sum_{k=0}^{m}\left|\mathrm{Res}\left(f^{(l,*)}(u,z)e^{(-1)^{j}i\frac{2n\pi}{h}u^2},u{(j,k)}\right)\right|
\right\}\notag\\
=&\mathcal{O}(1)\cdot\left\{\begin{array}{ll}
T^le^{-T}=T^le^{-\sigma\alpha\sqrt{N_1}}=T^le^{-\pi\eta^{-1}\sqrt{2(2-\beta)N_1\alpha}},&\sigma\le \sigma_{\rm opt},\\
e^{-\pi\eta\sqrt{2(2-\beta)N_1\alpha}},&\sigma> \sigma_{\rm opt},\notag
\end{array}\right.
\end{align}
where we have used the identity $$e^{-\eta^2T}=e^{-\eta\sigma\alpha\sqrt{N_1}\frac{\sigma_{\rm opt}}{\sigma}}=e^{-\pi\eta\sqrt{2(2-\beta)N_1\alpha}},$$
and the fact that $Te^{-T}=\mathcal{O}(e^{-\eta^2T})$ is uniformly bounded for all $T>0$ when $\eta<1$. Furthermore, the implicit constants in the $\mathcal{O}(1)$-notations in \eqref{Poissonsum10} are independent of $z\in S_\beta'$ and $T$.

\bigskip
Finally, combining the asymptotic bounds \eqref{Poissonsum10}, we can estimate the quadrature error in \eqref{errquad0} from Poisson's summation formula \eqref{QuadratureErrorfor_w}.
Note that from the definition of ${\bar f}^{(l)}$ and $N_th\ge (\kappa+1)^2T^2$,
\begin{align*}
&\int_{-\infty}^{+\infty}{\bar f}^{(l)}(u,z) \,\mathrm{d}u=2\int_{0}^{+\infty}{\bar f}^{(l)}(u,z) \,\mathrm{d}u\notag\\
=&2hf^{(l)}(h,z)+2\left(\int_{h}^0+\int_0^{N_th}\int_{N_th}^{+\infty}\right)f^{(l)}(u,z) \,\mathrm{d}u\\
=&2hf^{(l)}(h,z)+2I^{(l)}(z)\\
&+2\left\{\int_{h}^{0}+\int_{N_th}^{+\infty}\right\}\frac{1}{2 \sqrt{u}}\frac{zC^{\alpha}(\sqrt{u}-T)^le^{\sqrt{u}-T}}
{Ce^{\frac{1}{\alpha}(\sqrt{u}-T)}+z}
\left(\prod_{k=1}^{m}\frac{z-s_k}{Ce^{\frac{1}{\alpha}(\sqrt{u}-T)}+s_k}\right)
\,\mathrm{d}u\notag\\
=&2I^{(l)}(z)+2hf^{(l)}(h,z)+2\int_{h}^{0}\frac{1}{2 \sqrt{u}}\frac{zC^{\alpha}(\sqrt{u}-T)^le^{\sqrt{u}-T}}
{Ce^{\frac{1}{\alpha}(\sqrt{u}-T)}+z}
\left(\prod_{k=1}^{m}\frac{z-s_k}{Ce^{\frac{1}{\alpha}(\sqrt{u}-T)}+s_k}\right)
\,\mathrm{d}u\notag\\
&+2\int_{\sqrt{N_th}-T}^{+\infty}\frac{zC^{\alpha}t^l e^t}{Ce^{\frac{1}{\alpha}t}+z}
\left(\prod\limits_{k=1}^{m}\frac{z-s_k}{Ce^{\frac{1}{\alpha}t}+s_k}\right)\,\mathrm{d}t\notag\\
:=&2I^{(l)}(z)+E_{fd}^{(l)}(z),
\end{align*}
where the remainder $E_{fd}^{(l)}(z)$
satisfies, uniformly for $z\in S_\beta'$,
\begin{align}\label{eq:fd}
\left|E_{fd}^{(l)}(z)\right|\le&\frac{3\sqrt{h}\aleph^mT^l C^{\alpha}e^{\sqrt{h}-T}}{{\cal X}(\beta)}
+\frac{2\aleph^m\kappa(\kappa+\kappa T)^le^{-T}}{C^{m+1-\alpha}{\cal X}(\beta)}.
\end{align}
The bound \eqref{eq:fd} follows from the estimates
\begin{align*}
\left|f^{(l)}(h,z)\right|=&\left|\frac{1}{2 \sqrt{h}}\frac{zC^{\alpha}(\sqrt{h}-T)^le^{\sqrt{h}-T}}
{Ce^{\frac{1}{\alpha}(\sqrt{h}-T)}+z}
\left(\prod_{k=1}^{m}\frac{z-s_k}{Ce^{\frac{1}{\alpha}(\sqrt{h}-T)}+s_k}\right)\right|
\le\frac{\mathbb{T}_{m,\beta}T^l C^{\alpha}e^{\sqrt{h}-T}}{2\sqrt{h}\delta^{m}{\cal X}(\beta)},
\end{align*}
as well as the integral bounds
\begin{align*}
&\left|\int_{0}^{h}\frac{1}{2 \sqrt{u}}\frac{zC^{\alpha}(\sqrt{u}-T)^le^{\sqrt{u}-T}}
{Ce^{\frac{1}{\alpha}(\sqrt{u}-T)}+z}
\left(\prod_{k=1}^{m}\frac{z-s_k}{Ce^{\frac{1}{\alpha}(\sqrt{u}-T)}+s_k}\right)
\,\mathrm{d}u\right|\\
\le& \frac{\mathbb{T}_{m,\beta}T^l C^{\alpha}e^{\sqrt{h}-T}}{\delta^{m}{\cal X}(\beta)}\int_0^h\frac{1}{2\sqrt{u}}\,\mathrm{d}u=
\frac{\mathbb{T}_{m,\beta}T^l C^{\alpha}\sqrt{h}e^{\sqrt{h}-T}}{\delta^{m}{\cal X}(\beta)}
\end{align*}
and
\begin{align*}
\left|\int_{\sqrt{N_th}-T}^{+\infty}\frac{zC^{\alpha}t^l e^t}{Ce^{\frac{1}{\alpha}t}+z}
\left(\prod\limits_{k=1}^{m}\frac{z-s_k}{Ce^{\frac{1}{\alpha}t}+s_k}\right)\,\mathrm{d}t\right|\le \int_{\kappa T}^{+\infty}\frac{\mathbb{T}_{m,\beta}t^le^{-\frac{1}{\kappa}t}}{C^{m+1-\alpha}{\cal X}(\beta)}\,\mathrm{d}t\le \frac{\mathbb{T}_{m,\beta}\kappa(\kappa+\kappa T)^le^{-T}}{C^{m+1-\alpha}{\cal X}(\beta)}.
\end{align*}
These estimates are obtained by applying \eqref{ine:minus_u}, \eqref{ieq:positive_u},  and \eqref{eq:inequ_pos}.

Applying Poisson's summation formula \eqref{QuadratureErrorfor_w} now gives
\begin{align}\label{eq:quaderror1}
\left|h\sum_{n\ne0}c_n\right|=&\left|\int_{-\infty}^{+\infty}{\bar f}^{(l)}(u,z) \,\mathrm{d}u
-h\sum_{j=-\infty}^{+\infty}{\bar f}^{(l)}(jh,z)\right|\notag\\
=&\left|E_{fd}^{(l)}(z)+2\left(I^{(l)}(z)-r^{(l)}_{N_t}(z)\right)-2h\sum_{j=N_t+1}^{+\infty}f^{(l)}(jh,z)\right|\\
\ge&2\left|E_Q^{(l)}(z)\right|-\left|E_{fd}^{(l)}(z)\right|-\frac{2\aleph^m\kappa(\kappa+\kappa T)^le^{-T}}{C^{m+1-\alpha}{\cal X}(\beta)},\notag
\end{align}
where we have used the bound
\begin{align*}
2h\sum_{j=N_t+1}^{+\infty}|f^{(l)}(jh,z)|
\le& 2h\sum_{j=N_t+1}^{+\infty}
\left|\frac{1}{2 \sqrt{jh}}\frac{zC^{\alpha}(\sqrt{jh}-T)^le^{\sqrt{jh}-T}}
{Ce^{\frac{1}{\alpha}(\sqrt{jh}-T)}+z}
\left(\prod_{k=1}^{m}\frac{z-s_k}{Ce^{\frac{1}{\alpha}(\sqrt{jh}-T)}+s_k}\right)\right|\\
\le& 2h\sum_{j=N_t+1}^{+\infty}\frac{1}{2 \sqrt{jh}}
\frac{C^{\alpha-m-1}\mathbb{T}_{m,\beta}(\sqrt{jh}-T)^le^{\sqrt{jh}-T}}
{{\cal X}(\beta)e^{\frac{m+1}{\alpha}(\sqrt{jh}-T)}}\notag\\
\le&\frac{2\mathbb{T}_{m,\beta}}
{{\cal X}(\beta)C^{m+1-\alpha}}\int_{(\kappa+1)^2 T^2}^{+\infty}\frac{1}{2 \sqrt{u}}(\sqrt{u}-T)e^{-\frac{1}{\kappa}(\sqrt{u}-T)}\,\mathrm{d}u\notag\\
\le&\frac{2\mathbb{T}_{m,\beta}}
{{\cal X}(\beta)C^{m+1-\alpha}}\int_{\kappa T}^{+\infty}te^{-\frac{1}{\kappa}t}\,\mathrm{d}t
\le \frac{2\aleph^m\kappa(\kappa+\kappa T)^le^{-T}}{C^{m+1-\alpha}{\cal X}(\beta)}
\notag
\end{align*}
which follows from \eqref{eq:est}, \eqref{ieq:positive_u} and \eqref{eq:inequ_pos} together with the fact that $\sqrt{jh}-T>0$ for $j\ge N_t+1$.

Therefore, combining \eqref{Poissonsum10}-\eqref{eq:quaderror1} yields \eqref{errquad0} in the case $\Re(u(1,0))>\sqrt{h}+\alpha$.

\bigskip

\textbf{Case 2: $\Re(u_0(z))=T+\alpha\log\frac{x}{C}\le \sqrt{h}+\alpha$ ($z\in S_\beta\setminus S_\beta'$ and $z\not=0$)}.
We then define ${\bar f}^{(l)}$
  in \eqref{eq:extension_of_fnew11110} by setting $\hbar=\lambda_0h$ (with $\lambda_0$ given by \eqref{eq:ML}),
 as follows
\begin{align*}
{\bar f}^{(l)}(u,z)=
\begin{cases}
f^{(l)}(u,z), & u\ge \hbar,\\
f^{(l)}(h,z), & -\hbar\le u\le \hbar,\\
f^{(l)}(-u,z), & u\le-\hbar.
\end{cases}
\end{align*}
Define $F(v,z)=\sum_{k=-\infty}^{\infty}{\bar f}^{(l)}(kh+v,z),\ v\in[0,h]$; then $F(\cdot,z)$ is h-periodic and uniformly convergent in $v$.

As before, under the condition in \eqref{eq:onT}, the function $f^{(l,*)}(u,z)$
 is holomorphic in the strip
$$\Omega=\left\{u\in\mathbb{C}:\ |\Re(u)|\ge \sqrt{\lambda_0h}, \ |\Im(u)|\le  2m_0\alpha\pi\right\}$$
except for simple poles \eqref{eq:strippoles} at $u(j,k)$ in $\Omega$ for $j=1,2,\ldots, 2m_0$ and $k=1,2,\ldots,m$.

Consequently, from \eqref{eq:fourier_c10}, \eqref{eq:fourier_cpath00} and \eqref{eq:Cauchycon0} applying Cauchy's integral theorem to $f^{(l,*)}(y,z)e^{\pm i\frac{2n\pi}{h}y^2}$
for $n\ge 1$  we obtain an expression analogous to \eqref{eq:fourier_cpath}
\begin{align*}
hc_n=&
\int_{\lambda_0h}^{+\infty} {\bar f}^{(l)}(u,z)e^{-i\frac{2n\pi}{h}u}du+
\int_{\lambda_0h}^{+\infty} {\bar f}^{(l)}(u,z)e^{i\frac{2n\pi}{h}u}du\notag\\
=&\int_{\sqrt{\lambda_0h}-2im_0\alpha\pi}^{+\infty-2im_0\alpha\pi}f^{(l,*)}(y,z)e^{- i\frac{2n\pi}{h}y^2} \,\mathrm{d}y+ \int_{\sqrt{\lambda_0h}+2im_0\alpha\pi}^{+\infty+2im_0\alpha\pi}f^{(l,*)}(y,z)e^{ i\frac{2n\pi}{h}y^2} \,\mathrm{d}y\\
&+i\int_{0}^{2m_0\alpha\pi}\big{[}f^{(l,*)}(\sqrt{\lambda_0h}+it,z)e^{i\frac{2n\pi}{h}(\sqrt{\lambda_0h}+it)^2}\notag\\
&\hspace{1.6cm}-f^{(l,*)}(\sqrt{\lambda_0h}-it,z)e^{-i\frac{2n\pi}{h}(\sqrt{\lambda_0h}-it)^2})\big{]}
\, \mathrm{d}t\notag\\
&-2\pi i\sum_{j=1}^{m_0}\sum_{k=1}^{m}\mathrm{Res}\left[f^{(l,*)}(u,z)e^{-i\frac{2n\pi}{h}u^2},u(2j-1,k)\right]\notag\\
&-2\pi i\sum_{j=1}^{m_0}\sum_{k=1}^{m}\mathrm{Res}\left[f^{(l,*)}(u,z)e^{i\frac{2n\pi}{h}u^2},u(2j,k)\right].\notag
\end{align*}

Clearly, from the proof \eqref{eq:intestimate1} of Lemma \ref{infty}, we see that
\begin{align*}
&\left|\int_{\sqrt{\lambda_0h}\pm 2im_0\alpha\pi}^{+\infty\pm 2im_0\alpha\pi}f^{(l,*)}(y,z)e^{\pm i\frac{2n\pi}{h}y^2} \,\mathrm{d}y\right|\\
&\le\int_{\sqrt{\lambda_0h}}^{+\infty}\frac{|z|C^{\alpha}(|t-T|+2m_0\alpha\pi)^l e^{t-T}e^{-\frac{8m_0n\alpha\pi^2}{h}t}}
{\bigg|Ce^{\frac{1}{\alpha}(t-T)}+z\bigg|}
\cdot\prod_{k=1}^{m}\frac{|z-s_k|}{Ce^{\frac{1}{\alpha}(t-T)}+s_k}\,\mathrm{d}t\notag\\
&\le\int_{\sqrt{h}}^{+\infty}\frac{|z|C^{\alpha}(|t-T|+2m_0\alpha\pi)^l e^{t-T}e^{-\frac{8m_0n\alpha\pi^2}{h}t}}
{\bigg|Ce^{\frac{1}{\alpha}(t-T)}+z\bigg|}
\cdot\prod_{k=1}^{m}\frac{|z-s_k|}{Ce^{\frac{1}{\alpha}(t-T)}+s_k}\,\mathrm{d}t\notag\\
&\le \frac{\aleph^m(T+2m_0\alpha\pi)^l}
{n^2\sigma\alpha{\cal X}(\beta) C^{-\alpha}}e^{-T}+\frac{\kappa(2m_0\alpha\pi+\kappa)^l\aleph^m}
{{\cal X}(\beta)C^{m+1-\alpha}}e^{-2nT},
\end{align*}
and by the proof of Case (ii) of Lemma \ref{residual}
\begin{align*}
&\sum_{n=1}^{+\infty}\sum_{j=1}^{2m_0}\sum_{k=1}^{m}\Big|\mathrm{Res}\left[f^{(l,*)}(u,z)e^{i(-1)^j\frac{2n\pi}{h}u^2},u(j,k)\right]\Big|\\
 \le& \sum_{n=1}^{+\infty}\sum_{j=1}^{2m_0}\sum_{k=0}^{m}\Big|\mathrm{Res}\left[f^{(l,*)}(u,z)e^{i(-1)^j\frac{2n\pi}{h}u^2},u(j,k)\right]\Big|
 =\left\{\begin{array}{ll}
\mathcal{O}(T^le^{-T}),&\sigma\le\sigma_{\rm opt},\\
\mathcal{O}(e^{-\eta^2T}),&\sigma> \sigma_{\rm opt}.\end{array}\right.
\end{align*}

Moreover, utilizing the inequalities $\sqrt{\lambda_0h}\ge \sqrt{h}+2\alpha$, $\Re(u(1,0))\le \sqrt{h}+\alpha$, along with
 \begin{align*}
 \bigg|e^{\frac{1}{\alpha}(\sqrt{\lambda_0h}-it-u(1,0))}-1\bigg|\ge& e^{\frac{1}{\alpha}(\sqrt{\lambda_0h}-\Re(u(1,0)))}-1\ge e-1,\\
 \left|e^{\frac{1}{\alpha}(u(1,k)-\sqrt{\lambda_0h}+it)}-1\right|\ge& e-1, \quad k=1,2,\ldots,m
 \end{align*}
which are implied by the assumption on $T$ \eqref{eq:onT}, namely, $T+\alpha\log{\frac{s_k}{C}})>T+\alpha\log{\frac{1}{C}}>\sqrt{\lambda_0h}+\alpha$, we obtain, by arguing analogously to the proof of \eqref{eq:intestimate3int} with $\sqrt{\lambda_0h}$ in place of $\sqrt{h}$, that
\begin{align*}
\int_0^{2m_0\pi\alpha}\left|f^{(l,*)}(\sqrt{\lambda_0h}+i t,z)e^{i\frac{2n\pi}{h}(\sqrt{\lambda_0h}+i t)^2}-f^{(l,*)}(\sqrt{\lambda_0h}-i t,z)e^{-i\frac{2n\pi}{h}(\sqrt{\lambda_0h}-i t)^2}\right|
\,\mathrm{d}t=\mathcal{O} (T^le^{-T}).
\end{align*}
Hence, the Poisson summation formula holds analogously to \eqref{Poissonsum10}, yielding
\begin{align*}
h\sum_{n\ne 0}|c_n|=\sum_{n\ne0}\bigg|\mathfrak{F}[\bar f^{(l)}\big{(}\frac{ n}{h}\big{)}\bigg|
=\mathcal{O}(1)\cdot\left\{\begin{array}{ll}
T^le^{-T}=T^le^{-\sigma\alpha\sqrt{N_1}},&\sigma\le \sigma_{\rm opt}\\
e^{-\pi\eta\sqrt{2(2-\beta)N_1\alpha}},&\sigma> \sigma_{\rm opt}
\end{array}\right.
\end{align*}
where the constant implied in $\mathcal{O}(1)$ is independent of $z\in S_\beta\setminus S_\beta'$ and $T$.

Finally, the error estimate \eqref{errquad0} follows in the same way from
\begin{align*}
&\int_{-\infty}^{+\infty}{\bar f}^{(l)}(u,z) \,\mathrm{d}u
=2\lambda_0hf^{(l)}(\lambda_0h,z)+2\left(\int_{\lambda_0h}^0+\int_0^{N_th}\int_{N_th}^{+\infty}\right)f^{(l)}(u,z) \,\mathrm{d}u\\
=&2\lambda_0hf^{(l)}(\lambda_0h,z)+2I(z)
+\left\{\int_{\sqrt{\lambda_0h}-T}^{-T}+\int_{\sqrt{N_th}-T}^{+\infty}\right\}\frac{zC^{\alpha}t^l e^t}{Ce^{\frac{1}{\alpha}t}+z}
\left(\prod\limits_{k=1}^{m}\frac{z-s_k}{Ce^{\frac{1}{\alpha}t}+s_k}\right)\mathrm{d}t.
\end{align*}

These two cases together complete the proof of \eqref{errquad0}.
\end{proof}

\begin{remark} (i) The optimal pole clustering parameter $\sigma_{\rm opt}=\frac{\pi\sqrt{2(2-\beta)}}{\sqrt{\alpha}}$ can alternatively be derived from the quadrature error $E_Q^{(l)}$ for $l=0$, given the truncated error $E_T^{(0)}=\mathcal{O}(e^{-T})$. By expressing the residual sum in the Poisson summation formula \eqref{Poissonsum10} in terms of $\eta^2=\frac{2\pi^2\alpha(2-\beta)}{h}$ and $h=\sigma^2\alpha^2$ (as defined in \eqref{asy0000} and \eqref{asy0k}), we obtain
\begin{align}\label{eq:optp}
&h\sum_{n\ne 0}|c_n|=\sum_{n\ne0}\bigg|\mathfrak{F}[{\bar f}^{(l)}]\big{(}\frac{n}{h}\big{)}\bigg|\notag\\
=&\mathcal{O}(1)\cdot\left\{e^{-T}+ \sum_{n=1}^{+\infty}\sum_{j=1}^{2m_0}\sum_{k=0}^{m}\left|\mathrm{Res}\left(f^{(l,*)}(u,z)e^{(-1)^{j}i\frac{2n\pi}{h}u^2},u{(j,k)}\right)\right|
\right\}\\
=&\mathcal{O}(1)\cdot\left\{e^{-T}+ \max_{x\in [x^*,1]} x^{\alpha}e^{-\frac{2\pi^2\alpha(2-\beta)}{h} (T+\alpha\log{\frac{x}{C}})}
+e^{-\frac{2}{2-\beta}\frac{2\pi^2\alpha(2-\beta)T}{h}}
\right\}.\notag
\end{align}
Together with the estimate in \eqref{eq:balance}, this gives
\begin{align*}
\max_{x\in [x^*,1]} x^{\alpha}e^{-\frac{2\pi^2\alpha(2-\beta)}{h} (T+\alpha\log{\frac{x}{C}})}=&\mathcal{O}(1)\cdot\max\left\{e^{-T},e^{-\frac{2\pi^2\alpha(2-\beta)T}{h}}\right\}\\
=&\mathcal{O}(1)\cdot\max\left\{e^{-\sigma\alpha\sqrt{N_1}},e^{-\frac{2\pi^2(2-\beta)\sqrt{N_1}}{\sigma}}\right\}.
\end{align*}
Minimizing the upper bound in \eqref{eq:optp} by balancing the dominant exponential terms yields $\sigma_{\rm opt}=\frac{\pi\sqrt{2(2-\beta)}}{\sqrt{\alpha}}$.

(ii) Moreover, the proof of Lemma \ref{residual} (specifically, the upper bounds in the final inequality of \eqref{eq:res1} and the penultimate line of \eqref{asy0000}) demonstrates that the error bound $E_Q^{(l)}$ is attainable on the V-shaped domain $V_\beta$. Consequently, $\sigma_{\rm opt}$  serves as the optimal pole clustering parameter for the LP approximation \eqref{eq:rat} on  $V_\beta$ for $0\le \beta<2$.
\end{remark}

\section{Proofs of Theorems \ref{mainthm} and \ref{mainthm2} and their generalizations}\label{sec:4}

\begin{proof}[Proof of Theorem \ref{mainthm}]
The conclusion of Theorem \ref{mainthm} follows directly from inequalities \eqref{ratappforzalpha} and \eqref{boundforEbar}, combined with the result of Theorem \ref{eq:quaderrthm}. We obtain the following estimate
\begin{align}\label{rat01}
|E(z)| &= |z^\alpha-r_N(z)| \notag \\
&\le \frac{2|\sin(\alpha\pi)|\aleph^mC^{\alpha}}{\alpha\pi{\cal X}(\beta)e^T}
+ \frac{|\sin(\alpha\pi)|\left[2\aleph^m+3\aleph^m\right]}{(1-(\alpha))\pi C^{1-\alpha}{\cal X}(\beta)e^{T}}
+ \frac{|\sin(\alpha\pi)|}{\alpha\pi}\left|E_Q^{(0)}(z)\right| \\
&= \max\left\{\frac{|\sin(\alpha\pi)|}{\alpha\pi},\frac{|\sin(\alpha\pi)|}{(1-(\alpha))\pi}\right\}
\left\{\begin{array}{ll}
\mathcal{O}(e^{-\sigma\alpha\sqrt{N_1}}), & \sigma\le \sigma_{\rm opt}, \\
\mathcal{O}(e^{-\pi\eta\sqrt{2(2-\beta)N_1\alpha}}, & \sigma> \sigma_{\rm opt},
\end{array}\right. \notag \\
&= \max\left\{\frac{|\sin(\alpha\pi)|}{\alpha\pi},\frac{|\sin(\alpha\pi)|}{(1-(\alpha))\pi}\right\}
\left\{\begin{array}{ll}
\mathcal{O}(e^{-\sigma\alpha\sqrt{N}}), & \sigma\le \sigma_{\rm opt}, \\
\mathcal{O}(e^{-\pi\eta\sqrt{2(2-\beta)N\alpha}}), & \sigma> \sigma_{\rm opt}.
\end{array}\right. \notag
\end{align}
In the final step of \eqref{rat01}, we utilized the fact that
\begin{align*}
\sqrt{N_1} = \sqrt{N-N_2} = \sqrt{N}\left[1+\mathcal{O}\left(\frac{N_2}{N}\right)\right] = \sqrt{N}+\mathcal{O}(1) =: \sqrt{N}+c_N,
\end{align*}
where $c_N < 0$ is uniformly bounded and independent of $N$.
\end{proof}

\begin{proof}[Proof of Theorem \ref{mainthm2}]
Analogously, from \eqref{ratappforzalphalog} and \eqref{boundforwidetildeE} together with Theorem \ref{eq:quaderrthm}, taking $m=\lceil\alpha\rceil$ gives
\begin{align}\label{rat01log}
|\widetilde{E}(z)|=&|z^\alpha\log z-\widetilde{r}_N(z)|\notag\\
\le&\frac{|\sin(\alpha\pi)|}{\alpha^2\pi}\bigg[
\frac{2(1+T)\aleph^mC^{\alpha}}{{\cal X}(\beta)e^{T}}
\left(1+\frac{\kappa^2}{C^{m+1}}\right)
+\frac{(3\kappa^2 T+ T)\aleph^{m}}{C^{m+1-\alpha}{\cal X}(\beta)}e^{-T}+\left|E_Q^{(1)}(z)\right|\bigg]\\
&+\left|\frac{\sin(\alpha\pi)\log{C}}{(-1)^m\alpha\pi}
+\frac{\cos(\alpha\pi)}{(-1)^{m}\alpha}\right|
\bigg[
\frac{2\aleph^mC^{\alpha}}{{\cal X}(\beta)e^T}\left(1+\frac{\kappa }{C^{m+1}}\right)+\frac{3\kappa\aleph^mC^{\alpha-m-1}}{e^{T}{\cal X}(\beta)}
+\left|E_Q^{(0)}(z)\right|\bigg]\notag\\
=&\left\{\begin{array}{ll}
\mathcal{O}\sqrt{N_1}e^{-\sigma\alpha\sqrt{N_1}}),&\sigma\le \sigma_{\rm opt},\\
\mathcal{O}(e^{-\pi\eta\sqrt{2(2-\beta)N_1\alpha}},&\sigma> \sigma_{\rm opt},
\end{array}\right.\notag\\
=&\left\{\begin{array}{ll}
\mathcal{O}(\sqrt{N}e^{-\sigma\alpha\sqrt{N}}),&\sigma\le \sigma_{\rm opt},\\
\mathcal{O}(e^{-\pi\eta\sqrt{2(2-\beta)N\alpha}},&\sigma> \sigma_{\rm opt}.\notag
\end{array}\right.
\end{align}
In the derivation of the second equality of \eqref{rat01log} for the case $\sigma> \sigma_{\rm opt}$, we used the identity
$$Te^{-T}=Te^{(\eta^2-1)T}e^{-\eta^2T}=\mathcal{O}(e^{-\eta^2T}),$$ which holds uniformly for all  $T\ge 0$ since $\eta<1$.

In particular, when $\alpha$ is a positive integer, the first inequality in \eqref{rat01log} yields
\begin{align*}
\left|\widetilde{E}(z)\right|
\le&\frac{|\cos(\alpha\pi)|}{\alpha}
\bigg[
\frac{2\aleph^mC^{\alpha}}{{\cal X}(\beta)e^T}\left(1+\frac{\kappa }{C^{m+1}}\right)
+\frac{3\kappa\aleph^me^{-T}}{C^{m+1-\alpha}{\cal X}(\beta)}
+\left|E_Q^{(0)}(z)\right|\bigg],
\end{align*}
which implies \eqref{eq: rate2}. This completes the proof of Theorem \ref{mainthm2}.
\end{proof}

\begin{theorem}\label{mainthmG}
Let $\alpha$, $\sigma>0$ and set $\sigma_{\rm opt}=\frac{\pi\sqrt{2(2-\beta)}}{\sqrt{\alpha}}$ and
$\eta=\frac{\sigma_{\rm opt}}{\sigma}$. If $g(z)$ is analytic in a neighborhood of $S_\beta$, then there exist coefficients $\{a^{(g)}_j\}_{j=1}^{N_1}$, $\{\widetilde{a}^{(g)}_j\}_{j=1}^{N_1}$ and polynomials $P^{(g)}_{N_2}$, $\widetilde{P}^{(g)}_{N_2}$ with degree
$N_2 = \mathcal{O}(\sqrt{N_1})=\mathcal{O}(\sqrt{N})$ such that the  LP approximations of the form \eqref{eq:rat} furnished with the poles \eqref{eq:tapered2}
satisfy the following uniform error estimates on $S_\beta$ as $N \to\infty$:
\allowdisplaybreaks
\begin{align}
\|r^{(g)}_N(z)- g(z)z^\alpha\|_{C(S_\beta)}=
\left\{\begin{array}{ll}
\mathcal{O}\left(e^{-\sigma\alpha\sqrt{N}}\right),&\sigma\le \sigma_{\rm opt},\\
\mathcal{O}\left(e^{-\pi\eta\sqrt{2(2-\beta)N\alpha}}\right),&\sigma> \sigma_{\rm opt},
\end{array}\right.\label{eq: rateG1}\\
\|\widetilde{r}^{(g)}_N(z)- g(z)z^\alpha\log{z}\|_{C(S_\beta)}=
\left\{\begin{array}{ll}
\mathcal{O}\left(\sqrt{N}e^{-\sigma\alpha\sqrt{N}}\right),&\sigma\le \sigma_{\rm opt},\\
\mathcal{O}\left(e^{-\pi\eta\sqrt{2(2-\beta)N\alpha}}\right),&\sigma> \sigma_{\rm opt},
\end{array}\right.\label{eq: rateG2}
\end{align}
and, when  $\alpha$ is  a positive integer,
\begin{equation}\label{eq: rateG22}
\|\widetilde{r}_N^{(g)}(z)-z^\alpha\log{z}\|_{C(S_\beta)}=\left\{\begin{array}{ll}
\mathcal{O}(e^{-\sigma\alpha\sqrt{N}}),&\sigma\le \sigma_{\rm opt},\\
\mathcal{O}(e^{-\pi\eta\sqrt{2(2-\beta)N\alpha}}),&\sigma> \sigma_{\rm opt}.
\end{array}\right.
\end{equation}
\end{theorem}
\begin{proof}
By Runge's approximation theorem (see Subsection \ref{subsec22}), there exists a polynomial $P^{(g)}_{N_2}(z)$ of degree $N_2=\mathcal{O}(\sqrt{N_1})$
such that
$$
\|g-P^{(g)}_{N_2}\|_{C(S_{\beta})}=\mathcal{O}(e^{-T}).
$$
Moreover, \eqref{rat01} and \eqref{rat01log} imply that both  $r_N$ and $\widetilde{r}_N$ in Theorems \ref{mainthm} and Theorem \ref{mainthm2}
are uniformly bounded on $S_\beta$. Hence the products
$$
P^{(g)}_{N_2}(z)r_N(z):=r^{(g)}_N(z), \quad\quad P^{(g)}_{N_2}(z)\widetilde{r}_N(z):=\widetilde{r}_N^{(g)}(z)
$$
are again LP approximations of the form \eqref{eq:rat}. From
\begin{align*}
\left|g(z)z^\alpha-P^{(g)}_{N_2}(z)r_N(z)\right|
&\le\|g(z)(z^\alpha-r_N(z))\|_{C(S_{\beta})}+|r_{N}(z)|\mathcal{O}(e^{-T}),\\
|g(z)z^\alpha\log{z}-P^{(g)}_{N_2}(z)\widetilde{r}_N(z)|
&\le\|g(z)(z^\alpha\log{z}-\widetilde{r}_N(z))\|_{C(S_{\beta})}+|\widetilde{r}_{N}(z)|\mathcal{O}(e^{-T}),
\end{align*}
and the bounds in Theorem \ref{mainthm} and Theorem  \ref{mainthm2}, we obtain the estimates \eqref{eq: rateG1}--\eqref{eq: rateG22}, which completes the proof.
\end{proof}

\section{Applied to corner singularities}\label{sec:5}
Suppose the corner domain $\Omega$  is determined by vertices $w_1,\ldots,w_\ell$. 
With the aid of the decomposition for Cauchy integrals in Gopal and Trefethen \cite[Theorem 2.3]{Gopal2019},  Theorem \ref{mainthm} and Theorem \ref{mainthm2} can be extended to the case in
which the domain $\Omega$ is a straight or curvy polygon with each internal angle $<2\pi$.

\begin{definition}\label{def}
A domain $\Omega$ is said to satisfy generalized convexity if, for each corner point $w_k$ ($k=1,2,\ldots,\ell$), the two tangent rays at $w_k$ lie entirely outside $\Omega$.
\end{definition}

From Definition \ref{def}, it follows that both any convex polygon and the domain depicted in Fig. \ref{tangent_covering_domain} possess generalized convexity. Specifically, for each vertex $w_k$, there exists a minimal sector $S_{\beta_k}$ centered at $w_k$  with interior angle $\varphi_k\pi(=\beta_k\pi)$ such that  $\Omega\subset S_{\beta_k}$. Subsequently, by applying the rigorous decomposition method presented in \cite[the proof of Theorem 2.3]{Gopal2019}, we express the function $f(z)$ as a superposition of
$2\ell$ Cauchy-type contour integrals
\begin{align}\label{decompose_singularity}
f(z)=\frac{1}{2\pi i}\sum_{k=1}^\ell\int_{\Lambda_k}\frac{f(\zeta)}{\zeta-z}\mathrm{d}\zeta
+\frac{1}{2\pi i}\sum_{k=1}^\ell\int_{\Gamma_k}\frac{f(\zeta)}{\zeta-z}\mathrm{d}\zeta
=:\sum_{k=1}^\ell f_k(z)+\sum_{k=1}^\ell g_k(z),
\end{align}
where  $\Lambda_k$ consists of the two sides of an exterior bisector at $w_k$, and $\Gamma_k$ connects the end of the slit contour at vertex $w_k$ to the beginning of the slit contour at vertex $w_{k+1}$ (denote $w_{\ell+1}=w_1$). Additionally, 
each $g_k$ is holomorphic in a larger domain $\mathbb{C}\setminus\Gamma_k$ including $\Omega$, and $f_k$ holomorphic in a slit-disk region $\mathbb{C}\setminus\Gamma_k$ around $w_k$ with the slit line $\Lambda_k$, $k=1,\cdots,\ell$. For further details, please refer to Gopal and Trefethen \cite{Gopal2019}, with Fig. \ref{decompose_path} illustrating an example.

\begin{figure}[htbp]
\centerline{\includegraphics[width=6cm]{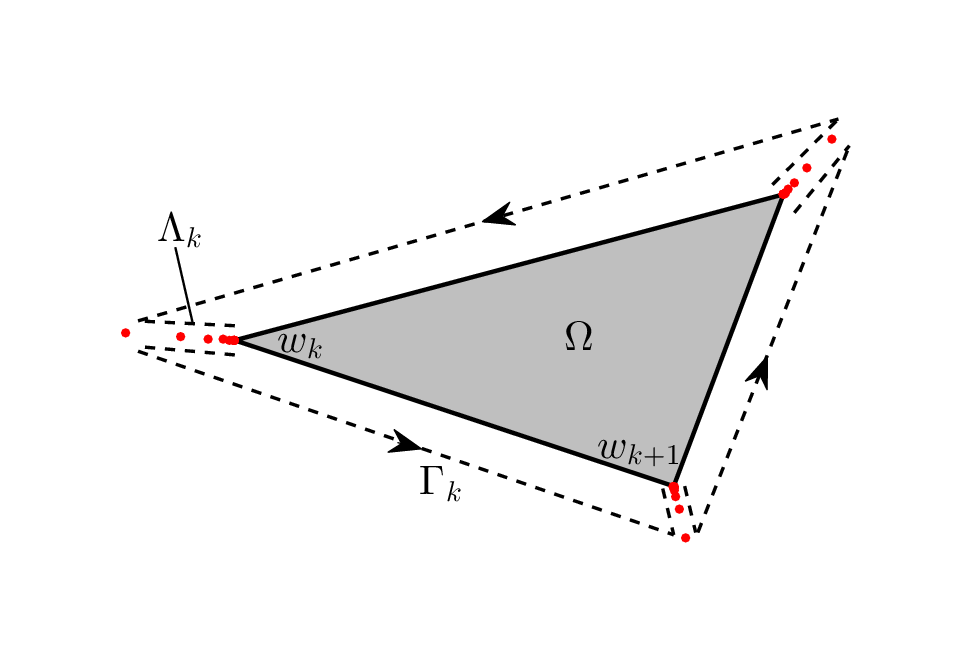}}
\caption{This figure is cited from \cite[{\sc Fig. 3}]{Gopal2019}: A holomorphic function $f(z)$ defined in the corner domain $\Omega$ is decomposed as the sum of $2\ell$ Cauchy-type integrals: $\sum_{k=1}^\ell f_k(z)+\sum_{k=1}^\ell g_k(z)$, with $f_k(z)=\frac{1}{2\pi i}\int_{\Lambda_k}\frac{f(\zeta)}{\zeta-z}d\zeta$ along the two sides of an exterior bisector slit to each corner, and $g_k(z)=\frac{1}{2\pi i}\int_{\Gamma_k}\frac{f(\zeta)}{\zeta-z}d\zeta$ along each line segment connecting the ends of those slit contours.}
\label{decompose_path}
\end{figure}

Suppose $f$ is defined on $\Omega$ satisfying  generalized convexity,  with isolated branch points at the vertices $\{w_k\}_{k=1}^{\ell}$, then $f$ can be efficiently approximated by
an LP approximation $r_n(z)$.
The approximation employs lightning poles $\{p_{k,j}\}_{j=1}^{N_{1,k}}$ that are tapered exponentially clustered with parameter $\sigma_k$ toward each corner $w_k$ of $\mathcal{S}_{\beta_k}$ along the exterior bisector (see Fig. \ref{tangent_covering_domain}), and is given in Gopal and Trefethen \cite{Gopal2019} by
\begin{align}\label{LP_cornerdomain}
r_n(z)=\sum_{k=1}^\ell\sum_{j=1}^{N_{1,k}}\frac{a_{k,j}}{z-p_{k,j}}+\sum_{j=0}^{N_2} b_{j}z^j.
\end{align}

\begin{theorem}\label{corner}
Let $\Omega$ be a generalized-convex domain (straight or curved) with corner  points $\{w_k\}_{k=1}^{\ell}$. For each $w_k$, let $\mathcal{S}_{\beta_k}$ be the smallest sector of angle $\beta_k$, aligned with the local tangent rays, such that $\Omega\subseteq \mathcal{S}_{\beta_k}$.
Assume $f$
 is analytic in a neighborhood of $\Omega$ except at the $w_k$ (defined on the slit plane along the exterior-angle bisectors), and admits the decomposition \eqref{decompose_singularity} with
 \begin{align}\label{DomainLP}
f_k(z)=& (z -w_k)^{\alpha_k}h_k(z)+\phi_k(z), \hspace{2.5cm} k=1,\ldots,K_1,\notag\\
f_k(z)=& (z -w_k)^{\alpha_k}\log (z-w_k)h_k(z)+\phi_k(z),\quad k=K_1+1,\ldots,\ell,
\end{align}
where each $\alpha_k>0$ and the functions $h_k(z)$, $\phi_k(z)$
are analytic in a neighbourhood of $\Omega$. Then there exists an LP approximant $r_n(z)$ of the form \eqref{LP_cornerdomain} where the lightning poles $\{p_{k,j}\}_{j=1}^{N_1}$ are determined by the parameters
$$\sigma_k=\frac{\pi\sqrt{2(2-\beta_k)}}{\sqrt{\alpha_k}} \quad\quad (k=1,2,\ldots,\ell),$$
$N_1$, $N_2$ satisfy  $N_2 = \mathcal{O}(\sqrt{N_1})=\mathcal{O}(\sqrt{N})$ with $N=N_1+N_2$, and  $n=\ell N_1+N_2$. For such an approximation the following uniform error estimates hold on $\Omega$ as $N\rightarrow \infty$:
\begin{equation}\label{eq: corrate0}
\|r_n-f\|_{C(\Omega)}=\max\left\{\begin{array}{ll}
\mathcal{O}\left(e^{-\pi\sqrt{2(2-\beta_k)N\alpha_k}}\right),&k=1,\ldots,K_1,\\
\mathcal{O}\left(\sqrt{N}e^{-\pi\sqrt{2(2-\beta_k)N\alpha_k}}\right),&k=K_1+1,\ldots,\ell
\end{array}\right\}.
\end{equation}

 In particular, set
 \begin{align}\label{eq:clusteropt}
 \alpha=\min_{1\le k\le \ell}\alpha_k,\quad\quad  \beta=\max_{1\le k\le \ell}\beta_k, \quad\quad\sigma=\frac{\pi\sqrt{2(2-\beta)}}{\sqrt{\alpha}},
 \end{align}
 and define
 $$\alpha'=\min_{1\le k\le K_1}\alpha_k,\quad\quad
 \alpha''=\min_{K_1+1\le k\le \ell}\alpha_k.$$
 Then
 \begin{equation}\label{eq: corrate1}
\|r_n-f\|_{C(\Omega)}=\left\{\begin{array}{ll}
\mathcal{O}\left(\sqrt{N}e^{-\pi\sqrt{2(2-\beta)N\alpha}}\right),&\alpha'\ge\alpha'',\\
\mathcal{O}\left(e^{-\pi\sqrt{2(2-\beta)N\alpha}}\right),&\mbox{otherwise}.\end{array}\right.
\end{equation}
When $\alpha''$
is a positive integer, the factor $\sqrt{N}$ in \eqref{eq: corrate1} can be omitted.
\end{theorem}
\begin{proof}
%
%
%
{\bf Case (i)}: For $k=1,\ldots,K_1$, the function $(z-w_k)^{\alpha_k}h_k(z)$ with $h_k(z)$ is analytic in a neighborhood of $\Omega$.
By taking $\sigma_{k}=\frac{\sqrt{2(2-\beta_k)}\pi}{\sqrt{\alpha_k}}$, we obtain an LP approximation
\begin{equation*}
r_{N,k}(z)=\sum_{j=1}^{N_{1}}\frac{a_{k,j}}{z-p_{k,j}}+\sum_{j=0}^{N_2} b_{k,j}z^j
\end{equation*}
which satisfies, uniformly for $z\in\Omega$,
\begin{equation}\label{LPwk}
\Big|r_{N,k}(z)-(z-w_k)^{\alpha_k}h_k(z))\Big|=\mathcal{O}\left(e^{-\pi\sqrt{2(2-\beta_k)N\alpha_k}}\right),\quad k=1,2,\ldots,K_1.
\end{equation}
This follows analogously to the proof of Theorem \ref{mainthmG}, under the restriction that $h_k$ is defined on $\Omega$.

{\bf Case (ii)}: For $k=K_1+1,\ldots,\ell$,  we analogously obtain, uniformly for $z\in \Omega$,
\begin{equation*}
\Big|r_{N,k}(z)-(z-w_k)^{\alpha_k}\log(z-w_k)h_k(z)\Big|
=\mathcal{O}\left(\sqrt{N}e^{-\pi\sqrt{2(2-\beta_k)N\alpha_k}}\right).
\end{equation*}
In particular, if $\alpha_k$
is a positive integer, the factor $\sqrt{N}$ in the above estimate  can be omitted.

Moreover, by Runge's approximation theorem, the smooth part $\sum_{k=1}^\ell(g_k(z)+\phi_k(z))$ can be approximated uniformly on $\Omega$ by a polynomial $\mathcal{T}_{N_2}(z)$ of degree $N_2=\mathcal{O}(\sqrt{N_1})$ with root-exponential accuracy
\begin{align}\label{polynomialT2z}
\bigg|\mathcal{T}_{N_2}(z)-\sum_{k=1}^\ell(g_k(z)+\phi_k(z))\bigg|
= \mathcal{O}\big(e^{-\pi\min_{1\le k\le \ell}{\sqrt{2(2-\beta_{k})N\alpha_k}}}\big),\ \ z\in\Omega.
\end{align}
Define the global LP approximant
\begin{equation*}
r_n(z)=\sum_{k=1}^{\ell}\sum_{j=1}^{N_{1}}\frac{a_{k,j}}{z-p_{k,j}}+\sum_{k=1}^{\ell}\sum_{j=0}^{N_2} b_{k,j}z^j+\mathcal{T}_{N_2}(z):=\sum_{k=1}^{\ell}\sum_{j=1}^{N_{1}}\frac{a_{k,j}}{z-p_{k,j}}+\sum_{j=0}^{N_2} b_{j}z^j,
\end{equation*}
where $n=\ell N_1+N_2$ and $N_2 = \mathcal{O}(\sqrt{N_1})=\mathcal{O}(\sqrt{N})$.

By combining the estimates \eqref{LPwk}, \eqref{eq:b72}, and \eqref{polynomialT2z}, we obtain, uniformly on $\Omega$ as $N\rightarrow \infty$,
\begin{equation*}
|r_n(z)-f(z)|=\max\left\{\begin{array}{ll}
\mathcal{O}\left(e^{-\pi\sqrt{2(2-\beta_k)N\alpha_k}}\right),&k=1,\ldots,K_1,\\
\mathcal{O}\left(\sqrt{N}e^{-\pi\sqrt{2(2-\beta_k)N\alpha_k}}\right),&k=K_1+1,\ldots,\ell.
\end{array}\right\}
\end{equation*}
Furthermore, when $\alpha_k$
is a positive integer for $k>K_1$, the factor $\sqrt{N}$ corresponding to $\alpha_k$ in \eqref{eq: corrate0}  can be omitted.

\bigskip
In particular, from Theorem \ref{mainthmG}, the functions $(z-w_k)^{\alpha_k}h_k(z)$ can be approximated by $r_{N,k}(z)$ with the unified  parameter $\sigma=\frac{\pi\sqrt{2(2-\beta)}}{\sqrt{\alpha}}$. The error satisfies
\begin{equation*}
\Big|r_{N,k}(z)-(z-w_k)^{\alpha_k}h_k(z)\Big|=\left\{\begin{array}{ll}
\mathcal{O}\big(e^{-\sigma\alpha_k\sqrt{N}}\big),&\sigma\le \sigma^{(k)}_{\mathrm{opt}},\\
\mathcal{O}\big(e^{-\pi\eta_k\sqrt{2(2-\beta_k)N\alpha_k}}\big),&\sigma> \sigma^{(k)}_{\mathrm{opt}},
\end{array}\right.\quad \eta_k:=\frac{\sigma^{(k)}_{\mathrm{opt}}}{\sigma},
\end{equation*}
and is bounded by
 $\mathcal{O}\big(e^{-\sigma\alpha\sqrt{N}}\big)
=\mathcal{O}\big(e^{-\pi\sqrt{2(2-\beta)N\alpha}}\big)$ for all $k=1,2,\ldots,K_1$. Indeed, if $\sigma\le \sigma_{\rm opt}^{(k)}$ then
$$\sqrt{2(2-\beta)N\alpha}= \sigma\alpha\le \sigma \alpha_k.$$ While for $\sigma> \sigma_{\rm opt}^{(k)}$ we have
$$
\eta_k\sqrt{2(2-\beta_k)N\alpha_k}=\sqrt{\frac{2(2-\beta_k)^2}{2-\beta}N\alpha}\ge \sqrt{2(2-\beta)N\alpha}.$$
Similarly,
\begin{equation*}
\Big|r_{N,k}(z)-(z-w_k)^{\alpha_k}\log(z-w_k)h_k(z)\Big|=\left\{\begin{array}{ll}
\mathcal{O}(\sqrt{N}e^{-\sigma\alpha_k\sqrt{N}}),&\sigma\le \sigma^{(k)}_{\mathrm{opt}},\\
\mathcal{O}(e^{-\pi\eta_k\sqrt{2(2-\beta_k)N\alpha_k}}),&\sigma> \sigma^{(k)}_{\mathrm{opt}},
\end{array}\right.
\end{equation*}
In addition, if $\alpha_k$
is a positive integer for $k>K_1$, the factor $\sqrt{N}$  in the above estimate can be omitted.

In terms of the parameters $\alpha',\,\alpha''$
introduced earlier, this bound can be written as
\begin{align}\label{eq:b72}
\Big|r_{N,k}(z)-(z-w_k)^{\alpha_k}\log(z-w_k)h_k(z)\Big|=\left\{\begin{array}{ll}
\mathcal{O}\big(\sqrt{N}e^{-\sigma\alpha\sqrt{N}}\big),&\alpha'\ge\alpha'',\\
\mathcal{O}\big(e^{-\sigma\alpha\sqrt{N}}\big),&\alpha'<\alpha''.\end{array}\right.
\end{align}
When $\alpha''$
is a positive integer, the factor $\sqrt{N}$ in \eqref{eq:b72} can be omitted.

Together with the polynomial approximation estimate \eqref{polynomialT2z} and the observation that
$$e^{-\pi\min_{1\le k\le \ell}{\sqrt{2(2-\beta_{k})N\alpha_k}}}\le e^{-\sigma\alpha\sqrt{N}},$$
we obtain the convergence rate \eqref{eq: corrate1} stated in Theorem \ref{corner}.
\end{proof}

\begin{example}
Finally, based on the Osgood-Carath\'{e}odory Theorem \cite[p. 346, Theorem 16.3a]{Henrici2}, we construct a conformal mapping from $\Omega$ to the unit disk by solving the Laplace equation
\begin{align}
\left\{\begin{array}{ll}
\Delta u=0,&z\in\Omega,\\
u(z)=-\log|z|,&z\in\partial\Omega.
\end{array}\right.\label{eq:Laplaceu}
\end{align}
This problem exhibits corner singularities on a generalized-convex, leaf-shaped curvilinear polygon domain (see Fig. \ref{Lalace_complexDomain}, middle). For this domain, the parameters $\beta_k$ ($k=1,2,3,4$) are all equal to $\frac{3}{2}$.

The relationship between the corner singularity exponent $\alpha$ and the corner angle $\beta\pi$
for the Laplace equation is well characterized in the literature \cite{Herremans2023, Wasow}. According to Theorem 5 in \cite{Wasow}, the dominant asymptotic behavior near a corner is governed by $\mathcal{O}(z^{\frac{1}{\beta}})$ when $\frac{1}{\beta}$  is non-integer, and by   $\mathcal{O}(z^{\frac{1}{\beta}}\log z)$ when $\frac{1}{\beta}$ is an integer, for $\beta\in (0,2)$. Consequently, we set
$\alpha=\min_{1\le k\le \ell}\alpha_k=\min_{1\le k\le \ell}\frac{1}{\beta_k}=\frac{2}{3}$, and employ the LP approximation $r_n$
of the form \eqref{LP_cornerdomain} for the solution \eqref{eq:Laplaceu}. Here, $r_n$ utilizes the exponentially clustered tapering value
$\sigma_{\rm opt}=\frac{\pi\sqrt{2(2-\beta)}}{\sqrt{\alpha}}$.
  Additionally, following the {\sc Matlab} function \texttt{laplace} \cite{Treweb}, we set $N_2={\rm ceil}\left(1.3\sum_{k=1}^{\ell}\sqrt{N_{1,k}}\right)={\rm ceil}\left(1.3\ell\sqrt{N_{1}}\right)$.

 \bigskip
 The numerical results presented in Fig. \ref{Lalace_complexDomain} illustrate the convergence rate of
 $r_n$ in comparison with that of $\overline{r}_n(z)$,  which is evaluated using the uniformly exponentially clustered alternative
$\sigma=\frac{\pi\sqrt{2-\beta}}{\sqrt{\alpha}}$.

As shown in Fig. \ref{Lalace_complexDomain},  the pointwise errors are plotted for $r_n(z)$ and $\overline{r}_n(z)$
 at their respective maximum allowed degrees. The black and red points represent the errors evaluated on the original and finer test sample points, respectively. Following the methodology outlined in \cite{Gopal2019}, the terminal red square ``{\color{red}$\square$}'' on the convergence curve denotes the boundary error measurement benchmarked against a refined computational grid-specifically, a mesh with double the resolution of that employed in the original least-squares discretization. Significantly, congruence between the red square and terminal error data points serves as a critical validation metric; such alignment demonstrates numerical fidelity through successful convergence-verification and grid-independence attainment. Here, the ``\#poles'' are the number of poles clustered at every vertex and the ``angle on boundary w.r.t. $w_k$'' is measured by the direction angle centered at the black dots in the middle subplots.

\begin{figure}[htbp]
\centerline{\includegraphics[width=14.8cm]{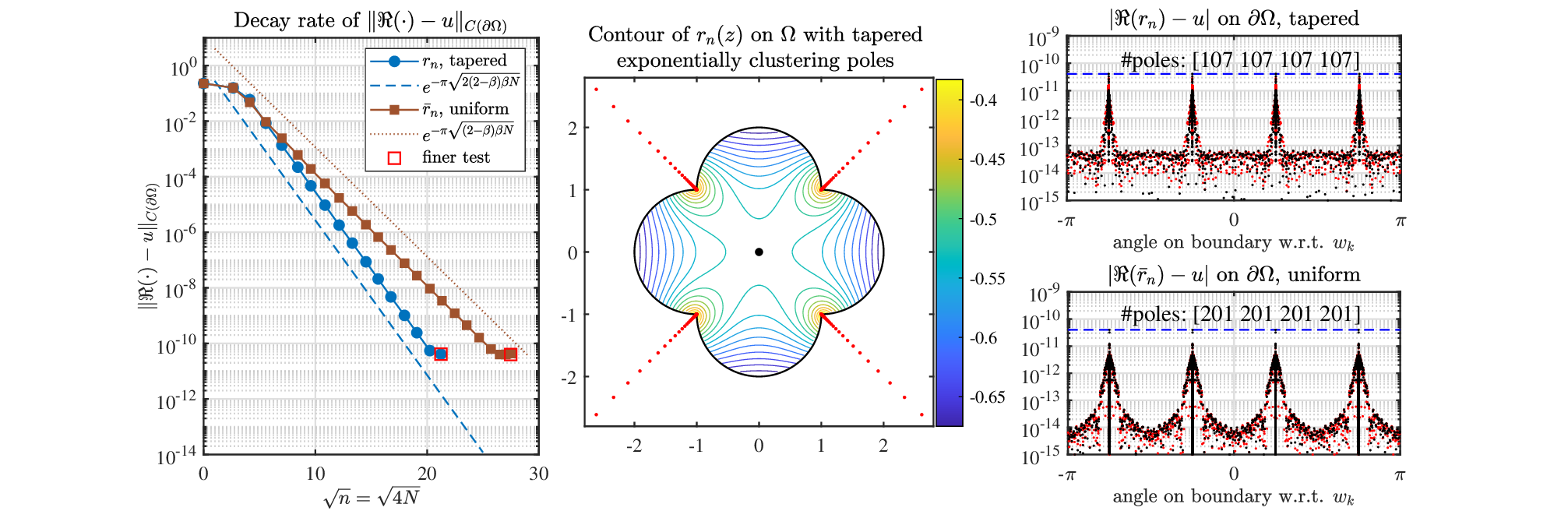}}
\caption{Error decay rates (left), contours with pole distributions (middle), and pointwise errors (right) for the LP solutions $r_n(z)$ and $\overline{r}_n(z)$ to the Laplace equation on a domain $\Omega$. Here, $r_n$ uses the exponentially clustered tapering value $\sigma_{\rm opt}=\frac{\pi\sqrt{2(2-\beta)}}{\sqrt{\alpha}}$, which is compared with $\overline{r}_n(z)$ using the uniformly exponentially clustered alternative $\sigma=\frac{\pi\sqrt{2-\beta}}{\sqrt{\alpha}}$. }\label{Lalace_complexDomain}
\end{figure}
\end{example}

\section{Conclusions}
\label{conclusion}

Building on a rigorous analysis of the root-exponential convergence of lightning plus polynomial approximations for corner singularities, this paper establishes the optimal convergence rate. Specifically, we consider the approximation of functions of the form $g(z)z^\alpha$ or $g(z)z^\alpha\log z$
in a sector-shaped domain with tapered exponentially clustered poles, where $g(z)$ is analytic on the domain and satisfies $g(0)\not=0$. By employing Poisson's summation formula, Runge's approximation theorem, and Cauchy's integral theorem, we confirm Conjecture 3.1 and Conjecture 5.3 in \cite{Herremans2023} regarding the root-exponential convergence rate. Our analysis further reveals that the parameter choice $\sigma_{\rm opt}=\frac{\pi\sqrt{2(2-\beta)}}{\sqrt{\alpha}}$  achieves the fastest convergence rate among all $\sigma>0$.

Furthermore, by leveraging the decomposition framework of Gopal and Trefethen \cite{Gopal2019}, we validate the root-exponential convergence rates for lightning plus polynomial schemes in corner domains $\Omega$ of generalized convexity. The study also explicitly determines the optimal lightning clustering parameter $\sigma$ for
 $\Omega$. The comprehensive analysis presented herein provides a rigorous theoretical framework for lightning-based approximation schemes.

\section*{Acknowledgement}
The first author is grateful to Prof. Nick Trefethen for his encouragement and helpful discussions at ICIAM 2023 Tokyo. 


\phantomsection

\end{document}